\documentclass[12pt]{amsart}
\usepackage{amssymb}
\usepackage{bbm}
\usepackage{mathtools}
\usepackage{enumitem}
\usepackage{tikz-cd}
\usepackage{calc}
\usepackage{caption}
\usepackage{float}

\newcommand{\id}{\mathrm{id}}

\newcommand{\supp}{\mathrm{supp}}

\newcommand{\ndN}{\mathbb{N}}

\newcommand{\ndZ}{\mathbb{Z}}
\newcommand{\fK}{\Bbbk}
\newcommand{\B}{\mathcal{B}}

\newcommand{\ot}{\otimes}

\newcommand{\tri}{\triangleright}

\newcommand{\pair}{\left\langle\cdot,\cdot\right\rangle}

\theoremstyle{plain}
\newtheorem{thm}{Theorem}[section]
\newtheorem*{thm*}{Theorem}
\newtheorem*{classification*}{Classification theorem}
\newtheorem{lem}[thm]{Lemma}
\newtheorem{cor}[thm]{Corollary}

\newtheorem*{notation*}{Notation}
\newtheorem{prop}[thm]{Proposition}

\newtheorem{defi}[thm]{Definition}

\theoremstyle{remark}
\newtheorem{rema}[thm]{Remark}
\newtheorem*{rema*}{Remark}
\newtheorem{exa}[thm]{Example}

\numberwithin{equation}{section}

\newtagform{roman}[]()

\begin{document}

\title[Left Coideal subalgebras of Nichols algebras]{Left Coideal subalgebras of Nichols algebras\footnote{This work was supported by the Deutsche Forschungsgemeinschaft under grant number HE 5375/9-1.}}
\author{Istvan Heckenberger}
\address{Fachbereich Mathematik und Informatik,
Philipps-Universit\"at Marburg,
Hans-Meerwein-Str.~6, 35032 Marburg, Germany}
\email{heckenberger@mathematik.uni-marburg.de}

\author{Katharina Sch\"afer}
\address{Fachbereich Mathematik und Informatik,
Philipps-Universit\"at Marburg,
Hans-Meerwein-Str.~6, 35032 Marburg, Germany}
\email{schaef7a@mathematik.uni-marburg.de}

\begin{abstract} 
  We determine all Nichols algebras of finite-di\-men\-sion\-al Yetter-Drinfeld modules over groups such that all its left coideal subalgebras in the category of $\ndN_0$-graded comodules over the group algebra are generated in degree one as an algebra. Here we confine ourselves to Yetter-Drinfeld modules in which each group-homogeneous component is at most one-dimensional. We present a strategy to extend left coideal subalgebras by adding a suitable generator in degree two, three or four to a smaller left coideal subalgebra. We also discuss some methods for the construction of left coideal subalgebras of a Nichols algebra in the category of $\ndN_0$-graded $H$-comodules, where $H$ is a Hopf algebra, that is not necessarily a group algebra. 
\end{abstract}

\maketitle

\section*{Introduction}
In a Nichols algebra $\B(V)$ of an irreducible Yetter-Drinfeld module $V$ over a Hopf algebra $H$ with bijective antipode the only non-zero left coideal subalgebras in the category of Yetter-Drinfeld modules over $H$ are $\fK 1$ and $\B(V)$ (see \cite[Lemma~14.1.1]{MR4164719}). However, typically there are more left coideal subalgebras of $\B(V)$ in the category of $\ndN_0$-graded $H$-comodules. Left coideal subalgebras of $\B(V)$ which are $\ndN_0$-graded $H$-comodules are in bijection with $\ndN_0$-graded left coideal subalgebras of the bosonization $\B(V)\# H$ intersecting $H$ trivially (by \cite[Lemma~12.4.4]{MR4164719} and since the objects are $\ndN_0$-graded). One-sided coideal subalgebras of pointed braided Hopf algebras are studied in several recent papers focusing on different aspects (e.g. \cite{MR3552907}, \cite{MR3096611}, \cite{MR2179722}, \cite{MR2415067}, \cite{MR3413681}, \cite{heckenberger2022decomposition}).
 
 We determine all Nichols algebras of finite-dimensional Yetter-Drin\-feld modules over groups $G$ such that all its left coideal subalgebras in the category of $\ndN_0$-graded $\fK G$-comodules are generated in degree one as an algebra. Here we confine ourselves to Yetter-Drinfeld modules in which each $G$-homogeneous component is at most one-dimensional. In \cite{MR3552907}, subalgebras of Fomin-Kirillov algebras are studied which are generated in degree one by $\mathbb{S}_n$-homogeneous elements. These subalgebras are by construction left coideal subalgebras in the category of $\ndN_0$-graded $\fK\mathbb{S}_n$-comodules.  
 
 To decide whether a given Nichols algebra contains a left coideal subalgebra that is not generated in degree one we describe and study the braided vector space of the underlying Yetter-Drinfeld module in terms of racks and two-cocycles. It turns out that there are only six indecomposable racks of conjugacy classes and in each case only a few two-cocycles such that all such left coideal subalgebras of the corresponding Nichols algebra are generated in degree one (see Theorem~\ref{thm:mainthm}). In Section~\ref{se:Introductionlcsa} we discuss left coideal subalgebras of Nichols algebras over a Hopf algebra $H$, that is not necessarily a group algebra. There we give some methods for the construction of left coideal subalgebras of $\B(V)$ in the category of $\ndN_0$-graded $H$-comodules. In particular, Theorem~\ref{thm:mulbijnoextension} provides a useful condition which implies that all such left coideal subalgebras of $\B(V)$ are generated in degree one. In Section~\ref{se:different group realization} we study the description of braided vector spaces coming from Yetter-Drinfeld modules over group algebras in terms of racks and two-cocycles. We provide arguments which show that it suffices to answer the question of the existence of not primitively generated left coideal subalgebras for indecomposable racks of conjugacy classes with two-cocycles and their Yetter-Drinfeld realization over the enveloping group of the rack. In Section~\ref{se:Extensionsdeg234} we present a strategy to extend left coideal subalgebras by adding a suitable degree $2$, $3$ or $4$ generator to a smaller left coideal subalgebra. In Sections~\ref{se:Transpos}-\ref{se:cube} we discuss racks of transpositions, the rack associated to the vertices of a tetrahedron and the rack associated to the faces of a cube in detail. We determine all two-cocycles on these racks for which all left coideal subalgebras of the corresponding Nichols algebra are generated in degree one. In Section~\ref{se:rackswithsubracks} we prove that these (and the rack with one element) are the only indecomposable racks of conjugacy classes that admit two-cocycles such that all left coideal subalgebras are generated in degree one.       

\newpage
\section{Preliminaries}
We recall some definitions and facts about Yetter-Drinfeld modules and Nichols algebras. Our main reference is the book \cite{MR4164719}. 

Let $\fK$ be an algebraically closed field with $\operatorname{char}\fK=0$. Let $H$ be a Hopf algebra over $\fK$ with invertible antipode $S$. We use Sweedler notation for the comultiplication, i.e. writing $\Delta(h)=h_{(1)}\ot h_{(2)}$ for all $h\in H$. For the left coaction of a left $H$-comodule $V$ we use the notation ${^H}\delta(v)=v_{(-1)}\ot v_{(0)}$ for all $v\in V$ and we abbreviate ${^H}\delta=\delta$ if no confusion is possible in the context. A \textbf{Yetter-Drinfeld module} over $H$ is a left $H$-module $V$ that is also a left $H$-comodule with structure map $\delta:V\rightarrow H\ot V$ satisfying the compatibility condition \[\delta(h\cdot v)=h_{(1)}v_{(-1)}S(h_{(3)})\ot h_{(2)}\cdot v_{(0)}\] for all $h\in H$, $v\in V$. We denote the category of Yetter-Drinfeld modules over $H$ by ${^H_H}\mathcal{YD}$. Morphisms in ${^H_H}\mathcal{YD}$ are morphisms of $H$-modules and $H$-comodules. The category of Yetter-Drinfeld modules over $H$ is a braided monoidal category (see e.g. \cite[Theorem~3.4.13]{MR4164719}). Let $V, W$ be Yetter-Drinfeld modules over $H$. The monoidal structure is given by diagonal action and diagonal coaction on $V\ot W$. The braiding \[c_{V,W}: V\ot W\rightarrow W\ot V\] is defined by $c_{V,W}(v\ot w)=v_{(-1)}\cdot w\ot v_{(0)}$ for all $v\in V$ and $w\in W$. In particular, $V$ is a braided vector space with braiding $c_V=c_{V,V}$, i.e. $c_V$ is a linear automorphism of $V^{\ot 2}$ satisfying the so-called Yang-Baxter equation \[(c_V\ot\id)(\id\ot c_V)(c_V\ot\id)=(\id\ot c_V)(c_V\ot\id)(\id\ot c_V).\]

The \textbf{Artin braid group} $\mathbb{B}_n$ is the group generated by elements $\sigma_1,\dots,\sigma_{n-1}$ with the \textbf{braid relations}
\begin{align}
&\sigma_i\sigma_{i+1}\sigma_i = \sigma_{i+1}\sigma_{i}\sigma_{i+1} &&\mathrm{for\ all}\ 1 \leq i \leq n - 2,\label{eq:braid121}\\ 
&\sigma_{i}\sigma_j = \sigma_j\sigma_{i} && \mathrm{for\ all}\ 1 \leq i, j \leq n - 1, \left|i - j\right| > 1.\label{eq:braid12}
\end{align}   
Let $\left\{s_1,...,s_{n-1}\right\}$ be the standard generators $s_i=(i \ i+1)$ of the symmetric group $\mathbb{S}_n$. Then $\nu:\mathbb{S}_n\rightarrow\mathbb{B}_n$ defined by mapping an element $s=s_{i_1}s_{i_2}\cdot\cdot\cdot s_{i_m}\in\mathbb{S}_n$ onto  $\sigma=\sigma_{i_1}\sigma_{i_2}\cdot\cdot\cdot \sigma_{i_m}$, where $m$ is the length of the permutation $s$, is a well-defined map, the so-called Matsumoto section. 
 Let $V$ be a braided vector space with braiding $c$. For $1\leq i\leq n-1$ define $c_i\in \mathrm{Aut}(V^{\ot n})$ by applying $c$ at the $i$-th position. Then \[\tau:\mathbb{B}_n \rightarrow\mathrm{Aut}(V^{\ot n}),\quad \sigma_i\mapsto c_i\] is a group homomorphism since the automorphisms $c_i$ satisfy the braid relations. For $n\geq 2$ and $1\leq i\leq n-1$ the \textbf{braided symmetrizer} map $S_n: V^{\ot n}\rightarrow V^{\ot n}$ is defined by \[S_n=\sum\limits_{s\in\mathbb{S}_n} \tau(\nu(s)).\] The \textbf{Nichols algebra} of a braided vector space $(V,c)$ is \[\B(V)=\bigoplus_{n\in\ndN_0}\B(V)(n)=\fK \oplus V \oplus\bigoplus_{n\geq 2} T(V)(n)/\mathrm{ker}(S_n),\] where $T(V)$ is the tensor algebra of $(V,c)$.

If $H$ is a Hopf algebra with bijective antipode and if $V$ is a Yetter-Drinfeld module over $H$ with Yetter-Drinfeld braiding $c_V$, then the Nichols algebra $\B(V)$ becomes a Hopf algebra in the category ${^H_H}\mathcal{YD}$. For details see e.g. \cite[Section~7.1]{MR4164719}. But note that the Nichols algebra as an algebra and a coalgebra only depends on the braided vector space $(V,c_V)$.

For an $\ndN_0$-graded coalgebra $C$ with comultiplication $\Delta$ we define \[\Delta_{1^n}=\pi_1^{\ot n}\Delta^{n-1}:C(n)\rightarrow C(1)^{\ot n}\] for all $n\in\ndN$, where $\pi_1$ is the projection onto degree $1$. As endomorphisms of $V^{\ot n}$, where $V$ is a braided vector space, the braided symmetrizer map $S_n$ equals the map $\Delta_{1^n}$, where $\Delta$ is the comultiplication of the tensor algebra $T(V)$ (see \cite[Corollary~6.4.9]{MR4164719}). 
We define the components of the comultiplication of an $\ndN_0$-graded coalgebra $C$ as 
\begin{align}\label{eq:comulcomp}
\Delta_{i,n-i}=(\pi_i\ot\pi_{n-i})\Delta: C(n)\rightarrow C(i)\ot C(n-i),
\end{align} 
where $n\in\ndN_0$, $0\leq i\leq n$ and $\pi_i$ is the projection onto degree $i$. We will mostly use the components $\Delta_{1,n-1}$ and $\Delta_{n-1,1}$. Proposition~\ref{prop:partialleftcoideal} and Proposition~\ref{pro:minmaxlcsa} show in which way they are useful to construct left coideal subalgebras. The following formulas are very practical to compute them in the tensor algebra $T(V)$ over a braided vector space $(V,c)$. 
\begin{align}
\Delta^{T(V)}_{1,n-1}:\ &T(V)(n)=V^{\ot n}\rightarrow V\ot T(V)(n-1)=V^{\ot n} \nonumber\\
&x\mapsto (1+c_1+c_1c_2+...+c_1c_2\cdots c_{n-1})(x),\label{eq:partial1n-1}\\
\Delta^{T(V)}_{n-1,1}:\ &T(V)(n)=V^{\ot n}\rightarrow T(V)(n-1)\ot V=V^{\ot n}\nonumber\\
&x\mapsto (1+c_{n-1}+c_{n-1}c_{n-2}+...+ c_{n-1}c_{n-2}...c_1)(x).\label{eq:partialn-11}
\end{align}
Informations about how to derive them can be found in \cite[Sections~1.7-1.9]{MR4164719} and in particular in \cite[Theorem~1.9.1]{MR4164719}.

For the kernel of the counit $\varepsilon$ of a coalgebra $C$ we use the symbol $C^+=\ker(\varepsilon\vert C)$. If $C$ is a connected $\ndN_0$-graded coalgebra, then $C^+=\bigoplus_{i\geq 1}C(i)$.

We denote the multiplication in an algebra $A$ by $\mu$ and the subalgebra (with one) generated by a subset $B\subseteq A$ by $\langle B\rangle$.

\section{Introduction to left coideal subalgebras}\label{se:Introductionlcsa}

In this section, let $H$ be a Hopf algebra with bijective antipode $S$ and let $V\in {^H_H}\mathcal{YD}$.

\begin{defi}
A \textbf{left coideal} of $\B(V)$ is a subspace $C\subseteq\B(V)$ such that $\Delta(C)\subseteq\B(V)\ot C$. A \textbf{left coideal subalgebra of $\B(V)$ in the category of $\ndN_0$-graded (left) $H$-comodules} is an $\ndN_0$-graded left $H$-subcomodule of $\B(V)$ which is a subalgebra and a left coideal\footnote{Note that each left coideal of $\B(V)$ contains 1 since $\B(V)$ is connected.}. 
\end{defi}
 
\begin{defi}\label{def:lcsa}
 Let $C\subseteq\B(V)$ be a left coideal subalgebra in the category of $\ndN_0$-graded $H$-comodules. We say that $C$ is \textbf{generated in degree one} if $C=\langle C(1)\rangle$. We say that $C$ can be extended in degree $n$ if there is an element $x\in\B(V)(n)$ such that $x\notin C$ and $\langle C,x\rangle$ is a left coideal subalgebra in the category of $\ndN_0$-graded $H$-comodules. In this case we call $x$ an \textbf{extension} of $C$.
 \end{defi} 

To prove that a given subspace is a left coideal subalgebra, we will often use the following tools.

\begin{lem}\label{le:degreeonepartlc}
Let $0\neq C\subseteq\B(V)$ be a left coideal. If $C\cap V=0$, then $C=\fK 1$.
\end{lem}

\begin{proof}
A proof in a much more general setting can be found in \cite[Corollary~1.3.11]{MR4164719}.     
\end{proof}

\begin{prop}\label{prop:generatedalgebralcsa}
Let $W\subseteq \B(V)$ be an $\ndN_0$-graded $H$-subcomodule. If $W$ is a left coideal, then the subalgebra $\langle W\rangle\subseteq\B(V)$ is a left coideal subalgebra in the category of $\ndN_0$-graded $H$-comodules. 
\end{prop}

\begin{proof}
The subalgebra generated by the $\ndN_0$-graded subspace $W$ is obviously $\ndN_0$-graded by construction. Since $W$ is an $H$-subcomodule by assumption and since multiplication in $\B(V)$ is an $H$-colinear map, the subalgebra $\langle W\rangle\subseteq\B(V)$ is an $H$-subcomodule. Indeed, for all $v_1,\dots,v_n\in W$, \[\delta(v_1\cdots v_n)=v_{1(-1)}\cdots v_{n(-1)}\ot v_{1(0)}\cdots v_{n(0)}\in H\ot\langle W\rangle\] since $\delta(v_i)=v_{i(-1)}\ot v_{i(0)}\in H\ot W$ for all $i\in\{1,\dots,n\}$. Since $W$ is an $H$-subcomodule, \[c(w\ot b)=w_{(-1)}\cdot b\ot w_{(0)}\in\B(V)\ot W\] for all $w\in W$, $b\in\B(V)$ and hence $c(W\ot\B(V))\subseteq\B(V)\ot W$, where $c$ is the Yetter-Drinfeld braiding of $\B(V)$. Since $W$ is a left coideal, we know that $\Delta(W)\subseteq\B(V)\ot W$. It follows that
\begin{align*}
\Delta(WW)&=\Delta(W)\Delta(W)\subseteq(\B(V)\ot W)(\B(V)\ot W)\\
&\subseteq\B(V)c(W\ot\B(V))W\subseteq\B(V)\ot\langle W\rangle
\end{align*}
since $\Delta$ is a morphism of algebras in ${^{H}_{H}}\mathcal{YD}$. Hence $\Delta(\langle W\rangle)\subseteq\B(V)\ot\langle W\rangle$. 
\end{proof}

\begin{cor}\label{cor:degreeonegeneratedlcsa}
Let $W\subseteq V$ be an $H$-subcomodule. Then the subalgebra $\langle W\rangle\subseteq\B(V)$ is a left coideal subalgebra in the category of $\ndN_0$-graded $H$-comodules.  
\end{cor}

\begin{proof}
Since all elements of $V$ are primitive, we have \[\Delta(w)=1\ot w+w\ot 1\in\B(V)\ot (W+\fK1)\] for all $w\in W$. Hence $W+\fK 1$ is a left coideal in the category of $\ndN_0$-graded $H$-comodules and the claim follows from Proposition~\ref{prop:generatedalgebralcsa}. 
\end{proof}

\begin{lem}\label{le:comultcomponents}
Let $B$ be an $\ndN_0$-graded coalgebra. Then \[(\Delta_{1^i}\ot\id)\Delta_{i,n-i}=(\id^{\ot i-1}\ot\Delta_{1,n-i})\cdots(\id\ot\Delta_{1,n-2})\Delta_{1,n-1}\] as maps from $B(n)$ to $B(1)^{\ot i}\ot B(n-i)$ for all $n\geq 2$, $1<i\leq n$ . 
\end{lem}

\begin{proof}
First note that
\begin{equation}\label{eq:a}
\begin{aligned}
    \pi_1^{\ot i}\Delta^{i-1}\vert B(j)=
    \begin{cases}
        \Delta_{1^i}=\pi_1^{\ot i}\Delta^{i-1} &\mathrm{if}\ i=j\\
        0 &\mathrm{if}\ i\neq j,
    \end{cases}
\end{aligned}
\end{equation}
for all $i\geq 1$ and $j\geq 0$ since $\Delta$ is graded. Further,
\begin{equation}\label{eq:b}
\begin{aligned}
    (\pi_1\ot\pi_i)\Delta\vert B(j)=
    \begin{cases}
        \Delta_{1,i} &\mathrm{if}\ 1+i=j\\
        0 &\mathrm{if}\ 1+i\neq j,
    \end{cases}
\end{aligned}
\end{equation}
for all $i,j\geq 0$. Using Equations~(\ref{eq:a}) and (\ref{eq:b}) and the coassociativity of $\Delta$ it follows that both sides of the equation in the claim are equal to $(\pi_1^{\ot i}\ot\pi_{n-i})\Delta^{i}$. 
\end{proof}

\begin{prop}\label{prop:partialleftcoideal}
   Let $C\subseteq\B(V)$ be an $\ndN_0$-graded subspace. Then $C$ is a left coideal of $\B(V)$ if and only if $\Delta_{1,n-1}(C(n))\subseteq V\ot C(n-1)$ for all $n\in\ndN$. 
\end{prop}

\begin{proof}
Since $\B(V)$ is an $\ndN_0$-graded coalgebra, the subspace $C$ is a left coideal if and only if \[\Delta_{i,n-i}(C(n))\subseteq\B(V)(i)\ot C(n-i)\] for all $n\in\ndN$ and $0\leq i\leq n$. In particular, if $C$ is a left coideal, it follows that $\Delta_{1,n-1}(C(n))\subseteq V\ot C(n-1)$ for all $n\in\ndN$. For the other implication let $n\in\ndN$. We prove $\Delta_{i,n-i}(C(n))\subseteq\B(V)(i)\ot C(n-i)$ for all $0\leq i\leq n$. For $i=0$ we have $\Delta_{0,n}(x)=1\ot x$ for all $x\in C(n)$ 
since $\B(V)$ is a connected $\ndN_0$-graded coalgebra. For $i=1$ we have $\Delta_{1,n-1}(C(n))\subseteq V\ot C(n-1)$ by assumption. Let $1<i\leq n$. Since $\B(V)$ is a strictly $\ndN_0$-graded coalgebra, 
the morphism $\Delta_{1^i}$ is injective (see e.g. \cite[Proposition 1.3.14]{MR4164719}). Thus $\Delta_{i,n-i}(C(n))\subseteq\B(V)(i)\ot C(n-i)$ holds if and only if \[(\Delta_{1^i}\ot \id)\Delta_{i,n-i}(C(n))\subseteq V^{\ot i}\ot C(n-i).\] Since $\Delta_{1,n-j}(C(n-j+1))\subseteq V\ot C(n-j)$ for all $1\leq j\leq i$ by assumption and \[(\Delta_{1^i}\ot\id)\Delta_{i,n-i}
=(\id^{\ot i-1}\ot\Delta_{1,n-i})\cdots(\id\ot\Delta_{1,n-2})\Delta_{1,n-1}\] by Lemma~\ref{le:comultcomponents}, the claim follows inductively. 
\end{proof}

\begin{cor}\label{cor:lcsaExt}
Let $W\subseteq V$ be an $H$-subcomodule. Let $n\in\ndN$ and $x\in\B(V)(n)$ such that $\delta(x)\in H\ot x$ and $\Delta_{1,n-1}(x)\in V\ot\langle W\rangle(n-1)$. Then $\langle W+\fK x\rangle$ is a left coideal subalgebra in the category of $\ndN_0$-graded $H$-comodules. 
\end{cor}

\begin{proof}
    The subalgebra $\langle W\rangle$ is a coideal subalgebra in the category of $\ndN_0$-graded $H$-comodules by Corollary~\ref{cor:degreeonegeneratedlcsa}. Then $\langle W\rangle+\fK x$ is a left coideal by Proposition~\ref{prop:partialleftcoideal} and $\langle W+\fK x\rangle$ is a left coideal subalgebra in the category of $\ndN_0$-graded $H$-comodules by  Proposition~\ref{prop:generatedalgebralcsa}. 
\end{proof}

\begin{lem}\label{le:partial2}
Let $W\subsetneq V$ be an $H$-subcomodule. Let $n\in\mathbb{N}$ and $x\in\B(V)(n)$ such that $\delta(x)\in H\ot x$ and $\Delta_{n-1,1}(x)\in \B(V)\ot W$. Let $L\subseteq \B(V)$ be the smallest left coideal in the category of $\ndN_0$-graded $H$-comodules containing $x$. If $x\notin\langle W\rangle$, then $\langle L\rangle$ is a left coideal subalgebra of $\B(V)$ in the category of $\ndN_0$-graded $H$-comodules that is not generated in degree one. 
\end{lem}
 
\begin{proof}
 By Proposition~\ref{prop:generatedalgebralcsa}, the subalgebra $\langle L\rangle$ is a left coideal subalgebra of $\B(V)$ in the category of $\ndN_0$-graded $H$-comodules. Moreover, the subalgebra $\langle L\rangle$ is not generated in degree one since $(L\cap \B(V))(1)\subseteq W$ and $x\notin\langle W\rangle$.
\end{proof}

\begin{prop} \label{pro:minmaxlcsa}
  Let $W\subseteq V$ be an $H$-subcomodule, and let $\pi:V\to V/W$ be the canonical map.
  \begin{enumerate}
  \item 
      The subspace $\ker (\id \ot \pi \pi_1)\Delta \subseteq \B (V)$ is the largest $\ndN_0$-graded left coideal of $\B (V)$ with degree one part $W$. It is an $\ndN_0$-graded left coideal subalgebra of $\B (V)$ in the category of $H$-comodules.
      \item Let $C$ be an $\ndN_0$-graded left coideal of $\B (V)$ which is a subalgebra of $\B (V)$. If $C(1)=W$, then $\langle W\rangle \subseteq C\subseteq
      \ker (\id \ot \pi \pi_1)\Delta $.
  \end{enumerate}
\end{prop}

\begin{proof}
(1) Let $\overline{W}=\ker (\id \ot \pi \pi_1)\Delta \subseteq \B (V)$. Since $\B (V)$ is connected, each left coideal of $\B (V)$ contains $1$. Moreover, each left coideal with degree one part $W$ is contained in $\overline{W}$. It is clear that $\overline{W}$ is $\ndN_0$-graded, an $H$-subcomodule of $\B (V)$, and the degree one part of $\overline{W}$ is $\ker \pi \pi_1=W$.
  The coassociativity of $\Delta $ implies that $\overline{W}$ is a left coideal of $\B (V)$. Finally, $\overline{W}$ is a subalgebra of $\B (V)$: For each $x,y\in \overline{W}$ we obtain that
  \begin{align*}
   (\id \ot \pi \pi_1)&\Delta (xy)=
  (\id \ot \pi \pi_1)\big(\Delta (x)\Delta(y)\big)\\
  &=(x\ot 1)(\id \ot \pi \pi_1)\Delta(y)
  +\big((\id \ot \pi \pi_1)\Delta(x)\big)(y\ot 1).
  \end{align*}
  The last equation uses that $W$ is an $H$-subcomodule of $V$.
  
  (2) follows directly from (1).
\end{proof}

Assume that $V$ is a finite-dimensional object in $^{H}_{H}\mathcal{YD}$. Then the dual vector space $V^*=\mathrm{Hom}(V,\fK)$ is an object in $^{H}_{H}\mathcal{YD}$, where the $H$-module and the $H$-comodule structure are defined by 
\begin{align}
    (h\cdot f)(v)&=f(S(h)\cdot v),\label{eq:dualaction}\\ f_{(-1)}f_{(0)}(v)&=S^{-1}(v_{(-1)})f(v_{(0)})\label{eq:dualcoaction}
\end{align} 
for all $h\in H$, $f\in V^*$, $v\in V$ (see e.g. \cite[Lemm~4.2.2]{MR4164719}). By \cite[Corollary~7.2.8]{MR4164719}, there is a unique non-degenerate bilinear form $\pair:\B(V^*)\ot\B(V)\rightarrow\fK$ extending the evaluation map $V^*\ot V\rightarrow \fK$ satisfying
\begin{align}
\left\langle 1,1\right\rangle&=1,\\
\left\langle \B(V^*)(n),\B(V)(m)\right\rangle&=0 \quad \text{for all $n\neq m$},\label{eq:pairinggrad}
\end{align} 
and for all $f,g\in\B(V^*)$ and $v,w\in\B(V)$ 
\begin{align}
\left\langle fg,v\right\rangle&=\left\langle f,v^{(2)}\right\rangle\left\langle g,v^{(1)}\right\rangle,\label{eq:pairing1}\\ 
\left\langle f,vw\right\rangle&=\left\langle f^{(2)},v\right\rangle\left\langle f^{(1)},w\right\rangle,\label{eq:pairing2}
\end{align}
and for all $h\in H$, $v\in\B(V)$, $f\in\B(V^*)$
\begin{align}
\left\langle h\cdot f,v\right\rangle&=\left\langle f, S(h)\cdot v\right\rangle,\label{eq:YD1} \\ 
f_{(-1)}\left\langle f_{(0)},v \right\rangle&=S^{-1}(v_{(-1)})\left\langle f,v_{(0)}\right\rangle.\label{eq:YD2}
\end{align}
Here we use Sweedler notation for the coproduct in $\B(V)$ ($\B(V^*)$, resp.) in the form $\Delta(v)=v^{(1)}\ot v^{(2)}$ for all $v\in\B(V)$ ($\Delta(f)=f^{(1)}\ot f^{(2)}$ for all $f\in\B(V^*)$, resp.). Note that Equations \eqref{eq:pairing1}-\eqref{eq:YD2} are satisfied if and only if the bilinear form $\pair$ is a Hopf pairing in $^{H}_{H}\mathcal{YD}$, where the $H$-action and the $H$-coaction are defined to be diagonal on $\B(V^*)\ot\B(V)$ and trivial on $\fK$ (see e.g. \cite[Lemma~4.2.]{MR4164719}). For a subset $X\subseteq\B(V)$ we denote its orthogonal complement with respect to the bilinear form $\pair$ by \[X^{\perp}=\{f\in\B(V^*)\mid\langle f,x\rangle=0\ \mathrm{for}\ \mathrm{all}\ x\in X\}.\] 

\begin{prop}\label{prop:complcsa}
Let $V$ be a finite-dimensional object in $^{H}_{H}\mathcal{YD}$. Let $C\subseteq\B(V)$ be a left coideal subalgebra in the category of $\ndN_0$-graded $H$-comodules. Then $(C^+\B(V))^{\perp}\subseteq \B(V^*)$ is a left coideal subalgebra of $\B(V^*)$ in the category of $\ndN_0$-graded $H$-comodules.
\end{prop}

\begin{proof}
Let $M=C^+\B(V)$. We first prove that $M^{\perp}$ is a subalgebra of $\B(V^*)$. Clearly, $1\in M^{\perp}$ since $M\subseteq\bigoplus_{i\geq 1}\B(V)$. Let $n\in\ndN$, $a\in C^+(n)$ and $b\in\B(V)$. Since $\B(V)$ is a connected $\ndN_0$-graded coalgebra and $C$ is a left coideal, we have \[\Delta(a)\in a\ot 1+\B(V)\ot C^+.\] Note that $C^+$ is an $H$-subcomodule since $C$ is an $H$-subcomodule and $\varepsilon$ is $H$-colinear. Hence $c(C^+\ot\B(V))\subseteq\B(V)\ot C^+$ and it follows that
\begin{align*}
\Delta(ab)=(ab)^{(1)}\ot(ab)^{(2)}&\in ab^{(1)}\ot b^{(2)}+(\B(V)\ot C^+)(\B(V)\ot\B(V))\\
&\subset ab^{(1)}\ot b^{(2)}+\B(V)\ot M\\
&\subset M\ot\B(V)+\B(V)\ot M.
\end{align*}
Then it follows from Equation~(\ref{eq:pairing1}) that
\begin{align*}
\langle xy,ab\rangle&=\langle x,(ab)^{(2)}\rangle\langle y,(ab)^{(1)}\rangle=0
\end{align*} 
for all $x,y\in M^{\perp}$.

 Now we prove that $M^{\perp}$ is a left coideal. Let $x\in M^{\perp}$ and $m\in M$. Then $\langle x,mb\rangle=0$ for all $b\in\B(V)$ and hence $\langle x^{(2)},m\rangle\langle x^{(1)},b\rangle=0$ for all $b\in\B(V)$ by Equation~(\ref{eq:pairing2}). It follows that $\langle x^{(2)},m\rangle x^{(1)}=0$ since $\pair$ is non-degenerate. Hence $\Delta(x)\in\B(V)\ot M^{\perp}$.
 
 To see that $M^{\perp}$ is an $H$-subcomodule of $\B(V^*)$ note that $M=C^+M$ is an $H$-subcomodule since $C^+$ is an $H$-subcomodule of $\B(V)$ and multiplication in $\B(V)$ is $H$-colinear. Let $x\in M^{\perp}$ and $a\in M$. Then $\delta(a)=a_{(-1)}\ot a_{(0)}\in H\ot M$. Using Equation~(\ref{eq:YD2}) it follows that \[x_{(-1)}\langle x_{(0)},a \rangle=\mathcal{S}^{-1}(a_{(-1)})\langle x,a_{(0)}\rangle=0\] and hence $M^{\perp}$ is an $H$-subcomodule. Finally, $M^{\perp}$ is $\ndN_0$-graded by Equation~(\ref{eq:pairinggrad}) and since $M$ is $\ndN_0$-graded.
 \end{proof}

\begin{thm}\label{thm:mulbijnoextension} 
Let $V$ be a finite-dimensional object in $^{H}_{H}\mathcal{YD}$ and let $V_1,V_2\subseteq V$ be $H$-subcomodules such that $V=V_1\oplus V_2$. Let $N\geq 2$ and assume that \[\dim\bigoplus_{i=0}^N(\langle V_1\rangle\ot\langle V_2\rangle)(i)\geq\dim\bigoplus_{i=0}^N\B(V)(i).\] Assume that $\langle V_2\rangle\cong\langle V_2{}^*\rangle\subseteq\B(V^*)$ as $\ndN_0$-graded vector spaces, where $V_2{}^*$ is the dual $H$-comodule of $V_2$. 
Then $\langle V_1\rangle\subseteq\B(V)$ is a left coideal subalgebra in the category of $\ndN_0$-graded $H$-comodules that cannot be extended in a degree $2\leq n\leq N$. 
\end{thm}

\begin{proof}
The subalgebra $\langle V_1\rangle\subseteq\B(V)$ is a left coideal subalgebra in the category of $\ndN_0$-graded $H$-comodules by Corollary~\ref{cor:degreeonegeneratedlcsa}. Let $2\leq n\leq N$ and assume that there is an element $x\in\B(V)(n)$ that extends the left coideal subalgebra $\langle V_1\rangle$ in the category of $\ndN_0$-graded $H$-comodules. Let $K=\langle V_1+\fK x\rangle$. Then $\dim K(l)=\dim \langle V_1\rangle(l)$ for all $l<n$, $\dim K(n)=\dim\langle V_1\rangle(n)+1$ and $\dim K(l)\geq\dim\langle V_1\rangle(l)$ for all $l>n$. Let $M=K^+\B(V)$. Then $M$ is $\ndN_0$-graded and $M(1)=V_1$ since multiplication in $\B(V)$ is $\ndN_0$-graded and $K^+\subseteq\bigoplus_{i\geq 1}\B(V)(i)$. Hence $M^{\perp}(1)=V_2{}^*$. By Proposition~\ref{prop:complcsa} we know that $M^{\perp}\subseteq\B(V^*)$ is a (left coideal) subalgebra in the category of $\ndN_0$-graded $H$-comodules. This implies \[\dim M^{\perp}(m)\geq\dim\langle V_2{}^*\rangle(m)\] for all $m\in\ndN_0$. Moreover, there is an induced non-degenerate pairing of $\ndN_0$-graded left $H$-comodules $\pair': M^{\perp}\ot \B(V)/M\rightarrow\fK$. It follows \[\dim (\B(V)/M)(m)=\dim M^{\perp}(m)\geq\dim\langle V_2{}^*\rangle(m)\] for all $m\in\ndN_0$. From \cite[Corollary~6.3.10]{MR4164719} it is known that there is an $\ndN_0$-graded isomorphism \[K\ot \B(V)/M\rightarrow\B(V).\] It follows that
\begin{equation}\label{eq:dim}
\begin{aligned}
\dim\bigoplus_{i=0}^N\B(V)(i)&=\sum_{i=0}^N\sum_{l=0}^i\dim K(l)\dim(\B(V)/M)(i-l)\\
            &\geq\sum_{i=0}^N\sum_{l=0}^i\dim K(l)\dim\langle V_2{}^*\rangle(i-l)\\
            &=\sum_{i=0}^N\sum_{l=0}^i\dim K(l)\dim\langle V_2\rangle(i-l)\\
            &>\sum_{i=0}^N\sum_{l=0}^i\dim\langle V_1\rangle(l)\dim\langle V_2\rangle(i-l)\\
            &=\dim\bigoplus_{i=0}^N(\langle V_1\rangle\ot\langle V_2\rangle)(i)\\
            &\geq\dim\bigoplus_{i=0}^N\B(V)(i),
\end{aligned}
\end{equation}
where the third equality and the last inequality follow by assumption. Hence we get a contradiction and conclude that the left coideal subalgebra $\langle V_1\rangle$ cannot be extended in a degree $2\leq n\leq N$.
\end{proof}

We will apply Theorem~\ref{thm:mulbijnoextension} as described in the following Corollary. 

\begin{cor}\label{cor:noextension}
Let $V$ be a finite-dimensional object in $^{H}_{H}\mathcal{YD}$. Let $V_1,V_2\subseteq V$ be $H$-subcomodules such that $V=V_1\oplus V_2$.  
\begin{enumerate}
    \item Assume that $\langle V_2\rangle\cong\langle V_2{}^*\rangle\subseteq\B(V^*)$ as $\ndN_0$-graded vector spaces. Let $N\in\ndN_0$ and assume that the restricted multiplication map $\mu:(\langle V_1\rangle\ot\langle V_2\rangle)(n)\rightarrow\B(V)(n)$ is surjective for all $n\leq N$. Then the left coideal subalgebra $\langle V_1\rangle$ cannot be extended in a degree $2\leq n\leq N$. 
    \item Assume that $\B(V)$ is finite-dimensional, $\dim\langle V_2\rangle=\dim\langle V_2{}^*\rangle$ and $\dim\langle V_1\rangle\dim\langle V_2\rangle=\dim\B(V)$. Then the left coideal subalgebra $\langle V_1\rangle$ cannot be extended in a degree $2\leq n$.
\end{enumerate}
\end{cor}

\begin{proof}
(1) Follows directly from Theorem~\ref{thm:mulbijnoextension}.
(2) The proof is analogous to that of Theorem~\ref{thm:mulbijnoextension} replacing the calculation~(\ref{eq:dim}) with
\begin{align*}
\dim\B(V)&=\dim K\dim(\B(V)/M)\geq\dim K\dim\langle V_2^*\rangle\\
            &=\dim K\dim\langle V_2\rangle>\dim\langle V_1\rangle\dim\langle V_2\rangle=\dim\B(V).
\end{align*}
\end{proof}
 
\section{Yetter-Drinfeld modules over groups and racks}\label{se:different group realization}

From now on we are mainly dealing with Yetter-Drinfeld modules over the group algebra $H=\fK G$ of a group $G$. Due to the category isomorphism between $\fK G$-comodules and $G$-graded vector spaces (see e.g. \cite[Proposition~1.1.17]{MR4164719}) in this case Yetter-Drinfeld modules in ${^H_H}\mathcal{YD}$ can equivalently be defined as $G$-graded vector spaces $V=\bigoplus_{g\in G}V_g$ that are also left $\fK G$-modules such that
\[ g\cdot V_h\subseteq V_{ghg^{-1}}
\]
for all $g,h\in G$. There
\[ V_g=\{v\in V\mid\delta(v)=g\ot v\}=\{v\in V\mid \deg v=g\}
\]
for all $g\in G$. The subset \[\supp V=\{g\in G\mid V_g\neq 0\}\] is called the \textbf{support} of $V$. For objects $V,W\in {^H_H}\mathcal{YD}$ the Yetter-Drinfeld braiding is $c_{V,W}(v\ot w)=g\cdot w\ot v$ for all $v\in V_g$ and $w\in W$. 

For the explicit construction of all Yetter-Drinfeld modules over an arbitrary group $G$ there is a well-known useful equivalence of categories. We briefly describe it here and refer to \cite[Section~1.4]{MR4164719} and particulary to Proposition~1.4.17 there for more details. For an element $g\in G$ we denote its conjugacy class by $\mathcal{O}_g=\{hgh^{-1}\mid h\in G\}$. Let $L\subseteq G$ be a complete set of representatives of (pairwise distinct) conjugacy classes of $G$. Any Yetter-Drinfeld module $M\in {^{G}_{G}}\mathcal{YD}$ has a decomposition \[M=\bigoplus\limits_{l\in L}\bigoplus\limits_{s\in \mathcal{O}_l}M_s\] into a direct sum of Yetter-Drinfeld modules $\bigoplus_{s\in \mathcal{O}_l}M_s$, $l\in L$. Let $g\in G$ and let ${^{G}_{G}}\mathcal{YD}(\mathcal{O}_g)$ be the full subcategory of Yetter-Drinfeld modules containing each $M\in {^{G}_{G}}\mathcal{YD}$ with $M=\bigoplus_{s\in\mathcal{O}_g}M_s$. We denote the centralizer of $g$ by \[G^g=\{h\in G\mid hg=gh\}.\] There is an equivalence of categories between ${^{G}_{G}}\mathcal{YD}(\mathcal{O}_g)$ and the category ${_{\fK G^g}}\mathcal{M}$ of left $\fK G^g$-modules. The first involved functor $F: {_{\fK G^g}}\mathcal{M}\rightarrow {^{G}_{G}}\mathcal{YD}(\mathcal{O}_g)$ maps a $\fK G^g$-module $(\rho,V)$ to \[M(g,\rho)=\fK G\ot_{\fK G^g} V,\] which is a Yetter-Drinfeld module where the $\fK G$-action is multiplication on the first tensor factor and with $G$-grading \[\mathrm{deg}(h\ot v)=hgh^{-1}\] for all $h\in G$, $v\in V$. The decomposition of $M(g,\rho)$ is given by $M(g,\rho)=\bigoplus_{s\in\mathcal{O}_g}M(g,\rho)_s$, where $M(g,\rho)_{hgh^{-1}}=h\ot V$ for all $h\in G$. 
The quasi-inverse functor ${^{G}_{G}}\mathcal{YD}(\mathcal{O}_g)\rightarrow {_{\fK G^g}}\mathcal{M}$ is given by $M\mapsto M_g$.  

If $G$ is a group and $V$ a Yetter-Drinfeld module over $\fK G$ with Yetter-Drinfeld braiding $c_V$, then the Nichols algebra $\B(V)$ is a Hopf algebra in the category ${^{G}_{G}}\mathcal{YD}$. For details see e.g. \cite[Section~7.1]{MR4164719}. But (as mentioned above) the Nichols algebra as an algebra and a coalgebra only depends on the braided vector space $(V,c_V)$ and there may be many Yetter-Drinfeld modules over different groups that yield the same braided vector space. Braided vector spaces arising from Yetter-Drinfeld modules over groups can be described using racks and two-cocycles without referring to a special group.   

A \textbf{rack} is a pair $(X,\tri)$, for short $X$, where $X$ is a non-empty set and $\tri: X\times X\rightarrow X$ is a map denoted by $(x,y)\mapsto x\tri y$ for all $x,y\in X$ such that
\begin{enumerate}
\item the map $\varphi_i:X\rightarrow X$, $x\mapsto i\tri x$ is bijective for all $i\in X$, and
\item $x\tri (y\tri z)=(x\tri y)\tri (x\tri z)$ for all $x,y,z\in X$.
\end{enumerate}
A non-empty subset $Y\subseteq X$ of a rack is a \textbf{subrack} if $(Y,\tri)$ is a rack. A \textbf{quandle} is a rack with $x\tri x=x$ for all $x\in X$. A map $f:(X,\tri)\rightarrow (Y,\tri)$ is a $\textbf{morphism of racks}$ if $f(x\tri y)=f(x)\tri f(y)$ for all $x,y\in X$. We say that $x,y\in X$ \textbf{commute} if $x\tri y=y$.  

\begin{exa}\label{ex:conjugacyrack}
Let $G$ be a group and let $H\subseteq G$ be a subset which is stable under conjugation. Then $H$ is a quandle, where $x\tri y=xyx^{-1}$ for all $x,y\in H$. Moreover, in such a rack $H$ of conjugacy classes for all $x,y\in H$ we have $x\tri y=y$ if and only if $y\tri x=x$. In particular, $\supp V$ is a rack for any Yetter-Drinfeld module over $\fK G$. 
\end{exa}

\begin{defi}
A rack $(X,\tri)$ is called \textbf{braided} if it is a quandle and either $x\tri y=y$ or $x\tri (y\tri x)=y$ for all $x,y\in X$.
\end{defi}    

We will use the important "calculation rules" for braided racks in the following two Lemmata all the time without referring to it.

\begin{lem}\label{le:braidrack}\cite[Lemma 2,\ Lemma 3]{MR2891215}
Let $(X,\tri)$ be a braided rack and let $x,y,z\in X$. Then the following hold.
\begin{enumerate}
    \item If $y\neq z$ and $x\tri y=z$, then $y\tri z=x$ and $z\tri x=y$.
    \item If $y\tri x=x$, then $x\tri y=y$.
    \item If $x\tri (y\tri x)=y$, then $y\tri (x\tri y)=x$.
\end{enumerate}
\end{lem}

\begin{lem}\cite[Lemma~4]{MR2891215}\label{le:braidrackequi}
Let $(X,\tri)$ be a quandle. Then the following are equivalent.
\begin{enumerate}
\item $X$ is braided.
\item $x\tri (y\tri x)\in\{x,y\}$ for all $x,y\in X$.
\end{enumerate}
\end{lem}

For a rack $(X,\tri)$ the subgroup of the group of permutations $\mathbb{S}_X$ generated by the permutations $\varphi_i$, $i\in X$, is called \textbf{inner group} of $X$, and is denoted by $\mathrm{Inn}(X)$. If $\mathrm{Inn}(X)$ acts transitively on $X$, we say that $X$ is \textbf{indecomposable}. For any rack $(X,\tri)$ we define its \textbf{enveloping group} by \[G_X=\langle g_x\mid x\in X\rangle/(g_xg_y=g_{x\tri y}g_x\ \mathrm{for\ all}\ x,y\in X).\]  A rack is called \textbf{injective} if the map $X\rightarrow G_X$, $x\mapsto g_x$ is injective. 

The canonical map $i:X\rightarrow G_X$, $x\mapsto g_x$ has the following universal property: If $G$ is a group and $f:X\rightarrow G$ a morphism of racks, where $G$ is a rack by conjugation, then there is a unique morphism of groups $h:G_X\rightarrow G$ such that $h(g_x)=f(x)$ for all $x\in X$ (see e.g. \cite[Section~3.1]{MR2799090}). 

There is a canonical group action of $G_X$ on $X$ extending the rack operation $\tri$ such that $g_x\tri y=x\tri y$ for all $x,y\in X$ (see e.g. \cite[Section 1.3]{MR2891215}).
 
\begin{lem}\cite[Lemma~2.17,\ Lemma~2.18]{MR2803792}\label{le:relenvelgroup} Let $X$ be a finite indecomposable rack. Then all permutations $\varphi_x\in\mathbb{S}_X$, where $x\in X$, have the same order $n$. Moreover, in the enveloping group $G_X$ the relations $g_x{}^n=g_y{}^n$ hold and $g_x{}^n$ is a central element of $G_X$ for all $x,y\in X$.
\end{lem}

For a finite indecomposable rack $X$ we define the \textbf{finite enveloping group} of $X$ by $\overline{G_X}=G_X/(g_x{}^n)$, where $x\in X$ and $n=\mathrm{ord}(\varphi_x)$. By  Lemma~\ref{le:relenvelgroup}, the definition is independent from the choice of $x\in X$.

\begin{defi}
Let $(X,\tri)$ be a rack and let \[q: X\times X\rightarrow \fK^{\times},\quad (x,y) \mapsto q_{x,y}\] be a map. Then $q$ is called a scalar valued \textbf{two-cocycle} if \[q_{x\tri y,x\tri z}q_{x,z}=q_{x,y\tri z}q_{y,z}\] for all $x,y,z\in X$.
\end{defi} 

There is also the more general concept with (not scalar valued) two-cocycles, where the target is the automorphism group of a vector space instead of $\fK^{\times}$ (see e.g. \cite[Definition~1.5.11]{MR4164719}). However, we will consider Yetter-Drinfeld modules and their braidings whose $G$-homogeneous components are at most one-dimensional. These require only scalar valued two-cocycles. 

The following Lemma is a special case of \cite[Proposition~1.5.12]{MR4164719}.

\begin{lem}\label{le:rackbraiding}
Let $X$ be a non-empty set and let \[\tri: X\times X\rightarrow X,\ q:X\times X\rightarrow\fK^{\times}\] be maps. Let $c_q: \fK X\ot\fK X\rightarrow\fK X\ot\fK X$ be the linear map with \[c_q(x\ot y)=q_{x,y}(x\tri y)\ot x\] for all $x,y\in X$. Then $(\fK X,c_q)$ is a braided vector space if and only if $(X,\tri)$ is a rack and $q$ is a two-cocycle. 
\end{lem}

\begin{proof}
Let $x,y,z\in X$. Then 
\begin{align*}
    (c_q)_1(c_q)_2(c_q)_1(x\ot y\ot z)&=q_{x,z}q_{x\tri y,x\tri z}(x\tri y)\tri (x\tri z)\ot q_{x,y}(x\tri y)\ot x\\
    (c_q)_2(c_q)_1(c_q)_2(x\ot y\ot z)&=q_{y,z}q_{x,y\tri z}x\tri (y\tri z)\ot q_{x,y}(x\tri y)\ot x.
\end{align*}
Hence $c_q$ fulfills the Yang-Baxter equation if and only if $(X,\tri)$ is a rack and $q$ is a two-cocycle.
\end{proof}

The following definition is similarly used in \cite{MR2065444}.

\begin{defi}
    Let $X$ be a rack and $q:X\times X\rightarrow\fK^{\times}$ a two-cocycle. Let $(\fK X,c_q)$ be the associated braided vector space (as in Lemma~\ref{le:rackbraiding}). A \textbf{Yetter-Drinfeld realization} of $(\fK X,c_q)$ over a group $G$ is a Yetter-Drinfeld module structure over $\fK G$ on $\fK X$ such that the braiding $c_q$ coincides with the Yetter-Drinfeld braiding in ${^{G}_{G}}\mathcal{YD}$ of $\fK X$ and the elements of $X$ are $G$-homogeneous. We call a Yetter-Drinfeld realization \textbf{faithful} if the map $X\rightarrow G$, $x\mapsto\deg(x)$ is injective. 
\end{defi}

A rack $X$ with a two-cocycle $q$ admits many Yetter-Drinfeld realizations over different groups. Note that the Nichols algebras of all such Yetter-Drinfeld realizations are isomorphic as algebras and coalgebras since the Yetter-Drinfeld modules are isomorphic as braided vector spaces. The following Lemma~\ref{le:YDrealization1} is a variation of \cite[Proposition~1.5.6]{MR4164719}. It shows that any rack with two-cocycle can be realized as Yetter-Drinfeld module over the enveloping group of the rack. We will see in Proposition~\ref{prop:chooseenvelopingrealization} (2) that the existence of left coideal subalgebras of the Nichols algebra $\B(\fK X,q)$ in the category of $\ndN_0$-graded $\fK G$-comodules which are not generated in degree one does not depend on the choice of the group $G$ for the realization. We will mostly use the realization over $\fK G_X$.  

\begin{lem}\label{le:YDrealization1} 
 Let $X$ be a rack and $q:X\times X\rightarrow\fK^{\times}$ a two-cocycle. Let $(\fK X,c_q)$ be the associated braided vector space (as in Lemma~\ref{le:rackbraiding}). Then $\fK X$ is a Yetter-Drinfeld module over the enveloping group $G_X$ with $G_X$-action $g_x\cdot y=q_{x,y} (x\tri y)$  and $G_X$-gradation $\deg{x}=g_x$ for all $x,y\in X$. The braiding $c_q$ coincides with the Yetter-Drinfeld braiding $c_{\fK X}$. 
 \end{lem}

\begin{proof}
 The linear map \[f_x: \fK X\rightarrow \fK X,\ y\mapsto q_{x,y}(x\tri y)\] is a linear automorphism of $\fK X$ for all $x\in X$. Let $x,y,z\in X$. Then
 \begin{align*}
f_xf_y(z)&=q_{x,y\tri z}q_{y,z}(x\tri(y\tri z))\\
         &=q_{x\tri y,x\tri z}q_{x,z}((x\tri y)\tri (x\tri z))=f_{x\tri y}f_x(z),
\end{align*}
 where we used the self-distributivity of $\tri$ and the two-cocycle condition. Hence the group morphism $\langle g_x\mid x\in X\rangle\rightarrow \mathrm{Aut}(\fK X)$, $g_x\mapsto f_x$ induces a group morphism $G_X\rightarrow \mathrm{Aut}(\fK X)$. It follows that $\fK X$ is a $\fK G_X$-module with action $g_x\cdot y=q_{x,y} (x\tri y)$ for all $x,y\in X$. The Yetter-Drinfeld condition $g_x\cdot y\in (\fK X)_{g_xg_yg_x^{-1}}$ for all $x,y\in X$ holds since 
 $\deg(x\tri y)=g_{x\tri y}=g_xg_yg_x^{-1}\in G_X$. Moreover, we have \[c_{\fK X}(x\ot y)=g_x\cdot y\ot x=q_{x,y}(x\tri y)\ot x=c_q(x\ot y)\] for all $x,y\in X$. 
\end{proof}

Note that the Yetter-Drinfeld realization of a rack (with a two-cocycle) over the enveloping group is not faithful if the rack is not injective. If $X$ is an injective rack, then $g_x\neq g_y$ for each $x,y\in X$ with $x\neq y$, hence the Yetter-Drinfeld realization over $\fK G_X$ is faithful and then each $G_X$-homogeneous component of $\fK X$ is at most one-dimensional.


\begin{lem}\label{le:injectiverack}
Let $G$ be a group and $X\subseteq G$ a subset which is stable under conjugation. Then $X$ is an injective rack.
\end{lem}

\begin{proof}
Let $f: X\rightarrow G$, $x\mapsto x$ be the injective morphism of racks. By the universal property of the canonical map $i:X\rightarrow G_X$, there is a group homomorphism $g:G_X\rightarrow G$ such that $g\circ i=f$. Hence, the map $i$ is injective since $f$ is injective.
\end{proof}

\begin{lem}\label{le:YDrackcocylce}
Let $G$ be a group and $V$ a finite-dimensional Yetter-Drinfeld module over $\fK G$ such that $\dim V_g\leq 1$ for all $g\in G$. Let $X=\supp V$. For each $x\in X$ choose a non-zero element $v_x\in V_x$ and let $B=\{v_x\mid x\in X\}$. 
\begin{enumerate}
    \item For each $x\in X$ there is a map $q'_x:X\rightarrow \fK^{\times}$ such that $x\cdot v_y=q'_x(y)v_{xyx^{-1}}$ for all $y\in X$.
    \item $X$ is a rack with conjugation as rack operation and $q:X\times X\rightarrow \fK^{\times}$, $(x,y)\mapsto q'_x(y)$ is a two-cocycle.
    \item The linear map $f: (V,c_V)\rightarrow (\fK X,c_q)$, $v_x\mapsto x$ is an isomorphism of braided vector spaces. There $c_V$ is the Yetter-Drinfeld braiding of $V$ and $(\fK X,c_q)$ is the braided vector space associated to the rack $X$ and two-cocycle $q$ as in Lemma~\ref{le:rackbraiding}.
\end{enumerate}
\end{lem}

\begin{proof}
 (1) Let $x\in X$. By the Yetter-Drinfeld condition we have $x\cdot v_y\in V_{xyx^{-1}}$ for all $y\in X$. Since $\dim V_{xyx^{-1}}=1$ for all $y\in\supp V$ (by assumption and Example~\ref{ex:conjugacyrack}), it follows that for each $y\in X$ there is an element $q_{x,y}\in\fK^{\times}$ such that $x\cdot v_y=q_{x,y}v_{xyx^{-1}}$. Hence the claim follows with $q'_x:X\rightarrow \fK^{\times}$, $y\mapsto q_{x,y}$.

(2) The set $X=\supp V$ is a rack by Example~\ref{ex:conjugacyrack}. Note that the map $X\rightarrow B$, $x\mapsto v_x$ is an isomorphism of racks, where $B$ is a rack with the induced rack operation $v_x\tri v_y=v_{xyx^{-1}}$ for all $x,y\in X$. The Yetter-Drinfeld braiding $c_V$ of $V$ is given by \[c_V(v_x\ot v_y)=x\cdot v_y\ot v_x=q_{x,y}v_{xyx^{-1}}\ot v_x\] for all $x,y\in\supp V$. Hence the map $q$ is a two-cocycle on $B\cong X$ by Lemma~\ref{le:rackbraiding}.

(3) Obviously, the linear map $f$ is a well-defined isomorphism. Moreover,  
\begin{align*}
   (f\ot f)c_V(v_x\ot v_y)&=f(x\cdot v_y)\ot f(v_x)=q'_{x}(y) f(v_{xyx^{-1}})\ot f(v_x)\\
   &=q_{x,y}(x\tri y)\ot x=c_q(x\ot y)  
\end{align*}
for all $x,y\in X$. 
\end{proof}

\begin{prop}\label{prop:chooseenvelopingrealization}
Let $G$ be a group and $V$ a finite-dimensional Yetter-Drinfeld module over $\fK G$ such that $\dim V_g\leq 1$ for all $g\in G$. Let $X=\supp V$. Let $V_{G_X}$ be the Yetter-Drinfeld realization of the underlying braided vector space over $\fK G_X$.  
\begin{enumerate}
    \item The Nichols algebras $\B(V)$ and $\B(V_{G_X})$ are isomorphic as $\ndN_0$-graded braided Hopf algebras.\footnote{Since the isomorphism is induced by the identity map $\id: V=(V,{^{\fK G}}\delta_V)\rightarrow V_{G_X}=(V,{^{\fK G_X}}\delta_V)$, we will identify $\B(V)$ and $\B(V_{G_X})$ as vector spaces without referring to the isomorphism.} In particular, a subspace $C\subseteq\B(V)$ is a left coideal subalgebra if and only if $C\subseteq\B(V_{G_X})$ is a left coideal subalgebra.\footnote{Here we do not require that $C$ is a $\fK G$-comodule ($\fK G_X$-comodule, resp.).}   
    \item The Nichols algebra $\B(V)$ contains a left coideal subalgebra in the category of $\ndN_0$-graded $\fK G$-comodules that is not generated in degree one if and only if $\B(V_{G_X})$ contains a left coideal subalgebra in the category of $\ndN_0$-graded $\fK G_X$-comodules that is not generated in degree one.
\end{enumerate}
\end{prop}

\begin{proof}
(1) The identity map $\id: V=(V,{^{\fK G}}\delta_V)\rightarrow V_{G_X}=(V,{^{\fK G_X}}\delta_V)$ is an isomorphism of braided vector spaces by construction (see Lem\-ma~\ref{le:YDrackcocylce}~(3) and Lemma~\ref{le:YDrealization1}). The claim follows directly from the construction of the Nichols algebra of a braided vector space. (More explicit arguments can be found in \cite[Remark~7.1.4]{MR4164719}.)    

(2) By Lemma~\ref{le:injectiverack}, the rack $X$ is injective. This implies that the Yetter-Drinfeld realization $V_{G_X}$ over the enveloping group is faithful. Hence $\fK G$-subcomodules of $V$ are (as vector subspaces of $V$) the same as the $\fK G_X$-subcomodules of $V_{G_X}$ and the claim follows directly from Proposition~\ref{pro:minmaxlcsa} and (1).

\end{proof}

\begin{prop}\label{prop:reducibleext}
Let $G$ be a group, let $V_1,V_2\in {^{G}_{G}}\mathcal{YD}$ be finite-dimensional and $V=V_1\oplus V_2$. Let $c=c_V$ be the Yetter-Drinfeld braiding and assume that $c{}^2\vert V_2\ot V_1\neq\id_{V_2\ot V_1}$. Then there is a left coideal subalgebra of $\B(V)$ in the category of $\ndN_0$-graded $\fK G$-comodules that is not generated in degree one.   
\end{prop}

\begin{proof}
 Since $c^2\vert V_2\ot V_1\neq\id_{V_2\ot V_1}$ and $V_1,V_2$ are $\fK G$-subcomodules, there are $G$-homgeneous elements $x\in V_2$ and $y\in V_1$ such that $c^2(x\ot y)\neq x\ot y$. Define $z=\mu(x\ot y-c(x\ot y))\in\B(V)$. Then $z$ is $G$-homogeneous and \[\Delta_{1,1}(z)=(1+c)(x\ot y)-(1+c)(c(x\ot y))=x\ot y-c^2(x\ot y)\in V_2\ot V_1\] by Equation~(\ref{eq:partial1n-1}). It follows that $\Delta_{1,1}(z)\neq 0$ since $c^2(x\ot y)\neq x\ot y$ and $z\notin\langle V_2\rangle$ since $\Delta_{1,1}(\langle V_2\rangle(2))\subseteq V\ot V_2$ by Corollary~\ref{cor:degreeonegeneratedlcsa} and Proposition~\ref{prop:partialleftcoideal}. Hence the subalgebra $\langle V_2\oplus\fK z\rangle\subseteq\B(V)$ is a left coideal in the category of $\ndN_0$-graded $\fK G$-comodules that is not generated in degree one (see Corollary~\ref{cor:lcsaExt}).   
\end{proof}

Let $G$ be a group, $V_1,V_2\in {^{G}_{G}}\mathcal{YD}$ and $V=V_1\oplus V_2$. Let $c=c_V$ be the Yetter-Drinfeld braiding. For $i\in\{1,2\}$ the inclusion $V_i\subseteq V$ induces an injective morphism $\B(V_i)\rightarrow \B(V)$ of $\ndN_0$-graded Hopf algebras in ${^{G}_{G}}\mathcal{YD}$ (see \cite[Remark~1.6.19]{MR4164719}). We identify $\B(V_i)$ with the image of this injective map. Let $\mu_{12}:\B(V_1)\ot\B(V_2)\rightarrow\B(V)$ be the restricted multiplication map. If $c^2\vert V_2\ot V_1=\id_{V_2\ot V_1}$, then $\B(V_1)\ot\B(V_2)$ is a Hopf algebra in ${^{G}_{G}}\mathcal{YD}$ and $\mu_{12}$ is an isomorphism of braided Hopf algebras in ${^{G}_{G}}\mathcal{YD}$ by \cite[Proposition~1.10.12]{MR4164719}.

\begin{prop}\label{prop:lcsadecompextension}
Let $G$ be a group, let $V_1,V_2\in {^{G}_{G}}\mathcal{YD}$ be finite-dimensional and $V=V_1\oplus V_2$. Assume that $c^2\vert V_2\ot V_1=\id_{V_2\ot V_1}$ and that $\dim V_g\le 1$ for all $g\in G$. 
Let $W_1\subseteq V_1$ and $W_2\subseteq V_2$ be $\fK G$-subcomodules and let $W=W_1\oplus W_2$. Then the following are equivalent.
\begin{enumerate}
    \item $\langle W\rangle $ is the only $\ndN_0$-graded left coideal subalgebra $K$ of $\B (V)$ with $K(1)=W$.
    \item $\langle W_1\rangle $ is the only $\ndN_0$-graded left coideal subalgebra $K_1$ of $\B (V_1)$ with $K_1(1)=W_1$, and $\langle W_2\rangle $ is the only $\ndN_0$-graded left coideal subalgebra $K_2$ of $\B (V_2)$ with $K_2(1)=W_2$.
\end{enumerate}
\end{prop}

Note that by Proposition~\ref{pro:minmaxlcsa}, 
the property in Proposition~\ref{prop:lcsadecompextension}~(1)
holds if and only if
$\langle W\rangle $ is the only $\ndN_0$-graded left coideal subalgebra $K$ of $\B (V)$ in the category of $\fK G$-subcomodules with $K(1)=W$.

\begin{proof}

Let $\pi :V\to V/W$ be the canonical map.
By Proposition~\ref{pro:minmaxlcsa}, (1) holds if and only if
\[ \ker (\id \ot \pi \pi_1)\Delta_{\B (V)} =\langle W\rangle. \]
From $c^2|V_2\ot V_1=\id_{V_2\ot V_1}$ it follows that the relation $vu=\mu c(v\ot u)$ holds in $\B(V)$ for all $u\in V_1$, $v\in V_2$ and that $ghg^{-1}=h$ for all $h\in\supp V_1$, $g\in\supp V_2$. From this and since $W=W_1\oplus W_2$ and $\dim V_g\le 1$ for all $g\in G$, we obtain  
\[ \langle W\rangle =\langle W_1\rangle \langle W_2\rangle
\subseteq \B (V_1)\B (V_2).
\]
The multiplication map $\mu_{12}:\B (V_1)\ot \B (V_2)\to \B (V)$ is bijective by \cite[Proposition~1.10.12]{MR4164719}. This implies that
\[ \langle W_1\rangle =\langle W\rangle \cap \B (V_1), \qquad
\langle W_2\rangle =\langle W\rangle \cap \B (V_2).
\]
Let $\pi_{W_1}:V_1\to V_1/W_1$ be the canonical map. Then (1) implies that
\begin{align*} \ker (\id \ot \pi _{W_1}\pi_1)\Delta_{\B (V_1)}
&=\ker (\id \ot \pi\pi_1)\Delta_{\B (V)}|\B (V_1)\\
&=\B (V_1)\cap \langle W\rangle =\langle W_1\rangle ,
\end{align*}
and similarly
$\ker (\id \ot \pi _{W_2}\pi_2)\Delta_{\B (V_2)}=\langle W_2\rangle $. Thus (1) implies (2).

Assume now (2). Note that $V=V_1\oplus V_2$ and $W$ are $\ndZ^2$-graded, where $\deg (V_1)=(1,0)$ and $\deg (V_2)=(0,1)$. Thus $\B (V)$ is $\ndZ^2$-graded and $\ker (\id \ot \pi \pi_1)\Delta_{\B (V)}$ is $\ndZ^2$-graded.
For all $k,l\ge 0$ and $x\in \B (V_1)(k)$, $y\in \B (V_2)(l)$, one has
\begin{align*}
    (\id \ot \pi \pi_1)\Delta(xy)
    &=(\id \ot \pi \pi_1)\Delta(x)(y\ot 1)
    +(x\ot 1)(\id \ot \pi \pi_1)\Delta(y),
\end{align*}
where $\Delta=\Delta_{\B(V)}$.
Since $\B (V)$ and $\ker (\id \ot \pi \pi_1)\Delta $ are $\ndZ^2$-graded and
since the map $\mu_{12}:\B (V_1)\ot \B (V_2)\to \B (V)$ is bijective, it follows that
\begin{align*} &\ker (\id \ot \pi \pi_1)\Delta _{\B(V)}\\
&\qquad =
\bigoplus_{k,l\ge 0}
\ker (\id \ot \pi \pi_1)\Delta_{\B(V)}|\big( \B (V_1)(k) \,\B (V_2)(l)\big)\\
&\qquad =\ker (\id \ot \pi _{W_1}\pi_1)\Delta_{\B (V_1)} \cdot \ker (\id \ot \pi _{W_2}\pi_1)\Delta_{\B (V_2)}\\
&\qquad =\langle W_1\rangle \langle W_2\rangle ,
\end{align*}
where the last equation holds by (2).
Since $\langle W\rangle =\langle W_1\rangle \langle W_2\rangle $, we conclude that (2) implies (1).
\end{proof}

 \begin{lem}\label{le:irredindecomp}  
 Let $G$ be a group and $V$ a finite-dimensional Yetter-Drinfeld module over $\fK G$ such that $\dim V_g\leq 1$ for all $g\in G$. Let $X=\supp V$ and assume that $X$ generates $G$. Then $V$ is an irreducible Yetter-Drinfeld module if and only if $X$ is indecomposable as rack. 
\end{lem}

\begin{proof} 
Under the assumptions of the lemma, Yetter-Drinfeld submodules of $V$ correspond to subracks of $\supp V$. This implies the claim. 
\end{proof}

\begin{lem}\label{le:groupactioniso}
 Let $G$ be a group and $V$ a finite-dimensional Yetter-Drinfeld module over $\fK G$. For all $g\in G$ define the map \[\alpha_g:\B(V)\rightarrow \B(V),\quad x\mapsto g\cdot x.\] 
 \begin{enumerate}
     \item The map $\alpha_g$ is a well-defined isomorphism of $\ndN_0$-graded algebras and coalgebras.
     \item The map $\alpha_g$ sends $G$-homogeneous elements to $G$-homogeneous elements.
     \item If $K\subseteq\B(V)$ is a left coideal subalgebra of $\B(V)$ in the category of $\ndN_0$-graded $\fK G$-subcomodules, then $\alpha_g(K)$ is a left coideal subalgebra of $\B(V)$ in the category of $\ndN_0$-graded $\fK G$-subcomodules and $\alpha_g(K)\cong K$ as $\ndN_0$-graded algebras. 
 \end{enumerate}
\end{lem}

\begin{proof}
   (1) Since $\B(V)$ is a $\fK G$-module, the map is well-defined. Since $\B(V)$ is an algebra and a coalgebra in the category of $\ndN_0$-graded $\fK G$-Yetter-Drinfeld modules, $\alpha_g$ is a morphism of $\ndN_0$-graded algebras and coalgebras. Since $g$ is invertible in $\fK G$, the map is bijective (with inverse $\alpha_{g^{-1}}$). (2) holds obviously and (3) follows from (1) and (2). 
\end{proof}

Let $G$ be a group and $V$ a finite-dimensional Yetter-Drinfeld module over $\fK G$ such that $\dim V_g\leq 1$ for all $g\in G$. It is our aim to be able to decide whether there is a left coideal subalgebra of the Nichols algebra $\B(V)$ in the category of $\ndN_0$-graded $\fK G$-comodules which is not generated in degree one. In this section we provided arguments which reduce the problem in the following sense: It suffices to answer the question for Yetter-Drinfeld realizations of finite indecomposable racks of conjugacy classes $X$ with two-cocycles over the enveloping group $G_X$. Our result can be found in Theorem~\ref{thm:mainthm}. The proof of Corollary~\ref{cor:braiddecomposition} shows in which way the answer in the more general case then can be deduced using arguments from the current Section. 

\section{Extensions for left coideal subalgebras in degrees 2,3 and 4}\label{se:Extensionsdeg234}

In this section we give sufficient criteria for the existence of non-trivial extensions of left coideal subalgebras in degrees two, three and four. Our main results are Proposition~\ref{prop:braid}, Theorem~\ref{thm:cocycle}, Theorem~\ref{thm:degree3} and Theorem~\ref{thm:degree4}, all of which describe explicit extensions. The proofs of Theorem~\ref{thm:degree3} and Theorem~\ref{thm:degree4} are particularly technical, but note that the assumptions in them are also very weak. 

Let $(X,\tri)$ be a union of conjugacy classes of a group, considered as a rack. In particular,
\begin{align}\label{eq:cclassrack}
x\tri y=y\quad \Rightarrow\quad  y\tri x=x
\end{align}
for all $x,y\in X$. Let $q:X\times X\rightarrow\fK^{\times}$ be a two-cocycle. Let $V=\fK X$ be the corresponding Yetter-Drinfeld realization over $\fK G_X$ as in Lemma~\ref{le:YDrealization1} and denote the braiding by $c_q$. By construction, there are non-zero elements $v_x\in V_x$ such that \[c_q(v_x\ot v_y)=q_{x,y}v_{x\tri y}\ot v_x\] for all $x,y\in X$.  

\begin{lem}\label{le:notbraidedrackproperty}
    Assume that $X$ is not braided. Then there are pairwise distinct elements $a,b,c,d\in X$ such that $b\tri a=c$ and $c\tri b=d$. 
\end{lem}

\begin{proof}
Since $X$ is not braided, there are distinct elements $a,b\in X$ such that $b\tri a\neq a$ and $b\tri (a\tri b)\neq a$. Let $c=b\tri a$. Then $c\neq a$. Since $X$ is a quandle, $\varphi_b$ is bijective and $a\neq b$, it follows that $c\neq b$. Let $d=c\tri b$. Then $d=b\tri (a\tri b)$ by self-distributivity of $\tri$ and the quandle property. Hence $d\neq a$ by the choice of $a$ and $b$. Assume that $d=b$. Then $a\tri b=b$ since $X$ is a quandle and $\varphi_b$ is bijective. Since $X$ is a rack of a union of conjugacy classes, it follows that $b\tri a=a$ by Equation~(\ref{eq:cclassrack}), a contradiction to the choice of $a$ and $b$. Hence $d\neq b$. Assume that $d=c$. Then $(b\tri a)\tri b=b\tri a$ and hence $b=b\tri a$ by the bijectivity of $\varphi_{b\tri a}$ and the quandle property. It follows that $b=a$ by the bijectivity of $\varphi_b$ and the quandle property, a contradiction to the choice of $a$ and $b$. Hence $d\neq c$.  
\end{proof}

\begin{prop}\label{prop:braid}
Assume that $X$ is not braided. Then there is a left coideal subalgebra of $\B(V)$ in the category of $\ndN_0$-graded $\fK G_X$-comodules that can be extended in degree two.   
\end{prop}

\begin{proof}
By Lemma~\ref{le:notbraidedrackproperty}, there are pairwise distinct elements $a,b,c,d\in X$ such that $b\tri a=c$ and $c\tri b=d$. Let
\[ x=v_bv_a-q_{b,a}v_cv_b \in\B(V). \]
Then $\Delta_{1,1}(x)=v_b\ot v_a-q_{c,b}q_{b,a}v_d\ot v_c$. Hence $x\ne 0$ and $\langle v_a,v_c,x\rangle$ is a left coideal subalgebra of $\B (V)$. Note that $x$ is $G_X$-homogeneous of degree $g_bg_a$ since the relation $g_bg_a=g_{b\tri a}g_b=g_cg_b$ holds in $G_X$. Thus $\langle v_a,v_c,x\rangle$ is a $\fK G_X$-comodule. We show that $x\notin \langle v_a,v_c\rangle$. The degree two component of the subalgebra generated by $v_a$ and $v_c$ is \[\langle v_a,v_c\rangle(2)=\fK\left\{v_a{}^2,v_av_c,v_cv_a,v_c{}^2\right\}\] and 
\[v_a{}^2\in\B(V)_{g_a{}^2},\ v_av_c\in\B(V)_{g_ag_c},\ v_cv_a\in\B(V)_{g_cg_a},\ v_c{}^2\in\B(V)_{g_c{}^2}.\] Since $g_bg_a=g_cg_b=g_dg_c$, it follows that $g_bg_a\notin\{g_a{}^2,g_ag_c,g_cg_a,g_c{}^2\}$ and hence $x\notin\langle v_a,v_c\rangle$. Thus $\langle v_a,v_c,x\rangle$ is a left coideal subalgebra in the category of $\ndN_0$-graded $\fK G_X$-comodules of $\B(V)$ that is not generated in degree one. 
\end{proof}

Proposition~\ref{prop:braid} ensures the existence of a not primitively generated left coideal subalgebra in the category of $\ndN_0$-graded $\fK G_X$-comodules in the Nichols algebra $\B(X,q)$ if $X$ is not braided independent from the choice of the two-cocycle $q$. 

\begin{thm}\label{thm:cocycle}
Assume that the rack $X$ is braided. 
\begin{enumerate}
\item 
Let $a, b\in X$ with $a\tri b=b\neq a$ and let \[x=\mu(\id-c_q)(v_a\ot v_b)\in\B(V).\] If $x\neq 0$, then $\langle v_b,x\rangle$ is a left coideal subalgebra of $\B(V)$ in the category of $\ndN_0$-graded $\fK G_X$-comodules that is not generated in degree one.  
\item
Let $a,b\in X$ with $a\tri b\neq b$ and let \[x=\mu(\id-c_q+c_q{^2})(v_a\ot v_b)\in\B(V).\] If $x\neq 0$, then $\langle v_b,x\rangle$ is a left coideal subalgebra of $\B(V)$ in the category of $\ndN_0$-graded $\fK G_X$-comodules that is not generated in degree one.
\end{enumerate}
\end{thm}

\begin{proof}
(1) Note first that $b\tri a=a$ because of Equation~\eqref{eq:cclassrack} and since $a\tri b=b$. Using Equation~(\ref{eq:partial1n-1}) it follows that
\begin{align*}
\Delta_{1,1}(x)&=(\id+c_q)(\id-c_q)(v_a\ot v_b)\\
               &=(\id-c_q{}^2)(v_a\ot v_b)\\
			   &=(1-q_{a,b}q_{b,a})v_a\ot v_b.
\end{align*}
Note that $x$ is $G_X$-homogeneous of degree $g_ag_b$ and hence $\fK x$ is a $\fK G_X$-subcomodule. Thus $\langle v_b,x\rangle$ is a left coideal subalgebra of $\B(V)$ in the category of $\ndN_0$-graded $\fK G_X$-comodules by Corollary~\ref{cor:lcsaExt}. By assumption, $x\neq 0$ and obviously $x\notin\langle v_b\rangle$. Hence $\langle v_b,x\rangle$ is not generated in degree one.

(2) Since $X$ is braided and $a\tri b\neq b$, it follows that $(a\tri b)\tri a=b$ and $b\tri (a\tri b)=a$ by Lemma~\ref{le:braidrack}. Using this we get
\begin{align*}
c_q{}^3(v_a\ot v_b)&=c_q{}^2(q_{a,b}v_{a\tri b}\ot v_a)=c_q(q_{a,b}q_{a\tri b,a}v_b\ot v_{a\tri b})\\
&=q_{a,b}q_{a\tri b,a}q_{b,a\tri b}v_{b\tri(a\tri b)}\ot v_{b}=q_{a,b}q_{a\tri b,a}q_{b,a\tri b}v_{a}\ot v_{b}.
\end{align*}
By Equation~(\ref{eq:partial1n-1}), it follows  
\begin{align*}
\Delta_{1,1}(x)&=(\id+c_q)(\id-c_q+c_q{}^2)(v_a\ot v_b)\\
               &=(\id+c_q{}^3)(v_a\ot v_b)\\
			   &=(1+q_{a,b}q_{a\tri b,a}q_{b,a\tri b})v_a\ot v_b.
\end{align*}
 Note that $x$ is $G_X$-homogeneous of degree $g_ag_b$ and hence $\fK x$ is a $\fK G_X$-subcomodule. Thus $\langle v_b,x\rangle$ is a left coideal subalgebra of $\B(V)$ in the category of $\ndN_0$-graded $\fK G_X$-comodules by Corollary~\ref{cor:lcsaExt}. By assumption, $x\neq 0$ and obviously $x\notin\langle v_b\rangle$. Thus $\langle v_b,x\rangle$ is not generated in degree one.
\end{proof}

Theorem~\ref{thm:cocycle} implies that if all left coideal subalgebras of $\B(V)$ in the category of $\ndN_0$-graded $\fK G_X$-comodules are generated in degree one, then in $\B(V)$ the following relations hold in degree two:
\begin{align}\label{eq:degreetworel1}
\mu(\id-c_q)(v_a\ot v_b)&=0 \quad \mathrm{for\ all}\ a,b\in X\ \mathrm{with}\ a\tri b=b\neq a,
\end{align}
\begin{align}
\mu(\id-c_q+c_q{}^2)(v_a\ot v_b)&=0 \quad \mathrm{for\ all}\ a,b\in X\ \mathrm{with}\ a\tri b\neq b. \label{eq:degreetworel2}   
\end{align}
We say that the Nichols algebra $\B(V)$ is \textbf{slim in degree two} if the relations (\ref{eq:degreetworel1}) and (\ref{eq:degreetworel2}) hold.

\begin{lem}\label{le:degreetworel}
Assume that the rack $X$ is braided. 
\begin{enumerate}
\item Let $a,b\in X$ with $a\tri b=b\neq a$. The relation \[\mu(\id-c_q)(v_a\ot v_b)=0\] holds in $\B(V)$ if and only if  $q_{a,b}q_{b,a}=1$ if and only if $c_q{}^2(v_a\ot v_b)=v_a\ot v_b$.   
\item Let $a,b,c\in X$ with $a\tri b=c\neq b$. The relation \[\mu(\id-c_q+c_q{}^2)(v_a\ot v_b)=0\] holds in $\B(V)$ if and only if $q_{a,b}q_{c,a}q_{b,c}=-1$ if and only if $c_q{}^3(v_a\ot v_b)=-v_a\ot v_b$.  
\end{enumerate}
\end{lem}

\begin{proof}
The claims follow from the calculations in the proof of Theorem~\ref{thm:cocycle}.
\end{proof}

For the following considerations in degree three we need some technical calculations, which we collect in the next Lemma.  

\begin{lem}\label{le:rackrelations}
Assume that the rack $X$ is braided and that there are pairwise distinct elements $a,b,c\in X$ with $a\tri b\neq b$. Let $d=c\tri (b\tri a)$ and $e=d\tri (c\tri b)$. Then the following hold.
\begin{enumerate}
    \item $e=c\tri a$. 
    \item Assume that $c$ commutes with none of $a$, $b$, $b\tri a$. Let $S=\{a,b,d,e\}$. Then 
    \begin{enumerate}
        \item $c\notin S$,
        \item $b\tri a\notin S$,
        \item $b\tri y\neq c$ for all $y\in S$ and
        \item $b\tri (a\tri y)\neq c$ for all $y\in S$. 
    \end{enumerate}
\end{enumerate}
\end{lem}

\begin{proof}
(1) Since $a\tri b\neq b$ and since $X$ is braided, we have \[e=d\tri (c\tri b)=(c\tri (b\tri a))\tri (c\tri b)=c\tri ((b\tri a)\tri b))=c\tri a.\]

(2)(a) By assumption, $c\notin\{a,b\}$ and hence $c\neq e$ by (1). Moreover, $c\neq d$ since $c\neq b\tri a$.

(2)(b)  By assumption, $b\tri a\notin\{a,b\}$. Moreover, $b\tri a\neq d$ since $c$ does not commute with $b\tri a$. Finally, $b\tri a\neq e$ since otherwise $b\tri a=c\tri a$ by (1). Since $a$ commutes with none of $b$, $c$ and since $X$ is braided, this would imply $b=c$.

(2)(c) Obviously, $b\tri a\neq c$ and $b\tri b=b\neq c$. Assume that $b\tri d=c$ or $b\tri e=c$. By (1) and since
\[ b\tri d=b\tri (c\tri (b\tri a))=c\tri (b\tri (c\tri a)),
\]
it follows that $b\tri (c\tri a)=c$. Hence $(c\tri a) \tri (b\tri (c\tri a))=a$. Since $X$ is braided, it follows that $b=a$ or $c\tri a=a$ by Lemma~\ref{le:braidrackequi}. Since $b\neq a$ and $c\tri a\neq a$ by assumption, it follows that $b\tri d\neq c$ and $b\tri e\neq c$. 

(2)(d)  We have $b\tri (a\tri a)=b\tri a\neq c$ and $b\tri (a\tri b)=a \neq c$. If $c=b\tri (a\tri d)=b\tri ((a\tri c)\tri b)$, it would follow $c=b$ or $c=a\tri c$ by Lemma~\ref{le:braidrackequi}. If $c=b\tri (a\tri e)$, it would follow $c=b\tri (a\tri (c\tri a))=b\tri c$. 
\end{proof}

Recall that we have defined linear maps 
\begin{align*}
    c_i:\ &V^{\ot n}\rightarrow V^{\ot n}\\
         &v_1\ot\dots\ot v_n\mapsto v_1\ot\dots\ot v_{i-1}\ot c_q(v_i\ot v_{i+1})\ot v_{i+2}\ot\dots\ot v_{n}     
\end{align*}
for all $n\in\ndN$ and $1\leq i\leq n-1$. Note that the maps $c_i$ depend on the two-cocycle although the symbols do not indicate it. Moreover, the maps are $G_X$-graded (that is, $\fK G_X$-comodule maps) since $c_q$ is. The maps $c_i$ fulfill the braid relations 
\begin{align}
&c_ic_{i+1}c_i= c_{i+1}c_ic_{i+1} &&\mathrm{for\ all}\ 1 \leq i \leq n - 2,\label{eq:cbraid121}\\ 
&c_{i}c_j= c_jc_{i} && \mathrm{for\ all}\ 1 \leq i, j \leq n - 1, \left|i - j\right| > 1,\label{eq:cbraid12}
\end{align}  
where the first relation holds since $c_q$ fulfills the Yang-Baxter equation and the second by definition of $c_i$.

\begin{thm} \label{thm:degree3}
Assume that the rack $X$ is braided, that there are pairwise distinct elements $a,b,c\in X$ with $a\tri b\neq b$ and that $\B(V)$ is slim in degree two. Let $d=c\tri (b\tri a)$, $e=d\tri (c\tri b)$ and $S=\{v_a,v_b,v_d,v_e\}$. Define   
\[Z(a,b,c)=\mu(\id-c_1c_2+(c_1c_2)^2)(v_c\ot v_b\ot v_a) \in\B(V),\] 
and let $C$ be the subalgebra of $\B(V)$ generated by $S\cup\{Z\}$. Then the following hold.
\begin{enumerate}
\item The subalgebra $C$ is a left coideal of $\B(V)$ in the category of $\ndN_0$-graded $\fK G_X$-comodules. 
\item Assume that $c$ commutes with none of $a$, $b$, $b\tri a$. If $a\tri b\neq c$ or $a\tri c\neq b\tri a$, then $Z\notin\langle S\rangle$ and hence $C$ is not generated in degree one.
\end{enumerate}
\end{thm}

\begin{proof}
(1) Since $Z$ is $G_X$-homogeneous of degree $g_cg_bg_a$, the subalgebra $C$ is a $\fK G_X$-subcomodule. Thus
 by Corollary~\ref{cor:lcsaExt}
 it suffices to show that $\Delta_{1,2}(Z)\in V\ot\langle S\rangle$. Using Equation~(\ref{eq:partial1n-1}) we get
\begin{align*}
\Delta_{1,2}(Z)=&(\id\ot\mu)(\id+c_1+c_1c_2)(\id-c_1c_2+(c_1c_2)^2)(v_c\ot v_b\ot v_a) \\
               =&(\id\ot\mu)(\id-c_1^2c_2+(c_1c_2)^3+c_1+c_1(c_1c_2)^2)(v_c\ot v_b\ot v_a).
\end{align*}
By the braid equation~(\ref{eq:cbraid121}), \[c_1(c_1c_2)^2=(c_1c_2)^2c_2=c_2c_1c_2^3.\] Moreover, one has $c_2^3(v_c\ot v_b\ot v_a)=-v_c\ot v_b\ot v_a$ since $b\tri a\neq a$, and $\B(V)$ is slim in degree two (see Lemma~\ref{le:degreeonepartlc} (2)). Thus
\[ (c_1+c_1(c_1c_2)^2)(v_c\ot v_b\ot v_a)=(\id-c_2)c_1(v_c\ot v_b\ot v_a)
\]
and it follows that
\begin{align*}
\Delta_{1,2}(Z)=&(\id\ot\mu)(\id-c_1^2c_2+(c_1c_2)^3)(v_c\ot v_b\ot v_a)\\
			    &+(\id\ot\mu)(\id-c_2)c_1(v_c\ot v_b\ot v_a)\\
			   =&v_c\ot v_bv_a- k_1v_{d\tri c}\ot v_dv_b +k_2v_{e\tri(d\tri c)}\ot v_ev_d \\
				&+(\id\ot\mu)(\id-c_2)c_1(v_c\ot v_b\ot v_a),								
\end{align*}
where $k_1,k_2\in\fK$ depending on the two-cocycle. To reason $(\id\ot\mu)(\id-c_2)c_1(v_c\ot v_b\ot v_a)\in V\ot\langle S\rangle$ we distinguish two cases and in both cases we use that $\B(V)$ is slim in degree two.

Case (a). Let $c\tri a=a$. Then \[(\id\ot\mu)(\id-c_2)c_1(v_c\ot v_b\ot v_a)=(\id\ot\mu)(\id-c_2)(q_{c,b}v_{c\tri b}\ot v_c\ot v_a)=0\] since Equation~(\ref{eq:degreetworel1}) holds. 

Case (b). Let $c\tri a\neq a$. Then 
\begin{align*}
(\id\ot\mu)(\id-c_2)c_1(v_c\ot v_b\ot v_a)&=(\id\ot\mu)(\id-c_2)(q_{c,b}v_{c\tri b}\ot v_c\ot v_a)\\
&=(\id\ot\mu)(-c_2{}^2)(q_{c,b}v_{c\tri b}\ot v_c\ot v_a)\\
&=k_3v_{c\tri b}\ot v_av_e\in V\ot\langle S\rangle
\end{align*}
for a $k_3\in\fK$, where we used Equation~(\ref{eq:degreetworel2}) in the second equality and Lemma~\ref{le:rackrelations} (1) in the last step. 

In both cases it follows that $\Delta_{1,2}(Z)\in V\ot\langle S\rangle$, which proves the claim.
 
(2) Let $F$ be the linear functional on $V^{\ot 3}$ with $F(v_c\ot v_b\ot v_a)= 1$ and $F(v_i\ot v_j\ot v_k)=0$ for all $(i,j,k)\in X^3\setminus\{(c,b,a)\}$. We prove that $F(\Delta_{1^3}(Z))=1$ but $F(\Delta_{1^3}(Y))=0$ for all $Y\in\langle S\rangle(3)$. It is well-known, and follows by similar arguments as in Lemma~\ref{le:comultcomponents} that $\Delta_{1^3}=(\id\ot\Delta_{1,1})\Delta_{1,2}$. With the result of our calculation of $\Delta_{1,2}(Z)$ in part (1) of the proof (in the case $c\tri a\neq a$) we get 
\begin{align*}
\Delta_{1^3}(Z)=&(\id\ot\Delta_{1,1})\Delta_{1,2}(Z)\\
               =&(\id\ot\Delta_{1,1})(v_c\ot v_bv_a- k_1v_{d\tri c}\ot v_dv_b +k_2v_{e\tri(d\tri c)}\ot v_ev_d \\
							  &+k_3v_{c\tri b}\ot v_av_e),
\end{align*} 
where $k_i\in\fK$. Since $b\neq a$, we obtain $F((\id\ot\Delta_{1,1})(v_c\ot v_bv_a))=1$. Since $X$ is braided, the assumption $c\tri (b\tri a)\neq b\tri a$ implies $d\tri c\neq c$ and hence $F((\id\ot\Delta_{1,1})(v_{d\tri c}\ot v_dv_b)=0$. Now consider the summand 
\begin{align*}
    F((\id\ot\Delta_{1,1})&(v_{e\tri(d\tri c)}\ot v_ev_d))\\
    &=F(v_{e\tri(d\tri c)}\ot v_e\ot v_d)+F(v_{e\tri(d\tri c)}\ot q_{e,d}v_{e\tri d}\ot v_e).
\end{align*} 
Since $c$ does not commute with $a$, it follows that $e\neq a$ by Lemma~\ref{le:rackrelations}~(1) and hence $F(v_{e\tri(d\tri c)}\ot q_{e,d}v_{e\tri d}\ot v_e)=0$. By assumption, we either have $a\tri b\neq c$ or $a\tri c\neq b\tri a$. If $a\tri b\neq c$, then $b=(a\tri b)\tri a\neq c\tri a=e$ and thus $F(v_{e\tri(d\tri c)}\ot v_e\ot v_d)=0$. If $a\tri c\neq b\tri a$, then $d\neq a$ since otherwise $c\tri (b\tri a)=a$ and $c\tri (a\tri c)=a$ and hence $b\tri a=a\tri c$. It follows that $F(v_{e\tri(d\tri c)}\ot v_e\ot v_d)=0$. Finally, we get $F((\id\ot\Delta_{1,1})(v_{c\tri b}\ot v_av_e))=0$ since $c\tri b\neq c$. We conclude $F(\Delta_{1^3}(Z))=1$. 

To prove $F(\Delta_{1^3}(\langle S\rangle(3)))=\{0\}$ we show $F(\Delta_{1^3}(v_xv_yv_z))=0$ for all $x,y,z\in\{a,b,d,e\}$. Using Equation~(\ref{eq:partial1n-1}) we get
\begin{align*}
\Delta_{1^3}(v_xv_yv_z)=&(\id\ot\Delta_{1,1})\Delta_{1,2}(v_xv_yv_z)\\
                 =&(\id\ot\Delta_{1,1}\mu)(\id+c_1+c_1c_2)(v_x\ot v_y\ot v_z)\\
                 =&v_x\ot v_y\ot v_z +l_1v_x\ot v_{y\tri z}\ot v_y\\
                  &+l_2v_{x\tri y}\ot v_x\ot v_z + l_3v_{x\tri y}\ot v_{x\tri z}\ot v_x\\
                  &+l_4v_{x\tri(y\tri z)}\ot v_x\ot v_y +l_5v_{x\tri (y\tri z)}\ot v_{x\tri y}\ot v_x,
\end{align*}
where $l_i\in \fK$ depending on the two-cocycle. Since $c\notin\{a,b,d,e\}$ by Lemma~\ref{le:rackrelations} (2)(a), $v_x$ can not be a multiple of $v_c$ and it follows that \[F(v_x\ot v_y\ot v_z +l_1v_x\ot v_{y\tri z}\ot v_y)=0.\] 
The equation 
\[F(v_{x\tri y}\ot v_x\ot v_z)=0\] 
holds since otherwise ($x=b$ and) $b\tri y=c$, but $b\tri y\neq c$ for all $y\in\{a,b,d,e\}$ by Lemma~\ref{le:rackrelations} (2)(c). 
The equation \[F(v_{x\tri y}\ot v_{x\tri z}\ot v_x)=0\] holds since otherwise ($x=a$ and) $a\tri z=b$ and hence $z=b\tri a$, but $b\tri a\notin\{a,b,d,e\}$ by Lemma~\ref{le:rackrelations} (2)(b). 
The equation 
\[F(v_{x\tri(y\tri z)}\ot v_x\ot v_y)=0\] 
holds since otherwise ($x=b$, $y=a$ and) $b\tri (a\tri z)=c$, but $b\tri (a\tri z)\neq c$ for all $z\in \{a,b,d,e\}$ by Lemma~\ref{le:rackrelations} (2)(d). Finally,
\[F(v_{x\tri (y\tri z)}\ot v_{x\tri y}\ot v_x)=0\] 
since otherwise $a\tri y=b$ and hence $y=b\tri a$, but $b\tri a\notin\{a,b,d,e\}$ by Lemma~\ref{le:rackrelations} (2)(b). We conclude $F(\Delta_{1^3}(v_xv_yv_z))=0$.  
\end{proof}

\begin{lem}\label{le:lyinginS}
Assume that the rack $X$ is braided, that there are pairwise distinct elements $a,b,c\in X$ with $a\tri b\neq b$ and that $\B(V)$ is slim in degree two. Let $d=c\tri (b\tri a)$, $e=d\tri (c\tri b)$ and $S=\{v_a,v_b,v_d,v_e\}$. Define   
\[Z=\mu(\id-c_1c_2+(c_1c_2)^2)(v_c\ot v_b\ot v_a)\in\B(V),\]
and let $C$ be the subalgebra of $\B(V)$ generated by $S\cup\{Z\}$. Then the following hold.
\begin{enumerate}
\item If $c\tri b=b$ and $c\tri a\neq a$, then $Z\in\langle S\rangle$.
\item If $c\tri a=a$ and $c\tri b\neq b$, then $Z\in\langle S\rangle$.
\item If $c\tri (b\tri a)=b\tri a$ and $c\tri a\neq a$, then $Z\in\langle S\rangle$.
\end{enumerate}
\end{lem}

\begin{proof}
(1) Our strategy is to prove that in $\B(V)$ the relation
\begin{align}\label{eq:relationdeg3a}
  \mu(\id-c_1c_2+(c_1c_2)^2-(c_1c_2)^3)(v_b\ot v_a\ot v_e)=0 
\end{align}
holds. Note that $c_1c_2(v_b\ot v_a\ot v_e)\in\fK v_c\ot v_b\ot v_a$ since \[b\tri (a\tri e)=b\tri (a\tri (c\tri a))=b\tri c=c,\] where we used $e=c\tri a$ from Lemma~\ref{le:rackrelations} (1). Thus the claim of the Lemma follows from relation~(\ref{eq:relationdeg3a}) since it implies \[v_bv_av_e-q\mu(\id-c_1c_2+(c_1c_2)^2)(v_c\ot v_b\ot v_a)=0\in\B(V)\] and hence $v_bv_av_e=qZ$ in $\B(V)$ for a $q\in\fK$ and $v_bv_av_e\in\langle S\rangle$. For the proof of Equation~(\ref{eq:relationdeg3a}) we use Equation~(\ref{eq:partial1n-1}) and compute 
\begin{align*}
&\Delta_{1,2}\mu(\id-c_1c_2+(c_1c_2)^2-(c_1c_2)^3)(v_b\ot v_a\ot v_e)\\
=&(\id\ot\mu)(\id+c_1-c_1^2c_2+c_1(c_1c_2)^2-c_1(c_1c_2)^3-(c_1c_2)^4)(v_b\ot v_a\ot v_e)\\
=&(\id\ot\mu)(\id-c_1^2c_2-c_1(c_1c_2)^3)(v_b\ot v_a\ot v_e)\\
 &+(\id\ot\mu)(c_1+c_1(c_1c_2)^2-(c_1c_2)^4)(v_b\ot v_a\ot v_e).
\end{align*}
By Lemma~\ref{le:rackrelations} (1) and since $X$ is braided, we have $a\tri e=c$ and hence $b$ commutes with $a\tri e$ by assumption. Since $\B(V)$ is slim in degree two, we obtain $c_1^2c_2(v_b\ot v_a\ot v_e)=c_2(v_b\ot v_a\ot v_e)$ by Lemma~\ref{le:degreetworel} (1). Using the braid relation~(\ref{eq:cbraid121}) repeatedly we get $c_1(c_1c_2)^3=c_2c_1^2c_2c_1^3$. Note that $a\tri e=c$ hence $e\neq c$ by assumption and $(b\tri a)\tri e=e$ since $a\tri ((b\tri a)\tri e)=b\tri c=c$. Thus \[c_2c_1^2c_2c_1^3(v_b\ot v_a\ot v_e)=-c_2c_1^2c_2(v_b\ot v_a\ot v_e)=-c_2^2(v_b\ot v_a\ot v_e),\] where we used that $\B(V)$ is slim in degree two and Lemma~\ref{le:degreetworel} (2) for the first and (1) for the second equality. It follows that
\begin{align*}
  &(\id\ot\mu)(\id-c_1^2c_2-c_1(c_1c_2)^3)(v_b\ot v_a\ot v_e)\\
  =&(\id\ot\mu)(\id-c_2+c_2^2)(v_b\ot v_a\ot v_e)=0, 
\end{align*}
where we used again that $\B(V)$ is slim in degree two and $a\tri e\neq e$. Using the braid Equation~(\ref{eq:cbraid121}) we obtain $c_1(c_1c_2)^2=c_2c_1c_2^3$ and $c_2c_1c_2^3(v_b\ot v_a\ot v_e)=-c_2c_1(v_b\ot v_a\ot v_e)$ by Lemma~\ref{le:degreetworel}~(2) since $a\tri e\neq e$. The braid relation (\ref{eq:cbraid121}) leads to $(c_1c_2)^4=c_1^2c_2^2c_1c_2^3$ and we have \[c_1^2c_2^2c_1c_2^3(v_b\ot v_a\ot v_e)=-c_1^2c_2^2c_1(v_b\ot v_a\ot v_e)=-c_2^2c_1(v_b\ot v_a\ot v_e),\] where the first equation holds again since $a\tri e\neq e$ and by Lemma~\ref{le:degreetworel}~(2) and the second equation holds since $(b\tri a)\tri e=e$ and by Lemma~\ref{le:degreetworel}~(1) since $\B(V)$ is slim in degree two. Note that \[b\tri e=b\tri (c\tri a)=c\tri (b\tri a)=d\] and $e=c\tri a\neq c\tri (b\tri a)=d$ since $a\neq b\tri a$. Hence $b\tri e\neq e$ and it follows that
\begin{align*}
&(\id\ot\mu)(c_1+c_1(c_1c_2)^2-(c_1c_2)^4)(v_b\ot v_a\ot v_e)\\
=&(\id\ot\mu)(\id-c_2+c_2^2)c_1(v_b\ot v_a\ot v_e)=0,
\end{align*}
by Lemma~\ref{le:degreetworel}~(2). We conclude \[\Delta_{1,2}\mu(\id-c_1c_2+(c_1c_2)^2-(c_1c_2)^3)(v_b\ot v_a\ot v_e)=0,\] which implies $\mu(\id-c_1c_2+(c_1c_2)^2-(c_1c_2)^3)(v_b\ot v_a\ot v_e)=0$ since $\Delta_{1,2}$ is injective. 

(2) Let $Z'=Z-\mu(c_1c_2)^3(v_c\ot v_b\ot v_a)$. We claim $Z'=0\in\B(V)$, which can be shown by proving $\Delta_{1,2}(Z')=0$ with similar arguments as in (1). We only describe how the assertion follows from the claim. By the braid relation~(\ref{eq:cbraid121}), the equation $(c_1c_2)^3=(c_2^2c_1)^2$ holds. Since $c\tri a=a$ and since $\B(V)$ is slim in degree two, we have \[(c_2^2c_1)^2(v_c\ot v_b\ot v_a)=c_2^2c_1^2(v_c\ot v_b\ot v_a)\] by Lemma~\ref{le:degreetworel}~(1). Moreover, \[c_2^2c_1^2(v_c\ot v_b\ot v_a)\in\fK c_2^2(v_b\ot v_{c\tri b}\ot v_a)\] since $(c\tri b)\tri c=b$ (since $X$ is braided) and \[c_2^2(v_b\ot v_{c\tri b}\ot v_a)\in\fK v_b\ot v_{((c\tri b)\tri a)\tri (c\tri b)}\ot v_{(c\tri b)\tri a}.\] Since $a$ commutes with $c$ but not with $b$, it does not commute with $c\ot b$ and $((c\tri b)\tri a)\tri (c\tri b))=a$ since $X$ is braided. Hence \[\mu(c_1c_2)^3(v_c\ot v_b\ot v_a)\in\fK v_bv_av_d\in\langle S\rangle.\] Thus $Z'=0$ implies $0=v_bv_av_d-kZ$ in $\B(V)$ for a $k\in\fK$ and thus $Z\in\langle S\rangle$. 

(3) Let $Z'= Z-\mu(c_1c_2)^3(v_c\ot v_b\ot v_a)$. We claim $Z'=0 \in\B(V)$, which can be shown by proving $\Delta_{1,2}(Z')=0$ with similar arguments as in (1). We only describe how the assertion follows from the claim. Note that $(c_1c_2)^3=(c_1^2c_2)^2$ by the braid relation~(\ref{eq:cbraid121}). Since $c\tri (b\tri a)=b\tri a$ and since $\B(V)$ is slim in degree two, we have \[(c_1^2c_2)^2(v_c\ot v_b\ot v_a)=c_1^2c_2^2(v_c\ot v_b\ot v_a)\] by Lemma~\ref{le:degreetworel}~(1). Moreover, \[c_1^2c_2^2(v_c\ot v_b\ot v_a)\in\fK (v_{(c\tri a)\tri c}\ot v_{c\tri a}\ot v_{b\tri a})=\fK (v_a\ot v_e\ot v_d)\] and hence $\mu(c_1c_2)^3(v_c\ot v_b\ot v_a)\in\fK v_av_ev_d\subset\langle S\rangle$. Thus $Z'=0$ implies $Z\in\langle S\rangle$.
\end{proof}

In the next Lemma and the following Theorem we use abbreviating notations. For a pair of elements $x,y\in X$ we write $x\square y$ if $x\tri y=y$ but $x\neq y$. For elements $x_1,x_2,...,x_n,y\in X$ we abbreviate \[x_1x_2\cdots x_n\tri y=x_1\tri (x_2\tri (...\tri(x_n\tri y))).\] 

As e.g. in \cite[Section~1.3]{MR2891215} for $n\in\ndN$ we define the \textbf{Hurwitz action} of the braid group $\mathbb{B}_n$ on $X^n$ by \[\sigma_i(x_1,\ldots,x_n)=(x_1,\ldots,x_{i-1},x_i\tri x_{i+1},x_i,x_{i+2},\ldots,x_n),\] for all $x_1,\ldots,x_n\in X$, $i\in\{1,\ldots,n-1\}$. We denote the orbit of a tuple $(x_1,\ldots,x_n)$ with respect to the Hurwitz action by $\mathcal{O}(x_1,\ldots,x_n)$.

\begin{lem}\label{le:hurwitzorbit}
Assume that $X$ is braided and let $a,b,c,d\in X$ be pairwise not commuting elements except $a\square c$. Assume that $b\square (c\tri d)$ and $a\square (b\tri d)$. Then the Hurwitz orbit $\mathcal{O}(a,b,c,d)$ consists of the following four types of tuples  
\begin{itemize}
    \item[(x)] $(x_1,x_2,x_3,x_4)$ such that $x_1\square x_3$, $x_2\square (x_3\tri x_4)$ and $x_1\square (x_2\tri x_4)$.  
    \item[(y)] $(y_1,y_2,y_3,y_4)$ such that $y_2\square y_3$, $y_1\square y_4$ and $y_1\square (y_2\tri(y_3\tri y_4))$.
    \item[(z)] $(z_1,z_2,z_3,z_4)$ such that $z_2\square z_4$, $z_1\square (z_3\tri z_4)$ and $z_1\square (z_2\tri z_3)$.
    \item[(u)] $(u_1,u_2,u_3,u_4)$ such that $u_1\square u_2$, $u_3\square u_4$ and $u_1\square (u_2\tri (u_3\tri u_4))$.
\end{itemize}
The unmentioned components of the tuples do not commute.
\end{lem}

\begin{proof}
We only roughly sketch the proof here because it is lengthy but straightforward using the definition of a braided rack. The tuple $(a,b,c,d)\in\mathcal{O}(a,b,c,d)$ is of type $(x)$ by assumption. As an example we show that the action of $\sigma_1$ on a tuple of type $(x)$ leads to a tuple of type $(y)$. If we apply $\sigma_1$ on a tuple of the form $(x)$, we get a tuple $y=(y_1,y_2,y_3,y_4)=(x_1\tri x_2,x_1,x_3,x_4)$. From $x_1\square x_3$ it follows directly that $y_2\square y_3$. Since $x_1\square (x_2\tri x_4)$, we have $x_1\tri (x_2\tri x_4)=x_2\tri x_4$. Applying $x_4\tri$ on that equation we get $(x_4\tri x_1)\tri x_2=x_2$ since $X$ is braided and $x_4\tri x_2\neq x_2$. Hence \[y_4\tri y_1=x_4\tri (x_1\tri x_2)=x_1\tri((x_4\tri x_1)\tri x_2)=x_1\tri x_2=y_1.\] Moreover, $y_1\neq y_4$ since otherwise $x_1\tri x_2=x_4$ and hence $x_1=x_2\tri x_4$. It follows that $y_1\square y_4$. The relation $x_2\square (x_3\tri x_4)$ implies $(x_1\tri (x_2\tri x_1))\square (x_3\tri x_4)$ and $(y_1\tri y_2)\square (y_3\tri y_4),$ what is equivalent to $y_1\square (y_2y_3\tri y_4)$. It is easy to check that other components of $y$ do not commute. 
In this way we proved in the other $11$ cases that applying $\sigma_1,\sigma_2$ or $\sigma_3$ on a tuple of the form $(x),(y),(z)$ or $(u)$ always again yields a tuple of the form $(x),(y),(z)$ or $(u)$. Figure~\ref{di:hurwitzorbit} shows the action of the generators of the braid group on the four types of tuples.

\begin{figure}[h]
\begin{tikzcd}[row sep=large, column sep=large]
        & (x) \ar[d,"{\sigma_1, \sigma_3}"]  \ar[ddl,"\sigma_2"] & \\
        & (y) \ar[dl,"{\sigma_1, \sigma_3}"] \ar[loop right,"\sigma_2"]& \\
      (z) \ar[rr,"\sigma_2"] \ar[uur, bend left,"{\sigma_1, \sigma_3}"] &   & (u) \ar[uul, bend right,"\sigma_2", labels=swap] \ar[loop right, "{\sigma_1, \sigma_3}"]    
\end{tikzcd}
\caption{Hurwitz action on quadruples}
\label{di:hurwitzorbit}
\end{figure}
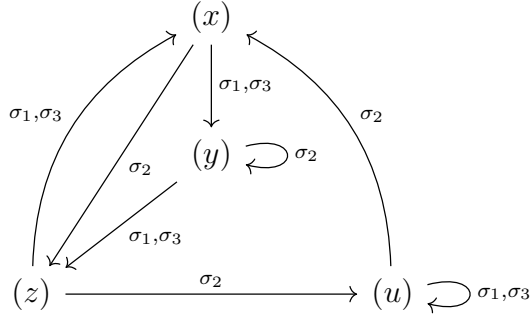
\end{proof}

\begin{lem}\label{le:Scommute}
Assume that $X$ is braided and let $a,b,c,d\in X$ be pairwise not commuting elements except $a\square c$. Assume that $b\square (c\tri d)$ and $a\square (b\tri d)$. Let $e=abc\tri d$, $f=eab\tri c$, $g=fea\tri b$ and $S=\{b,c,d,e,f,g\}$. Then two distinct elements in $S$ do not commute, at best $d\tri e=e$, $c\tri f=f$ and $b\tri g=g$.
\end{lem}

\begin{proof}
In $X$ the relations
\begin{align*}
e&=abc\tri d=ac\tri d, \\
f&=eab\tri c=a\tri((c\tri d)\tri(b\tri c))=a\tri (b\tri d)=b\tri d, \\
g&=fea\tri b=(b\tri d)\tri((a\tri(c\tri d))\tri (a\tri b))=(b\tri d)\tri (a\tri b)= a\tri d
\end{align*}
hold. Since $b\tri a\neq a$, $bc\tri d=c\tri d$ and $ac\tri d\neq c\tri d$, it follows that $b\tri e=b\tri(ac\tri d)\neq ac\tri d=e$. Since $b\tri d\neq d$, it follows that $b\tri f=bb\tri d\neq b\tri d=f$. 

From $(c\tri d)\tri c=d\neq c$, $c\tri a=a$ and $ac\tri d\neq c\tri d$, it follows that $c\tri e=cac\tri d\neq ac\tri d=e$. Since $c\tri d\neq d$, $c\tri a=a$ and $a\tri d\neq d$, it follows that $c\tri g=ca\tri d\neq a\tri d=g$. 

The relation $d\tri f\neq f$ follows from $d\tri b\neq b$, the relation $d\tri g\neq g$ from $d\tri a\neq a$. 

The equation $g\tri e=e$ would imply $c\tri d=d$ and hence $d=c$, a contradiction. The equation $f\tri g=g$ would imply the contradiction $b\tri a=a$. 

Finally, $f\tri e\neq e$ since otherwise $a\tri((b\tri d)\tri(c\tri d))=ac\tri d$ and hence $(b\tri d)\tri (c\tri d)=c\tri d$ which would imply $b\tri c=c$. 
\end{proof}

\begin{lem}\label{le:hurwitz}
Assume that $X$ is braided and let $a,b,c,d\in X$ be pairwise not commuting elements except $a\square c$. Assume that $b\square (c\tri d)$ and $a\square (b\tri d)$. Let $e=abc\tri d$, $f=eab\tri c$, $g=fea\tri b$ and $S=\{b,c,d,e,f,g\}$. Then $S^4\cap\mathcal{O}(a,b,c,d)=\emptyset$. 
\end{lem}

\begin{proof}
First we prove that $abcd\tri x\in S$ and $abcdabcd\tri x=x$ for all $x\in S$. Indeed, using $e=ac\tri d$, $f=b\tri d$ and $g=a\tri d$ (from the proof of Lemma~\ref{le:Scommute}) it follows that
\begin{align*}
abcd\tri b&=ab\tri((c\tri d)\tri (c\tri b)))=a\tri ((c\tri d)\tri c)=a\tri d=g,\\
abcd\tri c&=ab\tri d=b\tri d=f,\\
abcd\tri d&=abc\tri d=ac\tri d=e,\\
abcd\tri e&=abcd\tri(ac\tri d)=abc\tri ((d\tri c)\tri a),\\
           &=ab\tri (d\tri a)=(b\tri d)\tri b=d,\\
abcd\tri f&=abcd\tri (b\tri d)=abc\tri b=a\tri c=c\ \text{and}\\
abcd\tri g&=abcd\tri (a\tri d)=abc\tri a=ab\tri a=b.
\end{align*} 
Since $\tri$ is left self-distributive, \[\sigma_1\sigma_2\sigma_3(t_1,t_2,t_3,t_4)=(abcd\tri t_4,t_1,t_2,t_3)\] for all $t=(t_1,t_2,t_3,t_4)\in\mathcal{O}(a,b,c,d)$. It follows that $S^4\cap\mathcal{O}$ is invariant under the action of $\sigma_1\sigma_2\sigma_3$. Now assume that there is a tuple $t\in S^4\cap\mathcal{O}$. By Lemma~\ref{le:hurwitzorbit}, the tuple $t$ is of the form $(x),(y),(z)$ or $(u)$ and hence at least two entries of $t$ commute. Thus $t$ has an entry $h\in\{b,c,d\}$ and an entry $abcd\tri h$ by our calculations at the beginning of the proof and Lemma~\ref{le:Scommute}. But then multiple application of $\sigma_1\sigma_2\sigma_3$ on $t$ yields a tuple $t'\in S^4\cap\mathcal{O}$ with two components equal to $h$ or two components equal to $abcd\tri h$. This is a contradiction since tuples of the form $(x),(y),(z),(u)$ do not have two equal entries and there are no other tuples in $\mathcal{O}(a,b,c,d)$ by Lemma~\ref{le:hurwitzorbit}. We conclude $S^4\cap\mathcal{O}=\emptyset$. 
 \end{proof}

\begin{thm}\label{thm:degree4}
Assume that $X$ is braided. Let $a,b,c,d\in X$ pairwise not commuting elements except $a\square c$ and such that $b\square (c\tri d)$ and $a\square (b\tri d)$. Let $e=abc\tri d$, $f=eab\tri c$ and $g=fea\tri b$. Assume that $\B(V)$ is slim in degree two. Define the element \[Z(a,b,c,d)=\mu(\id-c_1c_2c_3+(c_1c_2c_3)^2-(c_1c_2c_3)^3)(v_a\ot v_b\ot v_c\ot v_d)\] of $\B(V)$. Let $S=\{v_b,v_c,v_d,v_e,v_f,v_g\}$ and let $C$ be the subalgebra of $\B(V)$ generated by $S\cup\{Z\}$. Then the following hold.
\begin{enumerate}
\item The subalgebra $C$ is a left coideal subalgebra of $\B(V)$ in the category of $\ndN_0$-graded $\fK G_X$-comodules. 
\item If $e\neq d$, $f\neq c$ or $g\neq b$, then $C$ is not generated in degree one.
\end{enumerate}
\end{thm}

\begin{proof}
(1) Since $Z$ is $G_X$-homogeneous of degree $g_ag_bg_cg_d$, the subalgebra $C$ is a $\fK G_X$-subcomodule. By Corollary~\ref{cor:lcsaExt}, it suffices to prove $\Delta_{1,3}(Z)\in V\ot\langle S\rangle$. For this we use Equation~(\ref{eq:partial1n-1}) and get \[\Delta_{1,3}(Z)=s_1+s_2+s_3,\] where
\begin{align*}
s_1=&(\id\ot\mu)(\id-c_1(c_1c_2c_3)+c_1c_2(c_1c_2c_3)^2\\
     &-(c_1c_2c_3)^4)(v_a\ot v_b\ot v_c\ot v_d),\\
s_2=&(\id\ot\mu)(c_1-c_1c_2c_1c_2c_3-c_1(c_1c_2c_3)^3)(v_a\ot v_b\ot v_c\ot v_d),\\
s_3=&(\id\ot\mu)(c_1c_2+c_1(c_1c_2c_3)^2-c_1c_2(c_1c_2c_3)^3)(v_a\ot v_b\ot v_c\ot v_d).	 
\end{align*}
The first summand $s_1$ lies in $V\ot\langle S\rangle$ since 
\begin{align*}
s_1=&v_a\ot v_bv_cv_d +k_1v_{e\tri a}\ot v_ev_bv_c+k_2 v_{f\tri(e\tri a)}\ot v_fv_ev_b\\
    &+k_3 v_{g\tri(f\tri(e\tri a))}\ot v_gv_fv_e,
\end{align*} 
where $k_i\in\fK$ depending on the two-cocycle. 

Next we consider the second summand $s_2$ of $\Delta_{1,3}(Z)$. Using the braid Equation~(\ref{eq:cbraid121}) we get $c_1c_2c_1c_2c_3=c_2c_1c_2^2c_3$ and \[c_2c_1c_2^2c_3(v_a\ot v_b\ot v_c\ot v_d)=c_2c_1c_3(v_a\ot v_b\ot v_c\ot v_d),\] by Lemma~\ref{le:degreetworel} (1), where we use that $b\square (c\tri d)$ and that $\B(V)$ is slim in degree two. Using the braid Equation~(\ref{eq:cbraid12}) we conclude \[c_1c_2c_1c_2c_3(v_a\ot v_b\ot v_c\ot v_d)= c_2c_3c_1(v_a\ot v_b\ot v_c\ot v_d).\] From the braid equations we obtain $c_1(c_1c_2c_3)^3=c_2c_3c_2c_1c_2^2c_3c_2^3$ and from Lemma~\ref{le:degreetworel} it follows that \[c_2c_3c_2c_1c_2^2c_3c_2^3(v_a\ot v_b\ot v_c\ot v_d)=-c_2c_3c_2c_1c_3(v_a\ot v_b\ot v_c\ot v_d)\] since $b\tri c\neq c$, $b\square (c\tri d)$ and $\B(V)$ is slim in degree two. Using again the braid equation we conclude \[c_1(c_1c_2c_3)^3(v_a\ot v_b\ot v_c\ot v_d)=(c_2c_3)^2c_1(v_a\ot v_b\ot v_c\ot v_d).\] Hence for the second summand $s_2$ of $\Delta_{1,3}(Z)$ above we get 
\begin{align*}
s_2=&(\id\ot\mu)(\id-c_2c_3+(c_2c_3)^2)c_1(v_a\ot v_b\ot v_c\ot v_d)\\
=&(\id\ot\mu)(k_4v_{a\tri b}\ot (\id-c_1c_2+(c_1c_2)^2))(v_a\ot v_c\ot v_d).
\end{align*} 
Since $d\tri c\neq c,$ $a\tri c=c$ and $a\tri d\neq d$, we have \[\mu(\id-c_1c_2+(c_1c_2)^2)(v_a\ot v_c\ot v_d)\in\langle v_d,v_c,v_{c\tri g},v_g\rangle\subset\langle S\rangle\] by Lemma~\ref{le:lyinginS} (3). There we used $c\tri g=e$ what follows from $g=a\tri d$ and $e=ac\tri d$ (see the proof of Lemma~\ref{le:Scommute}). Hence $s_2\in V\ot\langle S\rangle$. 

Next we consider the third summand $s_3$ of $\Delta_{1,3}(Z)$ above. Using the braid equations we get $c_1(c_1c_2c_3)^2=c_2c_1c_3c_2c_3^3$ and \[c_2c_1c_3c_2c_3^3(v_a\ot v_b\ot v_c\ot v_d)=-c_2c_1c_3c_2(v_a\ot v_b\ot v_c\ot v_d)\] since $c\tri d\neq d$ and since $\B(V)$ is slim in degree two. Thus finally \[c_1(c_1c_2c_3)^2(v_a\ot v_b\ot v_c\ot v_d)=-c_2c_3c_1c_2(v_a\ot v_b\ot v_c\ot v_d).\] From the braid equations we obtain \[c_1c_2(c_1c_2c_3)^3=c_1c_2c_1^2c_2c_3c_2c_1c_2^2c_3=(c_2c_3)^2c_1c_2c_3^2c_2^2c_3\] and 
\begin{align*}
(c_2c_3)^2c_1c_2c_3^2c_2^2c_3(v_a\ot v_b\ot v_c\ot v_d)&=(c_2c_3)^2c_1c_2c_3^3(v_a\ot v_b\ot v_c\ot v_d)\\
&=-(c_2c_3)^2c_1c_2(v_a\ot v_b\ot v_c\ot v_d)
\end{align*} 
since $b\square (c\tri d)$ and $c\tri d\neq d$. Thus we obtain
\begin{align*}
s_2=&(\id\ot\mu)(\id-c_2c_3+(c_2c_3)^2)c_1c_2(v_a\ot v_b\ot v_c\ot v_d)\\
=&(\id\ot\mu)(k_5v_{a\tri (b\tri c)}\ot (\id-c_1c_2+(c_1c_2)^2))(v_a\ot v_b\ot v_d). 
\end{align*} 
Since $ab\tri d=b\tri d$, $a\tri d\neq d$, we have \[\mu(\id-c_1c_2+(c_1c_2)^2)(v_a\ot v_b\ot v_d)\in\langle v_d,v_b,v_f,v_g\rangle\subset\langle S\rangle\] by Lemma~\ref{le:lyinginS} (3). It follows that $s_3\in V\ot\langle S\rangle.$  

(2) First we prove that $Z\neq 0$. Let $F$ be the linear functional on $V^{\ot 4}$ with $F(v_a\ot v_b\ot v_c\ot v_d)= 1$ and $F(v_i\ot v_j\ot v_k\ot v_l)=0$ for all $(i,j,k,l)\in X^4\setminus\{(a,b,c,d)\}$. We prove that $F(\Delta_{1^4}(Z))=1$. It is well-known, and follows by similar arguments as in Lemma~\ref{le:comultcomponents} that \[\Delta_{1^4}=(\Delta_{1,1}\ot\id^{\ot 2})(\Delta_{2,1}\ot\id)\Delta_{3,1}.\] For the calculation of the components $\Delta_{n,1}$, $1\leq n\leq 3$, we will use Equation~(\ref{eq:partialn-11}). By definition, the element $Z$ is a linear combination of the elements $v_av_bv_cv_d$, $v_ev_av_bv_c$, $v_fv_ev_av_b$, $v_gv_fv_ev_a$ of $\B(V)$. 
Since $a,b,c,d$ are pairwise distinct, the only summand of \[\Delta_{3,1}(v_av_bv_cv_d)=(\mu\ot\id)(\id+c_3+c_3c_2+c_3c_2c_1)(v_a\ot v_b\ot v_c\ot v_d)\] such that the fourth tensor factor lies in $\fK^{\times} v_d$ is a multiple of $v_av_bv_c\ot v_d$. With similar arguments for the third, second and first tensor factor of the summands of $(\Delta_{1,1}\ot\id^{\ot 2})(\Delta_{2,1}\ot\id)(v_av_bv_c\ot v_d)$ it follows that $F(\Delta_{1^4}(v_av_bv_cv_d))=1$. 

Now we prove $F(\Delta_{1^4}(v_ev_av_bv_c))=0$. Since $a$, $b$ and $c$ are not equal to $d$, only the summand \[(\mu\ot\id)(c_3c_2c_1)(v_e\ot v_a\ot v_b\ot v_c)\] of $\Delta_{3,1}(v_ev_av_bv_c)$ can yield a non-zero element of $\fK v_a\ot v_b\ot v_c\ot v_d$ and it is only possible if $e=d$ and if $(\Delta_{1,1}\ot\id)\Delta_{2,1}(v_{d\tri a}v_{d\tri b}v_{d\tri c})$ supplies a non-zero multiple of $v_a\ot v_b\ot v_c$ as summand. But from the assumptions we have $d\tri a\neq c$, $d\tri b\neq c$ and $d\tri c\neq c$ and thus the last requirement can not be fulfilled. 

Now we prove $F(\Delta_{1^4}(v_fv_ev_av_b))=0$. Since $a$, $b$ are not equal to $d$ and $f=b\tri d\neq d$, only the summand \[(\mu\ot\id)(c_3c_2)(v_f\ot v_e\ot v_a\ot v_b)\] of $\Delta_{3,1}(v_fv_ev_av_b)$ can yield a non-zero multiple of $v_a\ot v_b\ot v_c\ot v_d$ and only if $e=d$ and if $(\Delta_{1,1}\ot\id)\Delta_{2,1}(v_fv_{d\tri a}v_{d\tri b})$ supplies a non-zero multiple of $v_a\ot v_b\ot v_c$ as summand. Since $d\tri a\neq c$ and $d\tri b\neq c$, the last condition can only be fulfilled if $f=c$. In this case we would have $c=f=b\tri d$ and it would follow that $d=c\tri b$. From $d=e=ac\tri d$ we then would get $c\tri b=d=ca\tri d$, which would imply $b=a\tri d=g$. Taken all together it follows $F(\Delta_{1^4}(v_fv_ev_av_b))=0$ since $d\neq e$, $c\neq f$ or $b\neq g$ by assumption.

Using $g\neq c$, similar arguments show that $F(\Delta_{1^4}(v_gv_fv_ev_a))= 0$. We conclude that $F(\Delta_{1^4}(Z))=1$ and thus $Z\neq 0$. By Lemma \ref{le:hurwitz}, we know $S^4\cap\mathcal{O}(a,b,c,d)=\emptyset$. Thus $Z\notin\langle S\rangle(4)$. 
\end{proof}

\section{Racks $\mathbb{T}_n$ of transpositions of $\mathbb{S}_n$}\label{se:Transpos}

We start with a general description of the action and coaction of the dual Yetter-Drinfeld module.

\begin{lem}\label{le:dualYD}
Let $G$ be a group and $V\in {^{G}_{G}}\mathcal{YD}$ a finite-dimensional Yetter-Drinfeld module. Assume that $\dim V_g\leq 1$ for all $g\in G$ and let $X=\supp V$. For each $x\in X$ let $v_x\in V_x$ be a non-zero element and let $q_{x,y}\in\fK^{\times}$ such that $x\cdot v_y=q_{x,y}v_{x\tri y}$, for all $x,y\in X$, where $x\tri y=xyx^{-1}$ (see Lemma~\ref{le:YDrackcocylce} (1)). Then the following explicit formulas for the $\fK G$-coaction and $\fK G$-action on the left dual Yetter-Drinfeld module $V^*$ hold. 
\begin{enumerate}
    \item $\delta_{V^*}(v_x{}^*)=x^{-1}\ot v_x{}^*$ for all $x\in X$. 
    \item $x\cdot v_y{}^*=q_{x,y}{}^{-1}v_{x\tri y}{}^*$ for all $x,y\in X$. 
    \item $x^{-1}\cdot v_y{}^*=q_{x,a}v_a{}^*$ for all $x,y\in X$, where $x\tri a=y$.
    \item If $X$ is braided (as rack), then 
\begin{align*}
x^{-1}\cdot v_y{}^*=
\begin{cases}
q_{x,y}v_y{}^* & \text{if } x\tri y=y\\
q_{x,y\tri x}v_{y\tri x}{}^* & \text{if } x\tri y\neq y
\end{cases}
\end{align*}
for all $x,y\in X$.
\end{enumerate}
\end{lem}

\begin{proof}
(1) Let $x\in X$. As in Equation~(\ref{eq:dualcoaction}) defined the $\fK G$-coaction on $V^*$ is given by  
\begin{align*}
    v_x{}^{*}{}_{(-1)}v_x{}^{*}{}_{(0)}(v_y)&=S^{-1}(v_y{}_{(-1)}) v_x{}^{*}(v_y{}_{(0)})\\
    &=S^{-1}(y)v_x{}^{*}(v_y)\\
    &=
    \begin{cases}
    x^{-1} & \text{if } x=y\\
    0 & \text{if } x\neq y
    \end{cases}
\end{align*}  
for all $y\in X$. Hence $v_x{}^{*}{}_{(-1)}\ot v_x{}^{*}{}_{(0)}=x^{-1}\ot v_x{}^{*}$. 

(2) Let $x,y\in X$. As in Equation~(\ref{eq:dualaction}) defined the $\fK G$-action on $V^*$ is given by \[(x\cdot v_y{}^*)(v_k)=v_y{}^*(S(x)\cdot v_k)=v_y{}^*(x^{-1}\cdot v_k)\] for all $k\in X$. From $x\cdot(x^{-1}\cdot v_k)=v_k$ it follows that $x^{-1}\cdot v_k=q_{x,h}{}^{-1}v_h$, where $x\tri h=k$. 
Hence 
\begin{align*}
    (x\cdot v_y{}^*)(v_k)&=v_y{}^*(q_{x,h}{}^{-1}v_h)\\
    &=\begin{cases}
    q_{x,y}{}^{-1} & \text{if } x\tri y=k\\
    0 & \text{else } 
    \end{cases}
\end{align*} 
and $x\cdot v_y{}^*=q_{x,y}{}^{-1}v_{x\tri y}{}^*$.

(3) From $x\cdot(x^{-1}\cdot v_y{}^*)=v_y{}^*$ and since the $G$-homogeneous parts of $V^*$ are one-dimensional (by assumption and part (1)) and by part (2), it follows that $x^{-1}\cdot v_y{}^*=q_{x,a}v_a{}^*$, where $x\tri a=y$.

(4) The claim follows from (3) and Lemma~\ref{le:braidrack}.
\end{proof}

\begin{defi}
Let $n\in\ndN$. We denote the rack of transpositions of $\mathbb{S}_n$ with conjugation as rack operation by $\mathbb{T}_n$ and the generators of the enveloping group $G_{\mathbb{T}_n}$ by $g_{(i\ j)}$ for all $1\leq i<j\leq n$. 
\end{defi}

\begin{rema}
Note that $\mathbb{T}_n$ is a braided rack for all $n\in\ndN$. Indeed, if $(i\ j), (k\ l)\in\mathbb{T}_n$ such that $(i\ j)\tri (k\ l)\neq (k\ l)$, then $(i\ j)\tri (k\ l)=(k\ l)\tri (i\ j)$ and hence $(i\ j)\tri ((k\ l)\tri (i\ j))=(k\ l)$. Moreover, $\mathbb{T}_n$ is indecomposable and generated by $(1\ 2), (1\ 3),\dots, (1\ n)$.
\end{rema}

\begin{lem}\label{le:transp_reps}
\begin{enumerate}
    \item Let $C$ be the subgroup of $G_{\mathbb{T}_n}$ generated by \[\{g_{(1\, 2)},g_{(i\, j)}\mid 2<i<j\leq n\}.\] Then $G_{\mathbb{T}_n}{}^{g_{(1\, 2)}}=C$, where $G_{\mathbb{T}_n}{}^{g_{(1\, 2)}}$ is the centralizer of $g_{(1\, 2)}\in G_{\mathbb{T}_n}$.
    \item Let $\rho$ be a one-dimensional representation of the centralizer $G_{\mathbb{T}_n}{}^{g_{(1\, 2)}}$. Then there are an element $t\in\fK^{\times}$ and a square root of unity $\lambda$ such that $\rho$ is isomorphic to the representation 
    \begin{align*}
        \rho_{t,\lambda}:\ &G_{\mathbb{T}_n}{}^{(1\, 2)}\rightarrow\mathrm{Aut}(\fK v)\\
        &g_{(1\, 2)}\mapsto (v\mapsto tv)\\
        &g_{(i\, j)}\mapsto (v \mapsto\lambda tv), \quad \text{for all $2<i<j\leq n$.}
    \end{align*}
\end{enumerate}    
\end{lem}

\begin{proof}
(1) 
By the universal property of the enveloping group, there is a group homomorphism $h: G_{\mathbb{T}_n}\rightarrow\mathbb{S}_n$ with $h(g_{(i\ j)})=(i\ j)$ for all $1\leq i<j\leq n$. The map $h$ is surjective since $\mathbb{S}_n$ is generated by the set of transpositions. Moreover, the restriction of $h$ on the conjugacy class of $g_{(1\ 2)}$ in $G_{\mathbb{T}_n}$ is injective. Hence $h^{-1}(\mathbb{S}_n{}^{h(g_{(1\ 2)})})=G_{\mathbb{T}_n}{}^{g_{(1\, 2)}}$ (see e.g. \cite[Remark~2.21]{MR2803792}). Since $\mathbb{S}_n{}^{(1\ 2)}=\{(1\, 2), (i\, j)\mid 2<i<j\leq n\}$ and $\ker h\subseteq C$, the claim follows.
    
(2) This follows since $g_{(i\, j)}{}^2=g_{(k\, l)}{}^2\in G_{\mathbb{T}_n}$ for all $1\leq i<j\leq n$, $1\leq k<l\leq n$ by Lemma~\ref{le:relenvelgroup}. 
\end{proof}

\begin{lem}\label{le:dualisotranspos} 
Let $V$ be a Yetter-Drinfeld module over $\fK G_{\mathbb{T}_n}$ such that $\supp V=\mathbb{T}_n$ and $\dim V_{g_x}=1$ for all $x\in\mathbb{T}_n$. Let $v_x\in V_{g_x}$ be non-zero elements and $q_{x,y}\in\fK^{\times}$ such that $g_x\cdot v_y=q_{x,y}v_{x\tri y}$ for all $x,y\in\mathbb{T}_n$. Assume that $q_{x,y}=q_{x,y\tri x}$ for all $x,y\in\mathbb{T}_n$ with $x\tri y\neq y$. Then the linear map $f: V\rightarrow V^*$ with $f(v_x)=v_x{}^*$ for all $x\in\mathbb{T}_n$ is an isomorphism of braided vector spaces with respect to the Yetter-Drinfeld braidings $c_V$ and $c_{V^*}$, respectively. 
\end{lem}

\begin{proof}
For all $x,y\in\mathbb{T}_n$ we have \[(f\ot f)c_V(v_x\ot v_y)=q_{x,y}v_{x\tri y}{}^*\ot v_x{}^*.\] Using Lemma~\ref{le:dualYD} (1) and Lemma~\ref{le:dualYD} (4) we get
\begin{align*}
c_{V^*}(f\ot f)(v_x\ot v_y)&=g_x^{-1}\cdot v_y{}^*\ot v_x{}^*\\
&=\begin{cases}
q_{x,y}v_y{}^*\ot v_x{}^* & \text{if } x\tri y=y\\
q_{x,y\tri x}v_{y\tri x}{}^*\ot v_x{}^* & \text{if } x\tri y\neq y
\end{cases} \\
&=q_{x,y}v_{x\tri y}{}^*\ot v_x{}^*,
\end{align*} 
where the last equation follows since for all $x,y\in\mathbb{T}_n$ either $x\tri y=y$ or $x\tri y=y\tri x$ hold and for all $x,y\in\mathbb{T}_n$ with $x\tri y\neq y$ we have $q_{x,y}=q_{x,y\tri x}$ by assumption. Hence the linear map $f$ is an isomorphism of braided vector spaces.
\end{proof}

\subsection{The rack $\mathbb{T}_3$}

As described at the beginning of Section~\ref{se:different group realization} a Yetter-Drinfeld module over $\fK G_{\mathbb{T}_3}$ such that $\supp V=\mathbb{T}_3$ and $\dim V_{g_x}=1$ for all $x\in\mathbb{T}_3$ is determined by a one-dimensional representation $\rho$ of the centralizer $G_{\mathbb{T}_3}{}^{g_{(1\ 2)}}$. Let $t\in\fK^{\times}$ and let $\rho_t$ be the one-dimensional representation of the centralizer $G_{\mathbb{T}_3}{}^{g_{(1\ 2)}}=\langle g_{(1\ 2)}\rangle$ as in Lemma~\ref{le:transp_reps}. Let $V=M(g_{(1\ 2)},\rho_t)$ and 
\begin{equation}\label{eq:T3basis}
\begin{aligned}
v_a&\in V_{g_{(1\ 2)}}\setminus\{0\},\\ 
v_b&=t^{-1}g_{(2\ 3)}\cdot v_a\in V_{g_{(1\ 3)}}\\ v_c&=t^{-1}g_{(1\ 3)}\cdot v_a\in V_{g_{(2\ 3)}}.
\end{aligned}
\end{equation}
Then the $\fK G_{\mathbb{T}_3}$-module structure on $V$ is given in Table~\ref{tab:twococycleS3}.
\begin{table}[h]
\[\begin{array}{c|ccc}
\cdot & v_a & v_b & v_c \\
\hline
g_{(1\ 2)} &tv_a &tv_c &tv_b \\
g_{(1\ 3)} &tv_c &tv_b &tv_a \\
g_{(2\ 3)} &tv_b &tv_a &tv_c\\
\end{array}\]
\caption{$\fK G_{\mathbb{T}_3}$-module structure on $V$}
\label{tab:twococycleS3}
\end{table}
 
\begin{lem}\label{le:S3mulbij}
Assume that $t=-1$ or $t\in\fK$ such that $t^2-t+1=0$. 
Let $V=M(g_{(1\ 2)},\rho_t)$ and $v_a$, $v_b$ and $v_c$ as in Equation~(\ref{eq:T3basis}). Let $V_1,V_2\subseteq V$ be $\fK G_{\mathbb{T}_3}$-subcomodules such that $V=V_1\oplus V_2$. Then $\mu: \langle V_1\rangle\ot\langle V_2\rangle\rightarrow \B(V)$ is surjective, where $\mu$ is the (restricted) multiplication in $\B(V)$.  
\end{lem}

\begin{proof}
Since $t^3=-1$, in $\B(V)$ the degree two relations 
\begin{align*}
    v_av_b-tv_cv_a+t^2v_bv_c&=0\\
    v_av_c-tv_bv_a+t^2v_cv_b&=0
\end{align*} 
hold by Lemma~\ref{le:degreetworel} (2). Let $S\subseteq\{v_a,v_b,v_c\}$ such that $V_1=\fK S$. Such an $S$ exists since $V_1$ is a $\fK G_{\mathbb{T}_3}$-subcomodule and $v_a,v_b,v_c$ are homogeneous of different $G_{\mathbb{T}_3}$-degrees. Then $V_2=\fK S^c$, where $S^c=\{v_a,v_b,v_c\}\setminus S$. Using the degree two relations it is possible to write any monomial in $\B(V)$ as a linear combination of monomials which are ordered with respect to $S$ and $S^c$, in the sense that each monomial is a product $st$, where $s\in\langle\fK S\rangle$ and $t\in\langle\fK S^c\rangle$. Thus $\mu:\langle\fK S\rangle\ot\langle\fK S^c\rangle\rightarrow \B(V)$ is surjective.
\end{proof}

\begin{thm}\label{thm:3elrack}
Let $t\in\fK^{\times}$, $V=M(g_{(1\ 2)},\rho_t)$ and $v_a$, $ v_b$ and $v_c$ as in Equation~ (\ref{eq:T3basis}).
\begin{enumerate}
\item If $t=-1$ or $t^2-t+1=0$, then all left coideal subalgebras of $\B(V)$ in the category of $\ndN_0$-graded $\fK G_{\mathbb{T}_3}$-comodules are generated in degree one.
\item If $t\neq -1$ and $t^2-t+1\neq 0$, then $\langle v_b,v_av_b-tv_cv_a+t^2v_bv_c\rangle$ is a left coideal subalgebra of $\B(V)$ in the category of $\ndN_0$-graded $\fK G_{\mathbb{T}_3}$-comodules, that is not generated in degree one. 
\end{enumerate}  
\end{thm}

\begin{proof}
(1) Let $C\subseteq\B(V)$ be a (non-trivial, i.e. $C\neq\fK 1$) left coideal subalgebra in the category of $\ndN_0$-graded $\fK G_{\mathbb{T}_3}$-comodules. Then $V\cap C\neq 0$ by Lemma~\ref{le:degreeonepartlc} and $V_1=V\cap C$ is a $\fK G_{\mathbb{T}_3}$-subcomodule. Since $\fK G_{\mathbb{T}_3}$ is cosemisimple, there exists a $\fK G_{\mathbb{T}_3}$-subcomodule $V_2$ such that $V=V_1\oplus V_2$. The restricted multiplication map $\mu: \langle V_1\rangle\ot\langle V_2\rangle\rightarrow\B(V)$ is surjective by Lemma~\ref{le:S3mulbij}. The linear map $f:V\rightarrow V^*$ with $f(v_i)=v_i{}^*$ for all $i\in\mathbb{T}_3$ is an isomorphism of braided vector spaces by Lemma~\ref{le:dualisotranspos} and it induces an isomorphism $\B(V)\rightarrow\B(V^*)$ of $\ndN_0$-graded algebras. In particular, $\langle V_2\rangle\cong\langle V_2{}^*\rangle$ as $\ndN_0$-graded algebras and the claim follows from Corollary~\ref{cor:noextension} (1).

(2) The claim follows from Lemma~\ref{le:degreetworel} (2) and Theorem~\ref{thm:cocycle} (2) since $(1\ 2)\tri (1\ 3)\neq (1\ 3)$ in ${\mathbb{T}_3}$ and $t^3\neq -1$ by assumption. 
\end{proof}

\begin{rema}\label{rema:12dim}
The Nichols algebra $\B(M(g_{(1\ 2)},\rho_{-1}))$ associated to the rack $\mathbb{T}_3$ with constant two-cocycle $-1$ considered in Theorem~\ref{thm:3elrack} (1) is well-known and has dimension $12$ (see e.g. \cite[Proposition~5.4]{MR2803792}). It is isomorphic to the Fomin-Kirillov algebra  $\mathcal{E}_3$. Its subalgebras are studied in \cite{MR3552907} and our result of Lemma~\ref{le:S3mulbij} about the surjectivity of the multiplication map restricted on the subalgebras generated by complementary $\fK G_{\mathbb{T}_3}$-subcomodules of $M(g_{(1\ 2)},\rho_{-1})$ is proved in \cite[Corollary~4.11]{MR3552907}. In the case where $t$ is a primitive $6$-th root of unity and the field $\fK$ has characteristic $0$, the Nichols algebra $\B(V)$ is infinite dimensional by \cite{heckenberger2023simple}.  
\end{rema}

\subsection{The racks $\mathbb{T}_4$, $\mathbb{T}_5$ and $\mathbb{T}_6$}

Let $n\geq 4$, $t\in\fK$ and let $\lambda$ be a square root of unity. Let $\rho_{t,\lambda}$ be a one-dimensional representation of the centralizer $G_{\mathbb{T}_n}{}^{g_{(1\ 2)}}$ as in Lemma~\ref{le:transp_reps}. Let $V(n,t,\lambda)=M(g_{(1\, 2)},\rho_{t,\lambda})$ and for all $1\leq i< j\leq n$ let $v_{(i\ j)}\in V(n,t,\lambda)_{g_{(i\, j)}}$ be non-zero elements. For all $1\leq i< j\leq n$ and $1\leq k< l\leq n$ define $q_{(i\ j),(k\ l)}\in\fK^{\times}$ by \[g_{(i\ j)}\cdot v_{(k\ l)}=q_{(i\ j),(k\ l)}v_{(i\ j)\tri (k\ l)}\] (see Lemma~\ref{le:YDrackcocylce}). 

\begin{lem}\label{le:transpostwococyc}
The following hold.
\begin{enumerate}
    \item $q_{(1\ 2),(1\ 3)}q_{(2\ 3),(1\ 2)}q_{(1\ 3),(2\ 3)}=t^3$
    \item $q_{(1\ 2),(3\ 4)}q_{(3\ 4),(1\ 2)}=t^2$
\end{enumerate}
\end{lem}

\begin{proof}
(1)  There are elements $\lambda_{(1\ 3)}, \lambda_{(2\ 3)} \in\fK^{\times}$ such that 
\[v_{(1\ 3)}=\lambda_{(1\ 3)}g_{(2\ 3)}\cdot v_{(1\ 2)} \text{ and } v_{(2\ 3)}=\lambda_{(2\ 3)}g_{(1\ 3)}\cdot v_{(1\ 2)}\] since $g_{(2\ 3)}g_{(1\ 2)}g_{(2\ 3)}{}^{-1}=g_{(1\ 3)}$ and $g_{(1\ 3)}g_{(1\ 2)}g_{(1\ 3)}{}^{-1}=g_{(2\ 3)}$ in $G_{\mathbb{T}_n}$. By definition of $M(g_{(1\, 2)},\rho_{t,\lambda})$, we have 
\begin{align*}
g_{(1\ 2)}\cdot v_{(1\ 3)}&=\lambda_{(1\ 3)}g_{(1\ 2)}g_{(2\ 3)}\cdot v_{(1\ 2)}=t\lambda_{(1\ 3)}g_{(1\ 3)}\cdot v_{(1\ 2)}\\
&=t\lambda_{(1\ 3)}\lambda_{(2\ 3)}{}^{-1}v_{(2\ 3)}\\
g_{(2\ 3)}\cdot v_{(1\ 2)}&=\lambda_{(1\ 3)}{}^{-1}v_{(1\ 3)}\\
g_{(1\ 3)}\cdot v_{(2\ 3)}&=\lambda_{(2\ 3)}g_{(1\ 3)}{}^2\cdot v_{(1\ 2)}=\lambda_{(2\ 3)}g_{(1\ 2)}{}^2\cdot v_{(1\ 2)}=t^2\lambda_{(2\ 3)}v_{(1\ 2)}.
\end{align*}
Hence 
\begin{align*} 
&q_{(1\ 2),(1\ 3)}=t\lambda_{(1\ 3)}\lambda_{(2\ 3)}{}^{-1}, \quad q_{(2\ 3),(1\ 2)}=\lambda_{(1\ 3)}{}^{-1},\\ 
&q_{(1\ 3),(2\ 3)}=t^2\lambda_{(2\ 3)}
\end{align*}
and $q_{(1\ 2),(1\ 3)}q_{(2\ 3),(1\ 2)}q_{(1\ 3),(2\ 3)}=t^3$.

(2) There is an element $\lambda_{(3\ 4)}\in\fK^{\times}$ such that \[v_{(3\ 4)}=\lambda_{(3\ 4)}g_{(1\ 3)}g_{(2\ 4)}\cdot v_{(1\ 2)}\] since $g_{(1\ 3)}g_{(2\ 4)}g_{(1\ 2)}(g_{(1\ 3)}g_{(2\ 4)}){}^{-1}=g_{(3\ 4)}$ in $G_{\mathbb{T}_n}$. By definition of $M(g_{(1\, 2)},\rho_{t,\lambda})$, we have  
\begin{align*}
g_{(1\ 2)}\cdot v_{(3\ 4)}&=\lambda_{(3\ 4)}g_{(1\ 2)}g_{(1\ 3)}g_{(2\ 4)}\cdot v_{(1\ 2)}\\
&=\lambda_{(3\ 4)}g_{(1\ 3)}g_{(2\ 4)}g_{(1\ 2)}^2g_{(3\ 4)}^{-1}\cdot v_{(1\ 2)}\\
&=t\lambda^{-1}\lambda_{(3\ 4)}g_{(1\ 3)}g_{(2\ 4)}\cdot v_{(1\ 2)}=t\lambda^{-1}v_{(3\ 4)}\ \text{ and }\\
g_{(3\ 4)}\cdot v_{(1\ 2)}&=t\lambda v_{(1\ 2)}.
\end{align*} 
Hence 
\begin{align*} 
q_{(1\ 2),(3\ 4)}=t\lambda^{-1}, \quad q_{(3\ 4),(1\ 2)}=t\lambda
\end{align*}
and $q_{(1\ 2),(3\ 4)}q_{(3\ 4),(1\ 2)}=t^2$.
\end{proof}

\begin{thm}\label{thm:generalexttranspos}
Assume that $t\neq -1$. Then there is a left coideal subalgebra of $\B(V(n,t,\lambda))$ in the category of $\ndN_0$-graded $\fK G_{\mathbb{T}_n}$-comodules that is not generated in degree one.
\end{thm}

\begin{proof}
If $t\neq -1$, then either $t^2\neq 1$ or $t^3\neq -1$. From Lemma~\ref{le:transpostwococyc} it follows that either 
\begin{align*}
q_{(1\ 2),(1\ 3)}q_{(2\ 3),(1\ 2)}q_{(1\ 3),(2\ 3)}=t^3\neq -1 \text{ or } q_{(1\ 2),(3\ 4)}q_{(3\ 4),(1\ 2)}=t^2\neq 1.
\end{align*}
Hence either \[x=\mu(\id-c+c^2)(v_{(1\ 2)}\ot v_{(1\ 3)})\neq 0\] or \[y=\mu(\id-c)(v_{(1\ 2)}\ot v_{(3\ 4)})\neq 0\] in $\B(V(n,t,\lambda))$ by Lemma~\ref{le:degreetworel}, where $c$ denotes the Yetter-Drinfeld braiding of $V(n,t,\lambda)$. It follows that either $\langle v_{(1\ 3)},x\rangle$ or $\langle v_{(3\ 4)},y\rangle$ is a left coideal subalgebra of $\B(V(n,t,\lambda))$ in the category of $\ndN_0$-graded $\fK G_{\mathbb{T}_n}$-comodules that is not generated in degree one by Theorem~\ref{thm:cocycle}.
\end{proof}

\begin{rema}\label{re:twisteq}
For each $n>3$ besides isomorphism there are the two Yetter-Drinfeld modules
\[V(n,-1,1)=M(g_{(1\, 2)},\rho_{-1,1}) \text{ and } V(n,-1,-1)=M(g_{(1\, 2)},\rho_{-1,-1})\]
over $\fK G_{\mathbb{T}_n}$ associated to the rack $\mathbb{T}_n$ and a one-dimensional representation $\rho_{t,\lambda}$ as in Lemma~\ref{le:transp_reps} with $t=-1$ since $\lambda\in\{1,-1\}$. Both (braided vector spaces) can also be realized over $\mathbb{S}_n$ since $t^2=1$. They are the only irreducible Yetter-Drinfeld modules over $\fK\mathbb{S}_n$ such that the corresponding Nichols algebra might be finite-dimensional (for $n=4$ see \cite[Theorem~2.8]{MR2317945}, for $n\geq 5$ see \cite[Theorem~1.1]{MR2786171}). 
In \cite{MR2944712} the two Yetter-Drinfeld modules $V(n,-1,1)$ and $V(n,-1,-1)$ are described in terms of two-cocycles on the rack $\mathbb{T}_n$. One is the constant two-cocycle $-1$, which corresponds to $V(n,-1,1)$. The other two-cocycle is 
\begin{align}\label{eq:SntwococycleChi}
    \chi(\sigma,\tau)=
    \begin{cases}
    1 &\text{if} \ \sigma(i)<\sigma(j)\\
    -1 &\text{if} \ \sigma(i)>\sigma(j) 
    \end{cases} 
\end{align}
for all transpositions $\sigma$ and $\tau=(i\ j)$ with $i<j$, which corresponds to $V(n,-1,-1)$. In \cite[Theorem~3.8]{MR2944712} it is shown that the two-cocycles $-1$ and $\chi$ are equivalent by twist. By \cite[Theorem~3.8]{MR2799090}, it follows that the identity on the underlying vector spaces (and comodules) $V(n,-1,1)\rightarrow V(n,-1,-1)$ is a morphism of braided vector spaces. In particular, it follows that $\B(V(n,-1,1))$ and $\B(V(n,-1,-1))$ are isomorphic as $\ndN_0$-graded algebras and coalgebras. Moreover, the Nichols algebra $\B(V(n,-1,1))$ contains a left coideal subalgebra in the category of $\ndN_0$-graded $\fK G_{\mathbb{T}_n}$-comodules which is not generated in degree one if and only if $\B(V(n,-1,1))$ contains a left coideal subalgebra in the category of $\ndN_0$-graded $\fK G_{\mathbb{T}_n}$-comodules which is not generated in degree one. For $n\in\{4,5\}$ it is known that both Nichols algebras are finite dimensional with 
\begin{align*}
    &\dim\B(V(4,-1,1))=\dim\B(V(4,-1,-1))=576 \\ &\dim\B(V(5,-1,1))=\dim\B(V(5,-1,-1))=8294400 
\end{align*}
(see e.g. \cite[Example~2.3,\ Example~2.4]{MR2944712}). Moreover, they are isomorphic to the Fomin-Kirillov algebra $\mathcal{E}_4$ ($\mathcal{E}_5$, resp.) by \cite{MR1800714}.
\end{rema}

\begin{thm}\label{thm:Tnrack}
Assume that $n\in\{4,5\}$, $t=-1$ and $\lambda\in\{1,-1\}$. Then all left coideal subalgebras of $\B(V(n,t,\lambda))$ in the category of $\ndN_0$-graded $\fK G_{\mathbb{T}_n}$-comodules are generated in degree one. 
\end{thm}

\begin{proof}
By Remark~\ref{re:twisteq}, it suffices to prove the claim for $V(n,-1,1)$. Moreover, we can choose the non-zero elements $v_{(i\ j)}\in V(n,-1,1)_{g_{(i\, j)}}$ such that \[g_{(i\ j)}\cdot v_{(k\ l)}=-v_{(i\ j)\tri (k\ l)}\] for all $1\leq i< j\leq n$ and $1\leq k< l\leq n$, i.e. the corresponding two-cocycle is constant $-1$ (Remark~\ref{re:twisteq}). Let $V_1,V_2\subseteq V(n,-1,1)$ be $\fK G_{\mathbb{T}_n}$-subcomodules such that $V(n,-1,1)=V_1\oplus V_2$. Then the restricted multiplication map \[\mu:\langle V_1\rangle\ot\langle V_2\rangle\rightarrow\B(V(n,-1,1))\] is bijective by \cite[Corollary 4.11]{MR3552907}, where we used that $\B(V(n,-1,1))$ is isomorphic to the Fomin-Kirillov algebra $\mathcal{E}_n$ for $n\in\{4,5\}$ (see Remark~\ref{re:twisteq}). The linear map $f:V(n,-1,1)\rightarrow V(n,-1,1){}^*$ with $f(v_i)=v_i{}^*$ for all $i\in\mathbb{T}_n$ is an isomorphism of braided vector spaces by Lemma~\ref{le:dualisotranspos} and it induces an isomorphism \[\B(V(n,-1,1))\rightarrow\B(V(n,-1,1){}^*)\] of $\ndN_0$-graded algebras. Hence $\langle V_2\rangle\cong\langle V_2{}^*\rangle$ as $\ndN_0$-graded algebras and the claim follows from Corollary~\ref{cor:noextension} (1) since $\B(V(n,-1,1))$ is finite dimensional (see Remark~\ref{re:twisteq}). 
\end{proof}

\begin{thm}\label{thm:T6rack}
Assume that $n=6$, $t\in \fK^{\times}$ and $\lambda\in\fK$ such that $\lambda^2=1$. Then there is a left coideal subalgebra of $\B(V(n,t,\lambda))$ in the category of $\ndN_0$-graded $\fK G_{\mathbb{T}_6}$-comodules that is not generated in degree one. 
\end{thm}

\begin{proof}
For $t\neq -1$ the claim is proved in Theorem~\ref{thm:generalexttranspos}. Let $t=-1$. By Remark~\ref{re:twisteq}, it suffices to prove the claim for $V(6,-1,1)$. Moreover, we can choose the basis $v_{(i\ j)}\in V(6,-1,1)_{g_{(i\, j)}}$ such that $q_{(i\ j),(k\ l)}=-1$ for all $1\leq i< j\leq 6$ and $1\leq k< l\leq 6$, i.e. the corresponding two-cocycle is constant $-1$ (Remark~\ref{re:twisteq}). Let 
\begin{align*}
    C_1&=\{v_{(1\ 2)}, v_{(2\ 3)}, v_{(2\ 4)}, v_{(2\ 5)}, v_{(3\ 5)}, v_{(4\ 5)}, v_{(5\ 6)}\}\\
    C_2&=\{v_{(1\ 3)}, v_{(1\ 4)}, v_{(1\ 5)}, v_{(1\ 6)}, v_{(2\ 6)}, v_{(3\ 4)}, v_{(3\ 6)}, v_{(4\ 6)}\}.
\end{align*}
The left coideal subalgebra \[C=\langle C_1\rangle\subset\B(V(6,-1,1))\] can be extended by an element of degree $7$. We found this using GAP and the package GBNP for computing Gröbner bases (\cite{GAP4}, \cite{GBNP1.0.5}). There we used the realization of the braided vector space $V(6,-1,1)$ over the finite enveloping group $\overline{G_{\mathbb{T}_6}}=G_{\mathbb{T}_6}/\langle g_{(1\ 2)}{}^2\rangle$, which is isomorphic to the symmetric group $\mathbb{S}_6$ (using Proposition~\ref{prop:chooseenvelopingrealization} (2))\footnote{The braided vector space $V(6,-1,1)$ can be realized as Yetter-Drinfeld module over the quotient $\overline{G_{\mathbb{T}_6}}$ since $g_{(1\ 2)}{}^2$ acts trivially in the representation $\rho_{-1,1}$.}. The extension of $C$ in degree $7$ has $\overline{G_{\mathbb{T}_6}}$-degree \[g_{(1\ 6)}g_{(2\ 5)}g_{(3\ 4)}.\] The idea to consider the subalgebra $C$ and to look for an extension in degree $7$ we got from \cite[Proposition~4.11]{MR3552907}. There the Fomin-Kirillov algebra $\mathcal{E}_6$ is considered, which has the same relations as $\B(V(6,-1,1))$ up to degree $7$.  
In \cite[Proposition~4.11]{MR3552907} it is shown that \[\dim\bigoplus_{i=0}^7(\langle C_1\rangle\ot\langle C_2\rangle)(i)\neq\dim\bigoplus_{i=0}^7\B(V(6,-1,1))(i).\] With GAP we calculated \[\dim\bigoplus_{i=0}^6(\langle C_1\rangle\ot\langle C_2\rangle)(i)=85774=\dim\bigoplus_{i=0}^6\B(V(6,-1,1))(i)\] but \[\dim(\langle C_1\rangle\ot\langle C_2\rangle)(7)=228854<\dim\B(V(6,-1,1))(7)=228855.\] From this we knew that there can not be an extension in degree smaller than $7$ by Theorem~\ref{thm:mulbijnoextension} using Lemma~\ref{le:dualisotranspos}.
\end{proof}

\section{The rack $\mathcal{T}$ associated with vertices of the tetrahedron}

Let $\mathcal{T}=\{a,b,c,d\}$ be the rack associated to the conjugacy class of $(2 \ 3 \ 4)$ in the alternating group $\mathbb{A}_4,$ where $a=(2\ 3\ 4)$, $b=(1\ 4\ 3)$, $c=(1\ 2\ 4)$ and $d=(1\ 3\ 2)$ (see e.g. \cite[7.2]{MR2891215}). The rack structure is given explicitly in Table~\ref{tab:rackTetrahedron}.  
\begin{table}[H]
\[\begin{array}{c|cccc}
\tri & a & b & c & d  \\
\hline
a & a & c & d & b \\
b & d & b & a & c \\
c & b & d & c & a \\
d & c & a & b & d \\
\end{array}\]
\caption{The rack $\mathcal{T}$}
\label{tab:rackTetrahedron}
\end{table}
In \cite[Examples~1.3]{MR1994219} the rack $\mathcal{T}$ is described as the rack of a tetrahedron with vertices $\{a,b,c,d\}\subset\mathbb{R}^3$ with center $0$: The rack operation of an element $x\in\mathcal{T}$ is the linear map which fixes $x$ and rotates the orthogonal plane by an angle of $2\pi/3$ with a fixed orientation (right-hand rule pointing the thumb to $x$).  

\begin{rema}
The rack $\mathcal{T}$ is braided (see \cite[Example~5]{MR2891215}), indecomposable and generated by each subset with two elements.
\end{rema}

\begin{lem}\label{le:centr_tetra}
\begin{enumerate}
    \item  \cite[Lemma 33]{MR2891215} The centralizer $G_{\mathcal{T}}{}^{g_a}$ of $g_a\in G_{\mathcal{T}}$ is the subgroup generated by $\{g_a,g_bg_d\}$. Moreover, the relation $g_a^4=(g_bg_d)^2$ holds in $G_{\mathcal{T}}$.
    \item Let $\rho$ be a one-dimensional representation of the centralizer ${G_{\mathcal{T}}}^{g_a}$. Then there are an element $t\in\fK^{\times}$ and an element $\lambda\in\fK^{\times}$ with $\lambda^2=t^4$ such that $\rho$ is isomorphic to the representation 
    \begin{align*}
        \rho_{t,\lambda}:\ &G_{\mathcal{T}}{}^{g_a}\rightarrow\mathrm{Aut}(\fK v)\\
        &g_a\mapsto (v\mapsto tv)\\
        &g_bg_d\mapsto (v \mapsto\lambda v).
    \end{align*}
\end{enumerate}    
\end{lem}

Let $t\in\fK^{\times}$, $\lambda\in\fK^{\times}$ with $\lambda^2=t^4$ and let $\rho_{t,\lambda}$ be a one-dimensional representation of $G_{\mathcal{T}}{}^{g_a}$ as in Lemma~\ref{le:centr_tetra}. Let $V=M(g_a,\rho_{t,\lambda})$ and 
\begin{equation}\label{eq:Tbasis}
\begin{aligned}
v_a&\in V_{g_a}\setminus\{0\},\\ 
v_b&=g_c\cdot v_a\in V_{g_b}\\ v_c&=g_d\cdot v_a\in V_{g_c}\\
v_d&=g_b\cdot v_a\in V_{g_d}.
\end{aligned}
\end{equation}
Then the $\fK G_{\mathcal{T}}$-module structure on $V$ is given in Table~\ref{tab:twococycleT}.
\begin{table}[h]
\[\begin{array}{c|cccc }
\cdot & v_a & v_b & v_c & v_d \\
\hline
g_a &tv_a &tv_c &tv_d &tv_b \\
g_b &v_d &tv_b &\lambda v_a &\frac{\lambda}{t} v_c \\
g_c &v_b &\frac{\lambda}{t} v_d &tv_c &\lambda v_a \\
g_d &v_c &\lambda v_a &\frac{\lambda}{t} v_b &tv_d  \\ 
\end{array}\]
\caption{$\fK G_{\mathcal{T}}$-module structure on $V$}
\label{tab:twococycleT}
\end{table}
  
\begin{lem}\label{le:Tsubsetpaths}
    Let $x,y,z\in\mathcal{T}$ be pairwise distinct elements. Then there are elements $g_1, g_2, g_3\in G_{\mathcal{T}}$ such that \[\{g_1\tri a\}=\{x\},\ \{g_2\tri a, g_2\tri b\}=\{x, y\} \text{ and } \{g_3\tri a, g_3\tri b, g_3\tri c\}=\{x, y, z\}.\] There $g\tri $ denotes the canonical action of $G_{\mathcal{T}}$ on $X$ by conjugation for all $g\in G_{\mathcal{T}}$. 
\end{lem} 

\begin{proof}
The existence of $g_1$ follows from the indecomposability of the rack $\mathcal{T}$. The existence of $g_2$ follows from the transitivity of the action of $g_a$ on $\mathcal{T}\setminus\{a\}$. The claim for the sets with three elements follows directly from the claim for the sets with one element since $g\tri$ defines a bijection on $X$ for all $g\in G_{\mathcal{T}}$. Alternatively, one can easily convince oneself of the correctness of the claim by using the geometrical idea of the rack as vertices of a tetrahedron with rotations as action.    
\end{proof}

\begin{lem}\label{le:Tcomoduletrafo}
 Let $t,\lambda\in\fK^{\times}$ with $\lambda^2=t^4$ be the parameters of a one-dimensional representation of $G_{\mathcal{T}}{}^{g_a}$ as in Lemma~\ref{le:centr_tetra}. Let $V=M(g_a,\rho_{t,\lambda})$ be the corresponding $\fK G_{\mathcal{T}}$-Yetter-Drinfeld module. Let $W\subseteq V$ be a $\fK G_{\mathcal{T}}$-subcomodule. 
 \begin{enumerate}
     \item  If $\dim W=1$, then there is an element $g\in G_{\mathcal{T}}$ such that $W=g\cdot V_{g_a}.$ 
     \item  If $\dim W=2$, then there is an element $g\in G_{\mathcal{T}}$ such that $W=g\cdot (V_{g_a}\oplus V_{g_b}).$ 
     \item  If $\dim W=3$, then there is an element $g\in G_{\mathcal{T}}$ such that $W=g\cdot (V_{g_a}\oplus V_{g_b}\oplus V_{g_c}).$ 
 \end{enumerate}
\end{lem}

\begin{proof}
The claim follows from Lemma~\ref{le:Tsubsetpaths} since $\dim V_{g_x}=1$ for all $x\in\mathcal{T}$.
\end{proof}

\begin{prop}\label{prop:5184}
Let $t\in \fK$ such that $t^2+t+1=0$ 
and $\lambda =-t^{-1}$ be the parameters of a one-dimensional representation of $G_{\mathcal{T}}{}^{g_a}$ as in Lemma~\ref{le:centr_tetra}. Let $V=M(g_a,\rho_{t,\lambda})$ be the corresponding $\fK G_{\mathcal{T}}$-Yetter-Drinfeld module. Let $v_a$, $v_b$, $v_c$ and $v_d$ as in Equation~(\ref{eq:Tbasis}).
\begin{enumerate}
\item The Nichols algebra $\B(V)$ has dimension $5184$.
\item \begin{align*} 
&\dim\langle V_{g_a} \rangle=3,\quad \dim\langle V_{g_a}+ V_{g_b}\rangle=72,\\ &\dim\langle V_{g_a}+V_{g_b}+ V_{g_c}\rangle=1728.
\end{align*}
\item Let $W\subset V$ be a $\fK G_{\mathcal{T}}$-subcomodule. If $\dim W=1$, then $\dim\langle W\rangle=3$, if $\dim W=2$, then $\dim\langle W\rangle=72$, if $\dim W=3$, then $\dim\langle W\rangle=1728$.
\item The linear map 
\begin{align*}
f:\ &V\rightarrow V^*,\\ &v_a\mapsto -tv_a{}^*,\ v_b\mapsto v_c{}^*,\ v_c\mapsto v_b{}^*,\ v_d\mapsto v_d{}^*
\end{align*}
is an isomorphism of braided vector spaces with respect to the Yetter-Drinfeld braidings $c_V$ and $c_{V^*}$.   
\item Let $W\subset V^*$ be a $\fK G_{\mathcal{T}}$-subcomodule.  If $\dim W=1$, then $\dim\langle W\rangle=3$, if $\dim W=2$, then $\dim\langle W\rangle=72$, if $\dim W=3$, then $\dim\langle W\rangle=1728$.   
\end{enumerate}
\end{prop}

\begin{proof}
(1) See \cite[Prop. 36]{MR2891215}.

(2) We proved this using the GAP package GBNP for computing Gröbner bases (\cite{GAP4}, \cite{GBNP1.0.5}). There we used the realization of the braided vector space $V$ over the finite enveloping group $\overline{G_{\mathcal{T}}}=G_{\mathcal{T}}/\langle g_a{}^3\rangle$ (and Proposition~\ref{prop:chooseenvelopingrealization} (2))\footnote{The braided vector space $V$ can be realized as Yetter-Drinfeld module over the quotient $\overline{G_{\mathcal{T}}}$ since $g_a{}^3$ acts trivially in the representation $\rho_{t,\lambda}$ since $t^3=1$ by assumption.}. 

(3) If  $\dim W=2$, then there is an element $g\in G_{\mathcal{T}}$ such that $W=g\cdot(V_{g_a}+ V_{g_b})$ by Lemma~\ref{le:Tcomoduletrafo}. The map \[\alpha_g:\B(V)\rightarrow \B(V),\quad x\mapsto g\cdot x\] is an isomorphism of $\ndN_0$-graded algebras by Lemma~\ref{le:groupactioniso}. Hence $\langle W\rangle\cong\langle V_{g_a}+ V_{g_b}\rangle$ as $\ndN_0$-graded algebras and the claim follows from part (2). If $\dim W\in\{1,3\}$, we can use analogous arguments. 

(4) Using Lemma \ref{le:dualYD} it follows that $\delta_{V^*}(v_x{}^*)=g_x{}^{-1}\ot v_x{}^*$ for all $x\in\mathcal{T}$ and the $\fK G_{\mathcal{T}}$-action on $V^*$ is given in Table~\ref{tab:GTactiondual}.  
\begin{table}[h]
\[\begin{array}{c|cccc }
\cdot & v_a{}^* & v_b{}^* & v_c{}^* & v_d{}^* \\
\hline
g_a{}^{-1} &tv_a{}^* &tv_d{}^* &tv_b{}^* &tv_c{}^* \\
g_b{}^{-1} &-t^2v_c{}^* &tv_b{}^* &-tv_d{}^* &v_a{}^* \\
g_c{}^{-1} &-t^2v_d{}^* &v_a{}^* &tv_c{}^* &-tv_b{}^* \\
g_d{}^{-1} &-t^2v_b{}^* &-tv_c{}^* &v_a{}^* &tv_d{}^*  \\ 
\end{array}\]
\caption{$\fK G_{\mathcal{T}}$-module structure on $V^*$}
\label{tab:GTactiondual}
\end{table}
Using this it is a simple calculation to check that \[(f\ot f)c_V(v_x\ot v_y)=c_{V^*}(f\ot f)(v_x\ot v_y)\] for all $x,y\in\mathcal{T}$.

(5) The isomorphism of braided vector spaces $f:V\rightarrow V^*$ in (4) induces an isomorphism $\B(V)\rightarrow\B(V^*)$ of $\ndN_0$-graded algebras. If $W\subseteq V^*$ is a $\fK G_{\mathcal{T}}$-subcomodule, then $f^{-1}(W)\subseteq V$ is a $\fK G_{\mathcal{T}}$-subcomodule since $f$ maps $G_{\mathcal{T}}$-homogeneous elements to $G_{\mathcal{T}}$-homogeneous elements. Hence the claim follows from (2). 
\end{proof}

\begin{thm}\label{thm:Track}
Let $t\in \fK^{\times}$ and  $\lambda\in\fK^{\times}$ with $\lambda^2=t^4$ be the parameters of a one-dimensional representation of $G_{\mathcal{T}}{}^{g_a}$ as in Lemma~\ref{le:centr_tetra}. Let $V=M(g_a,\rho_{t,\lambda})$ be the corresponding Yetter-Drinfeld module over $\fK G_{\mathcal{T}}$ and denote its braiding by $c_V$. Let $v_a$, $v_b$, $v_c$ and $v_d$ as in Equation~(\ref{eq:Tbasis}). Then the following hold.
\begin{enumerate}
\item If $t^6\neq 1$, then $\langle v_b,v_av_b-tv_cv_a+tv_bv_c\rangle$ is a left coideal subalgebra of $\B(V)$ in the category of $\ndN_0$-graded $\fK G_{\mathcal{T}}$-comodules, that is not generated in degree one. 
\item If $t=-1$ and $\lambda =1$ 
or $t^2-t+1=0$ and $\lambda =-t^{-1}$,  
then $\langle v_a,v_b,z\rangle$ is a left coideal subalgebra of $\B(V)$ in the category of $\ndN_0$-graded $\fK G_{\mathcal{T}}$-comodules that is not generated in degree one, where $z=v_cv_bv_a-\lambda v_av_cv_b+\lambda^2 v_bv_av_c$.  
\item If $t=1$ and $\lambda=-1$, then the left coideal subalgebra $\langle v_a,v_b\rangle$ of $\B(V)$ has an extension $z$ in degree $11$, where $z\in\langle v_a,v_b,v_c\rangle(11)\subset\B(V)(11)$.
\item If $t^2+t+1=0$ and $\lambda =-t^{-1}$, then all left coideal subalgebras of $\B(V)$ in the category of $\ndN_0$-graded $\fK G_{\mathcal{T}}$-comodules are generated in degree one. 
\end{enumerate}  
\end{thm}

\begin{proof}
(1) We have 
\begin{align*}
\Delta_{1,1}\left(\mu\left(1-c_V+c_V{}^2\right)\left(v_a\ot v_b\right)\right)&=\Delta_{1,1}\left(v_av_b-tv_cv_a+tv_bv_c\right)\\
&=\left(1+t\lambda\right)v_a\ot v_b 
\end{align*}
and $1+t\lambda\neq 0$ since $t^6\neq 1$ and $\lambda^2=t^4$ by assumption. Hence the claim follows from Theorem~\ref{thm:cocycle} (2). 

(2) First note that $\delta(z)=g_cg_a\ot z$. In $\B(V)$ the relation \[v_cv_a-v_bv_c+\lambda v_av_b=0\] holds. Indeed, we have \[\Delta_{1,1}(v_cv_a-v_bv_c+\lambda v_av_b)=(1+\lambda t)v_c\ot v_a\] and $1+\lambda t=0$ for the considered parameters $\lambda$ and $t$. Using this relation and Equation~(\ref{eq:partial1n-1}) we get 
\begin{align*}
\Delta_{1,2}(z)&=v_c\ot (1+\lambda^3)v_bv_a+v_d\ot (\lambda t^{-1}v_cv_a-\lambda tv_av_b+\lambda^2 v_bv_c) \\
&=v_c\ot (1+\lambda^3)v_bv_a \in V\ot \langle v_a,v_b\rangle(2).
\end{align*}
Hence by Corollary~\ref{cor:lcsaExt}, it follows that $\langle v_a,v_b,z\rangle$ is a left coideal subalgebra. Moreover, $z\neq 0$ since $(1+\lambda^3)v_bv_a\neq 0\in\B(V)$. To prove that $z\notin\langle v_a, v_b\rangle$ it suffices to prove that the Hurwitz orbit of the triple $(c,b,a)\in \mathcal{T}^3$ does not contain a triple with only $a$ and $b$ as entries. Then it follows that $\Delta_{1^3}(x)\neq \Delta_{1^3}(z)$ and thus $x\neq z$ for all elements $x\in\langle v_a,v_b\rangle(3)$. One checks that the Hurwitz orbit of $(c,b,a)\in \mathcal{T}^3$ is the union of the $\sigma_1\sigma_2$-orbits 
\begin{align*}
&\{(a,c,b),(b,a,c),(c,b,a)\},\{(a,b,d),(d,a,b),(b,d,a)\},\\
&\{(a,d,c),(c,a,d),(d,c,a)\},\{(b,c,d),(d,b,c),(c,d,b)\}.
\end{align*}

(3) We found this extension using the GAP package GBNP for computing Gröbner bases (\cite{GAP4}, \cite{GBNP1.0.5}). There we used the realization of the braided vector space $V$ over the finite enveloping group $\overline{G_{\mathcal{T}}}=G_{\mathcal{T}}/\langle g_a{}^3\rangle$ (and Proposition~\ref{prop:chooseenvelopingrealization} (2))\footnote{The braided vector space $V$ can be realized as Yetter-Drinfeld module over the quotient $\overline{G_{\mathcal{T}}}$ since $g_a{}^3$ acts trivially in the representation $\rho_{t,\lambda}$ since $t^3=1$ by assumption.}. 

(4) Let $C\subseteq\B(V)$, $C\neq\fK 1$, be a left coideal subalgebra in the category of $\ndN_0$-graded $\fK G_{\mathcal{T}}$-comodules. Then $V\cap C\neq 0$ by Lemma~\ref{le:degreeonepartlc} and $V_1=V\cap C$ is a $\fK G_{\mathcal{T}}$-subcomodule by assumption. Since $\fK G_{\mathcal{T}}$ is cosemisimple, there exists a $\fK G_{\mathcal{T}}$-subcomodule $V_2$ such that $V=V_1\oplus V_2$. Then $\dim\langle V_1\rangle\dim\langle V_2\rangle=\dim\B(V)$ by Proposition~\ref{prop:5184} (1) and (3) and $\dim\langle V_2\rangle=\dim\langle V_2{}^*\rangle$ by Proposition~\ref{prop:5184} (3)  and (5). Hence the claim follows from Corollary~\ref{cor:noextension} (2). 
\end{proof}
 
\begin{rema}\label{rema:72dim}
    The Nichols algebra in Theorem~\ref{thm:Track} (2) with parameters $t=-1$ and $\lambda=1$ appeared in \cite[Theorem~6.15]{MR1994219}. It has dimension $72$.  
\end{rema}

\section{The rack $\mathcal{B}$ associated with the faces of a cube}\label{se:cube}

Let $\mathcal{B}=\{a,b,c,d,e,f\}$ be the rack associated to the conjugacy class of $4$-cycles of the symmetric group $\mathbb{S}_4$, where $a=(2\ 3\ 4\ 5)$, $b=(1\ 5\ 6\ 3)$, $c=(1\ 2\ 6\ 4)$, $d=(1\ 3\ 6\ 5)$, $e=(1\ 4\ 6\ 2)$ and $f=(2\ 5\ 4\ 3)$ (see e.g. \cite[Example 2.2]{MR2803792}). The rack structure is explicitly given in Table~\ref{tab:rackCubic}.  
\begin{table}[h]
\[\begin{array}{c|cccccc}
\tri & a & b & c & d & e & f  \\
\hline
a & a & c & d & e & b & f\\
b & e & b & a & d & f & c\\
c & b & f & c & a & e & d\\
d & c & b & f & d & a & e\\
e & d & a & c & f & e & b\\
f & a & e & b & c & d & f\\
\end{array}\]
\caption{The rack $\mathcal{B}$}
\label{tab:rackCubic}
\end{table}

The rack $\mathcal{B}$ has also a geometric interpretation, where $\{a,b,c,d,e,f\}$ represents the set of faces of a cube with center $0$ (see e.g. \cite[Example~1.3,\ Example~3.5]{MR1994219}). The action of an element $x\in\mathcal{B}$ is given by $\pi/2$-rotation of the cube around the center of face $x$ with a fixed orientation (right-hand rule pointing the thumb to $x$, see Figure~\ref{fig:cube}). 

\begin{figure}[h]
\begin{tikzcd}
[x={(2cm,0cm)},
y={(0cm,2cm)},
z={({0.5*cos(45)},{0.5*sin(45)})},
]
\draw[] (0.5,0.95,0.5) node{a};
\draw[] (0.4,0.3,0.5) node{b};
\draw[] (1,0.5,0.5) node{c};

\draw[] (0.9,-0.55,0.5) node{f};
\draw [dashed,->] (0.8,-0.5,0.5)
    arc [start angle=-90, end angle=-180,
         x radius=0.6cm, y radius=0.4cm];

\draw[] (1.75,0.8,0.5) node{\varphi_a};
\draw [<-] (1.75,0.9)
    arc [start angle=0, end angle=-200,
         x radius=2.1cm, y radius=1.1cm];

\draw[] (0,0,0) -- (1,0,0) -- (1,1,0) -- (0,1,0) -- (0,0,0);
\draw[] (1,0,0) -- (1,0,1) -- (1,1,1) -- (1,1,0);
\draw[] (1,1,1) -- (0,1,1) -- (0,1,0);
\draw[densely dashed] (0,0,0) -- (0,0,1) -- (1,0,1);
\draw[densely dashed] (0,0,1) -- (0,1,1);
\end{tikzcd}
\caption{The rack $\mathcal{B}$ as rack of the faces of a cube}
\label{fig:cube}
\end{figure}
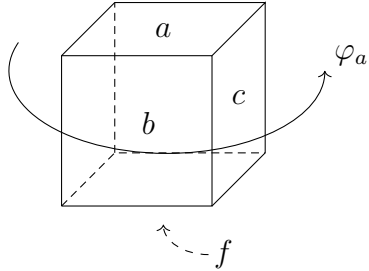

\begin{rema}
The rack $\mathcal{B}$ is braided (see \cite[Example~5]{MR2891215}), indecomposable and generated by each subset with two not-commuting elements.
\end{rema}

\begin{lem}\label{le:centralizerEnvGroupCubicRack}
\begin{enumerate}
    \item  \cite[Lemma 5.10]{MR2803792} The centralizer $G_{\mathcal{B}}{}^{g_a}$ of $g_a\in G_{\mathcal{B}}$ is the abelian subgroup generated by $\{g_a,g_f\}$. Moreover, the relation $g_a{}^4=g_f{}^4$ holds in $G_{\mathcal{B}}$.
    \item There is a group isomorphism \[G_{\mathcal{B}}{}^{g_a}\rightarrow \mathbb{Z}^2/4(e_1-e_2), \ a\mapsto e_1, f\mapsto e_2,\] where $e_1,e_2$ are generators of $\mathbb{Z}^2$. 
    \item Let $\rho$ be a one-dimensional representation of the centralizer $G_{\mathcal{B}}{}^{g_a}$. Then there are an element $t\in\fK^{\times}$ and an element $\lambda\in\fK^{\times}$ with $\lambda^4=1$ such that $\rho$ is isomorphic to the representation 
    \begin{align*}
        \rho_{t,\lambda}:\ G_{\mathcal{B}}{}^{g_a}&\rightarrow\mathrm{Aut}(\fK v)\\
        g_a&\mapsto (v\mapsto tv)\\
        g_f&\mapsto (v\mapsto\lambda tv).
    \end{align*}
\end{enumerate}    
\end{lem}

\begin{proof}
(2) Since $G_{\mathcal{B}}{}^{g_a}$ is abelian and $g_a{}^4=g_f{}^4$ by (1), it suffices to prove that $g_a{}^k\neq g_f{}^l$ for all $1<k,l<4$. Since $g_a\neq g_f$, it suffices to prove that $g_a{}^2\neq g_f{}^2$. To see this we consider the subset $M=\{A,B,C,D,E,F\}$ of the general linear group $\operatorname{GL}(2,\mathbb{F}_9)$, where 
\begin{align*}
& A=\left(\begin{array}{cc}
1 & 0 \\
0 & i	
\end{array}
\right), 
&& B=-(1+i)\left(\begin{array}{cc}
1 & -i \\
-i & 1	
\end{array}
\right), \\
& C=-(1+i)\left(\begin{array}{cc}
1 & -1 \\
1 & 1	
\end{array}
\right),
&& D=-(1+i)\left(\begin{array}{cc}
1 & i \\
i & 1	
\end{array}
\right), \\ 
& E=-(1+i)\left(\begin{array}{cc}
1 & 1 \\
-1 & 1	
\end{array}
\right), 
&& F=\left(\begin{array}{cc}
i & 0 \\
0 & 1	
\end{array}
\right),
\end{align*}
where $i^2=-1$.\footnote{The set $M\subset\operatorname{GL}(2,\mathbb{F}_9)$ is a conjugacy class and a generating set of the unitary group $U(2,3)$. The Hermitian form on the vector space $\mathbb{F}_9{}^2$ is defined with respect to the order two automorphism $\mathbb{F}_9\rightarrow\mathbb{F}_9$, $x\mapsto x^3$.} They form a rack, where the rack operation is conjugation, and the map 
\begin{align*}
\mathcal{B}\rightarrow M, \quad 
&a\mapsto A,\quad b\mapsto B,\quad c\mapsto C,\\&d\mapsto D,\quad e\mapsto E,\quad f\mapsto F
\end{align*}
is a rack isomorphism. Hence from $A^2\neq F^2$ in $\operatorname{GL}(2,\mathbb{F}_9)$, it follows that $g_a{}^2\neq g_f{}^2$ in $G_{\mathcal{B}}{}^{g_a}$.  
\end{proof}

Let $t\in\fK^{\times}$, $\lambda\in\fK^{\times}$ with $\lambda^4=1$ and let $\rho_{t,\lambda}$ be a one-dimensional representation of $G_{\mathcal{B}}{}^{g_a}$ as in  Lemma~\ref{le:centralizerEnvGroupCubicRack} (3). Let $V=M(g_a,\rho_{t,\lambda})$ and 
\begin{equation}\label{eq:Bbasis}
\begin{aligned}
&v_a\in V_{g_a}\setminus\{0\},\quad  
&v_b=g_c\cdot v_a\in V_{g_b},\quad &v_c=g_d\cdot v_a\in V_{g_c}\\
&v_d=g_e\cdot v_a\in V_{g_d},\quad
&v_e=g_b\cdot v_a\in V_{g_e},\quad
&v_f=g_b{}^2\cdot v_a\in V_{g_f}.
\end{aligned}
\end{equation}
The $\fK G_{\mathcal{B}}$-module structure on $V$ is given in Table~\ref{tab:twococycleB}.

\begin{table}[h]
\[\begin{array}{c|cccccc }
\cdot & v_a & v_b & v_c & v_d & v_e & v_f \\
\hline
g_a &tv_a &tv_c &tv_d &tv_e &tv_b & t\lambda v_f \\
g_b &v_e &tv_b &t^2\lambda v_a &t\lambda v_d &v_f &t^2\lambda^3 v_c \\
g_c &v_b &\lambda v_f &tv_c &t^2\lambda v_a &t\lambda v_e &t^2\lambda^2 v_d \\
g_d &v_c &t\lambda v_b &\lambda^2 v_f &t v_d &t^2\lambda v_a &t^2\lambda v_e  \\ 
g_e &v_d &t^2\lambda v_a &t\lambda v_c &\lambda^3 v_f &t v_e &t^2v_b \\
g_f &t\lambda v_a &t\lambda v_e &t\lambda v_b &t\lambda v_c &t\lambda v_d &tv_f \\
\end{array}\]
\caption{$\fK G_{\mathcal{B}}$-module structure on $V$}
\label{tab:twococycleB}
\end{table}

\begin{lem}\label{le:Bsubsetpaths}
    Let $x_1,x_2,\dots,x_5\in\mathcal{B}$ be pairwise distinct elements.
    \begin{enumerate}
        \item There is an element $g\in G_{\mathcal{T}}$ such that $\{g\tri a\}=\{x_1\}$.
        \item If $x_1\tri x_2=x_2$, then there is an element $g\in G_{\mathcal{T}}$ such that $\{g\tri a, g\tri f\}=\{x_1, x_2\}$.
        \item If $x_1\tri x_2\neq x_2$, then there is an element $g\in G_{\mathcal{T}}$ such that $\{g\tri a, g\tri b\}=\{x_1, x_2\}$.
        \item $x_1\tri x_2=x_2$ then there is an element $g\in G_{\mathcal{T}}$ such that $\{g\tri a, g\tri f, g\tri b\}=\{x_1, x_2, x_3\}$.
        \item If $x_1$, $x_2$ and $x_3$ do pairwise not commute, then there is an element $g\in G_{\mathcal{T}}$ such that $\{g\tri a, g\tri b, g\tri c\}=\{x_1, x_2, x_3\}$.
        \item $x_1\tri x_2=x_2$ and $x_3\tri x_4\neq x_4$ then there is an element $g\in G_{\mathcal{T}}$ such that $\{g\tri a, g\tri f, g\tri b, g\tri c\}=\{x_1, x_2, x_3, x_4 \}$.
        \item If $x_1\tri x_2=x_2$ and $x_3\tri x_4=x_4$, then there is an element $g\in G_{\mathcal{T}}$ such that $\{g\tri a, g\tri f, g\tri b, g\tri d\}=\{x_1, x_2, x_3, x_4 \}$.
        \item There is an element $g\in G_{\mathcal{T}}$ such that $\{g\tri a, g\tri b, g\tri c, g\tri d, g\tri e\}=\{x_1, x_2, x_3, x_4, x_5\}$.
    \end{enumerate}  
\end{lem} 

\begin{proof}
One can easily convince oneself of the correctness of the claim by using the geometrical idea of the rack as faces of a cube with rotations as action (see above). Note that commuting elements of $\mathcal{B}$ belong to opposite faces. Hence in order not to be too detailed we give the explicit proof only for (4). Note first that by (2) there is an element $h\in G_{\mathcal{T}}$ such that $\{x_1,x_2\}=\{h\tri a,h\tri f\}$. Since $x_3\notin \{x_1,x_2\}$, it follows that $h^{-1}\tri x_3\in \{b,c,d,e\}$. Since $a$ and $f$ are fixed by $g_a$ and since $g_a$ acts transitively on $\{b,c,d,e\}$, there exists $k\ge 0$ such that $g=g_a^kh$ and $\{g\tri a,g\tri f,g\tri b\}=\{x_1,x_2,x_3\}$.  
\end{proof}

\begin{lem}\label{le:Bcomoduletrafo}
Let $t\in\fK^{\times}$, $\lambda\in\fK^{\times}$ with $\lambda^4=1$ and $\rho_{t,\lambda}$ a one-dimensional representation of $G_{\mathcal{B}}{}^{g_a}$ as in  Lemma~\ref{le:centralizerEnvGroupCubicRack} (3). Let $V=M(g_a,\rho_{t,\lambda})$. Let $W\subseteq V$ be a $\fK G_{\mathcal{B}}$-subcomodule. Let $S=\{x\in\mathcal{B}\mid W\cap V_{g_x}\neq 0\}$. Note that $\vert S\vert=\dim W$ since $\dim V_{g_x}=1$ for all $x\in\mathcal{B}$ and $W$ is a $\fK G_{\mathcal{B}}$-subcomodule. 
  \begin{enumerate}
     \item  If $\dim W=1$, then there is an element $g\in G_{\mathcal{B}}$ such that $W=g\cdot V_{g_a}$. 
     \item  Assume that $\dim W=2$. If the elements of $S$ commute, then there is an element $g\in G_{\mathcal{B}}$ such that $W=g\cdot (V_{g_a}\oplus V_{g_f})$. If the elements of $S$ do not commute, then there is an element $g\in G_{\mathcal{B}}$ such that $W=g\cdot (V_{g_a}\oplus V_{g_b})$. 
     \item  Assume that $\dim W=3$. If two elements of $S$ commute, then there is an element $g\in G_{\mathcal{B}}$ such that $W=g\cdot (V_{g_a}\oplus V_{g_b}\oplus V_{g_f})$. If the elements of $S$ are pairwise not commuting, then there is an element $g\in G_{\mathcal{B}}$ such that $W=g\cdot (V_{g_a}\oplus V_{g_b}\oplus V_{g_c})$.
     \item Assume that $\dim W=4$. If precisely two elements of $S$ commute, then there is an element $g\in G_{\mathcal{B}}$ such that $W=g\cdot (V_{g_a}\oplus V_{g_b}\oplus V_{g_c}\oplus V_{g_f})$. If $S$ consists of two pairs of commuting elements, then there is an element $g\in G_{\mathcal{B}}$ such that $W=g\cdot (V_{g_a}\oplus V_{g_b}\oplus V_{g_d}\oplus V_{g_f})$. 
     \item If $\dim W=5$, then there is an element $g\in G_{\mathcal{B}}$ such that $W=g\cdot\oplus_{x\in\{a,b,c,d,e\}} V_{g_x}$.
 \end{enumerate}
\end{lem}

\begin{proof}
The claim follows from Lemma~\ref{le:Bsubsetpaths} since $\dim V_x=1$ for all $x\in\mathcal{B}$.
\end{proof}

\begin{prop}\label{prop:576NA}
Let $t=-1$ and $\lambda =1$ be the parameters of a one-dimensional representation of $G_{\mathcal{B}}{}^{g_a}$ as in Lemma~\ref{le:centralizerEnvGroupCubicRack}. Let $V=M(g_a,\rho_{t,\lambda})$ be the corresponding $\fK G_{\mathcal{B}}$-Yetter-Drinfeld module. Let $v_a$, $v_b$, $v_c$, $v_d$, $v_e$ and $v_f$ as in Equation~(\ref{eq:Bbasis}). Let $W\subset V$ be a $\fK G_{\mathcal{B}}$-subcomodule. Let $S=\{x\in\mathcal{B}\mid W\cap V_{g_x}\neq 0\}$.
\begin{enumerate}
\item The Nichols algebra $\B(V)$ has dimension $576$.
\item 
\begin{align*}
    &\dim\langle V_{g_a} \rangle=2,\\
    &\dim\langle V_{g_a}+ V_{g_b}\rangle=6,\quad \dim\langle V_{g_a}+ V_{g_f}\rangle=4, \\ 
    &\dim\langle V_{g_a}+V_{g_b}+V_{g_c}\rangle=24,\quad \dim\langle V_{g_a}+V_{g_b}+V_{g_f}\rangle=24,\\
    &\dim\langle V_{g_a}+V_{g_b}+V_{g_c}+V_{g_d}\rangle=96,\\ 
    &\dim\langle V_{g_a}+V_{g_b}+V_{g_d}+V_{g_f}\rangle=144,\\
    &\dim\langle V_{g_a}+V_{g_b}+V_{g_c}+V_{g_d}+V_{g_e}\rangle=288.
\end{align*}
\item If $\dim W=1$, then $\dim\langle W\rangle=2$.
\item Assume that $\dim W=2$. If the elements of $S$ commute, then $\dim\langle W\rangle=4$. If the elements of $S$ do not commute, then $\dim\langle W\rangle=6$.  
\item Assume that $\dim W=3$. If two elements of $S$ commute, then $\dim\langle W\rangle=24$. If the elements of $S$ do pairwise not commute, then $\dim\langle W\rangle=24$. 
\item Assume that $\dim W=4$. If precisely two elements of $S$ commute, then $\dim\langle W\rangle=96$. If $S$ consists of two pairs of commuting elements, then $\dim\langle W\rangle=144$.  
\item If $\dim W=5$, then $\dim\langle W\rangle=288$.
\item The linear map 
\begin{align*}
   F:V\rightarrow V^*,\quad  &v_c\mapsto v_e{}^*, v_e\mapsto v_c{}^*\\
   &v_x\mapsto v_x{}^* \text{ for } x\in\{a,b,d,f\}
\end{align*} 
is an isomorphism of the braided vector spaces $V$ and $V^*$ with respect to the Yetter-Drinfeld braidings $c_V$ and $c_{V^*}$.
\item Let $W^*\subseteq V^*$ be the dual $\fK G_{\mathcal{B}}$-subcomodule of $W$ and $\langle W^*\rangle\subset\B(V^*)$ the subalgebra generated by $W^*$. Then $\dim\langle W\rangle=\dim\langle W^*\rangle$. 
\end{enumerate}
\end{prop}

\begin{proof}
(1) See e.g. \cite[Proposition 5.11]{MR2803792} or \cite[Theorem 6.12]{MR1994219}. 

(2)  We proved this using the GAP package GBNP for computing Gröbner bases (\cite{GAP4}, \cite{GBNP1.0.5}). There we used the realization of the braided vector space $V$ over the finite enveloping group $\overline{G_{\mathcal{B}}}=G_{\mathcal{B}}/\langle g_a{}^4\rangle$ (and Proposition~\ref{prop:chooseenvelopingrealization} (2))\footnote{The braided vector space $V$ can be realized as Yetter-Drinfeld module over the quotient $\overline{G_{\mathcal{B}}}$ since $g_a{}^4$ acts trivially in the representation $\rho_{t,\lambda}$ since $t^4=1$ by assumption.}. 

(3) If $\dim W=1$, then there is an element $g\in G_{\mathcal{T}}$ such that $W=g\cdot V_{g_a}$ by Lemma~\ref{le:Bcomoduletrafo}. The map \[\alpha_g:\B(V)\rightarrow \B(V),\quad x\mapsto g\cdot x\] is an isomorphism of $\ndN_0$-graded algebras by Lemma~\ref{le:groupactioniso}. Hence $\langle W\rangle\cong\langle V_{g_a}\rangle$ as $\ndN_0$-graded algebras and the claim follows from part (2).   
 
(4)-(7) The arguments are analogous to those in part (3). 

(8) Using Lemma \ref{le:dualYD} it follows that $\delta_{V^*}(v_x{}^*)=g_x{}^{-1}\ot v_x{}^*$ for all $x\in\mathcal{B}$ and the $\fK G_{\mathcal{B}}$-action on $V^*$ is given in Table ~\ref{tab:GBactiondual}.  
\begin{table}[h]
\[\begin{array}{c|cccccc }
\cdot & v_a{}^* & v_b{}^* & v_c{}^* & v_d{}^* & v_e{}^* & v_f{}^* \\
\hline
g_a{}^{-1} &-v_a{}^* &-v_e{}^* &-v_b{}^* &-v_c{}^* & -v_d{}^* & -v_f{}^* \\
g_b{}^{-1} &v_c{}^* &-v_b{}^* &v_f{}^* &-v_d{}^* & v_a{}^* &v_e{}^* \\
g_c{}^{-1} &v_d{}^* &v_a{}^* &-v_c{}^* &v_f{}^* &-v_e{}^* & v_b{}^*  \\
g_d{}^{-1} &v_e{}^* &-v_b{}^* &v_a{}^* &-v_d{}^* & v_f{}^* &v_c{}^*  \\ 
g_e{}^{-1} &v_b{}^* &v_f{}^* &-v_c{}^* &v_a{}^* & -v_e{}^* & v_d{}^* \\
g_f{}^{-1} & v_a{}^* &-v_c{}^* &-v_d{}^* &-v_e{}^* & -v_b{}^* &-v_f{}^* \\
\end{array}\]
\caption{$\fK G_{\mathcal{B}}$-module structure on $V^*$  }
\label{tab:GBactiondual}
\end{table} 
 Using this it is a simple calculation to check that \[(F\ot F)c_V(v_x\ot v_y)=c_{V^*}(F\ot F)(v_x\ot v_y)\] for all $x,y\in\mathcal{B}$ and hence the map $F:V\rightarrow V^*$ is a braided isomorphism. 

(9) Define the linear map 
 \begin{align*}
   \widetilde{F}:V\rightarrow V,\quad  &v_c\mapsto v_e, v_e\mapsto v_c\\
   &v_x\mapsto v_x \text{ for all } x\in\{a,b,d,f\}.
\end{align*}  
Then $\dim\langle\widetilde{F}(W)\rangle=\dim\langle W\rangle$ since by (3)-(7) the dimensions only depend on the dimension of $W$ and the number of pairs of commuting elements of $S$, but $F$ is a bijection and the number of pairs of commuting elements of $S$ equals the number of pairs of commuting elements of $\{x\in\mathcal{B}\mid \widetilde{F}(W)\cap V_{g_x}\neq 0\}$. Moreover, $W^*=F(\widetilde{F}(W))$ by definition. The isomorphism of braided vector spaces $F:V\rightarrow V^*$ in (8) induces an isomorphism $\B(V)\rightarrow\B(V^*)$ of $\ndN_0$-graded algebras. It follows that \[\dim\langle W^*\rangle=\dim\langle F(\widetilde{F}(W))\rangle=\dim\langle\widetilde{F}(W)\rangle=\dim\langle W\rangle.\]
\end{proof}

\begin{thm}\label{thm:Brack}
Let $t\in\fK^{\times}$, $\lambda\in\fK^{\times}$ with $\lambda^4=1$ be the parameters of a one-dimensional representation of $G_{\mathcal{B}}{}^{g_a}$ as in Lemma~\ref{le:centralizerEnvGroupCubicRack}. Let $V=M(g_a,\rho_{t,\lambda})$ be the corresponding $\fK G_{\mathcal{B}}$-Yetter-Drinfeld module. Let $v_a$, $v_b$, $v_c$, $v_d$, $v_e$ and $v_f$ as in Equation~(\ref{eq:Bbasis}). Then the following hold.
\begin{enumerate}
\item If $t\neq -\lambda$, then $\langle v_b,v_av_b-tv_cv_a+tv_bv_c\rangle$ or $\langle v_f,v_av_f-t\lambda v_fv_a\rangle$ is a left coideal subalgebra in $\B(V)$ in the category of $\ndN_0$-graded $\fK G_{\mathcal{B}}$-comodules, that is not generated in degree one. 
\item If $t=1$ and $\lambda=-1$, then there is an extension of the left coideal subalgebra $\langle v_a,v_b,v_c\rangle$ in degree 11, which is an element of the subalgebra $\langle v_a,v_b,v_c,v_d\rangle$. 
\item If $\lambda^2=-1$ and $t=-\lambda$, 
then there is an extension of the left coideal subalgebra $\langle v_a,v_b,v_c\rangle$ in degree 7 in the subalgebra $\langle v_a,v_b,v_c,v_d\rangle$.
\item If $t=-1$ and $\lambda=1$, then all left coideal subalgebras of $\B(V)$ in the category of $\ndN_0$-graded $\fK G_{\mathcal{B}}$-comodules are generated in degree one. 
\end{enumerate}  
\end{thm}

\begin{proof}
(1) We have 
\begin{align*}
\Delta_{1,1}\left(\mu\left(1-c_V+c_V{}^2\right)\left(v_a\ot v_b\right)\right)&=\Delta_{1,1}\left(v_av_b-tv_cv_a+tv_bv_c\right)\\
&=\left(1+t^3\lambda\right)v_a\ot v_b 
\end{align*}
and 
\begin{align*}
\Delta_{1,1}\left(\mu\left(1-c_V\right)\left(v_a\ot v_f\right)\right)&=\Delta_{1,1}\left(v_av_f-t\lambda v_fv_a\right)\\
&=\left(1-t^2\lambda^2\right)v_a\ot v_f. 
\end{align*}
The claim follows from Theorem~\ref{thm:cocycle} since $1+t^3\lambda\neq 0$ or $1-t^2\lambda^2\neq 0$ since $t\neq -\lambda$ by assumption. 

(2), (3) We found these extensions using the GAP package GBNP for computing Gröbner bases (\cite{GAP4}, \cite{GBNP1.0.5}). There we used the realization of the braided vector space $V$ over the finite enveloping group $\overline{G_{\mathcal{B}}}=G_{\mathcal{B}}/\langle g_a{}^4\rangle$ (and Proposition~\ref{prop:chooseenvelopingrealization} (2))\footnote{The braided vector space $V$ can be realized as Yetter-Drinfeld module over the quotient $\overline{G_{\mathcal{B}}}$ since $g_a{}^4$ acts trivially in the representation $\rho_{t,\lambda}$ since $t^4=1$ by assumption.}.

(4) Let $C\subseteq\B(V)$, $C\neq\fK 1$, be a left coideal subalgebra in the category of $\ndN_0$-graded $\fK G_{\mathcal{B}}$-comodules. Then $V\cap C\neq 0$ by Lemma~\ref{le:degreeonepartlc} and $V_1=V\cap C$ is a $\fK G_{\mathcal{B}}$-subcomodule by assumption. Since $\fK G_{\mathcal{B}}$ is cosemisimple, there exists a $\fK G_{\mathcal{B}}$-subcomodule $V_2$ such that $V=V_1\oplus V_2$. By Proposition~\ref{prop:576NA} (9), we have $\dim\langle V_2\rangle=\dim\langle V_2{}^*\rangle$.
For $i\in\{1,2\}$ let $S_i=\{x\in\mathcal{B}\mid V_i\cap V_{g_x}\neq 0\}$. We now consider different cases depending on the dimension of $V_1$.
\begin{enumerate}
    \item[(i)] $\dim V_1=1$. Then $\dim V_2=5$ and it follows that
    \[ \dim\langle V_1\rangle\dim\langle V_2\rangle=2\cdot 288=576=\dim\B(V) \]
    by Proposition~\ref{prop:576NA} (3), (7).
    \item[(iia)] $\dim V_1=2$ and the elements of $S_1$ commute. Then $\dim V_2=4$ and $S_2=\mathcal{B}\setminus S_1$ consists of two pairs of commuting elements. It follows that $\dim\langle V_1\rangle\dim\langle V_2\rangle=4\cdot 144=576=\dim\B(V)$ by Proposition~\ref{prop:576NA} (4), (6).
    \item[(iib)] $\dim V_1=2$ and the elements of $S_1$ do not commute. Then $\dim V_2=4$ and exactly two elements of $S_2=\mathcal{B}\setminus S_1$ commute. It follows that $\dim\langle V_1\rangle\dim\langle V_2\rangle=6\cdot 96=576=\dim\B(V)$ by Proposition~\ref{prop:576NA} (4), (6).
    \item[(iii)] $\dim V_1=3$. Then $\dim V_2=3$. It follows that $\dim\langle V_1\rangle\dim\langle V_2\rangle=24\cdot 24=576=\dim\B(V)$ by Proposition~\ref{prop:576NA} (5). 
\end{enumerate}
If $\dim V_1>3$, then it follows that $\dim\langle V_1\rangle\dim\langle V_2\rangle=576=\dim\B(V)$ by changing the roles of $V_1$ and $V_2$. Hence the claim follows from Corollary~\ref{cor:noextension} (2).
\end{proof}

\section{Racks with special subracks}\label{se:rackswithsubracks}

Let $(X,\tri)$ be a finite indecomposable rack of a conjugacy class of a group and $q: X\times X\rightarrow\fK^{\times}$ a two-cocycle. Let $V=\fK X$ be the corresponding Yetter-Drinfeld realization over $\fK G_X$ and $v_x\in V_x\setminus\{0\}$, $x\in X$, such that $g_x\cdot v_y=q_{x,y}v_{x\tri y}$ for all $x,y\in X$. Denote the braiding by $c_q$ and let $\B(V)=\B(X,q)$ be the corresponding Nichols algebra.   

Our tactic in this section is to assume that all left coideal subalgebras of $\B(X,q)$ in the category of $\ndN_0$-graded $\fK G_X$-comodules are generated in degree one and to argue then that $X$ is either $\mathbb{T}_n$, $\mathcal{T}$ or $\mathcal{B}$, where $n\in\ndN$. In our main Theorem~\ref{thm:mainthm} we combine this result with our results about the racks $\mathbb{T}_n$, $\mathcal{T}$ and $\mathcal{B}$ in Sections~\ref{se:Transpos}-\ref{se:cube}.     

\begin{lem}\label{le:noextensiont_3tC}
Assume that all left coideal subalgebras of $\B(X,q)$ in the category of $\ndN_0$-graded $\fK G_X$-comodules are generated in degree one. Let $x,y\in X$ with $x\tri y\neq y$ and let $Y$ be the subrack of $X$ generated by $x$ and $y$. Then $Y$ is isomorphic to $\mathbb{T}_3$, $\mathcal{T}$ or $\mathcal{B}$.
\end{lem}

\begin{proof}
Since all left coideal subalgebras of $\B(X,q)$ are generated in degree one, it follows that $X$ is braided by Proposition~\ref{prop:braid} and that $\B(V)$ is slim in degree two by Theorem~\ref{thm:cocycle}. We distinguish three cases. 
\begin{enumerate}
    \item Assume that $x\tri(x\tri y)=y$. 
    Since $X$ is braided, this implies that $x\tri y=y\tri x$ and $Y=\{x,y,x\tri y\}$. Hence there is a rack isomorphism $Y\rightarrow\mathbb{T}_3$ such that \[x\mapsto (1\ 2),\ y\mapsto (1\ 3).\]
    \item Assume that $x\tri (x\tri y)\neq y$ and $x\tri(x\tri(x\tri y)=y$. 
    Since $X$ is braided, this implies that $x\tri (x\tri y)=y\tri x$ and $Y=\{x,y,x\tri y,y\tri x\}$. Hence there is a rack isomorphism $Y\rightarrow\mathcal{T}$ such that \[x\mapsto a=(2\ 3\ 4),\ y\mapsto b=(1\ 4\ 3)\] (see Table~\ref{tab:rackTetrahedron}).  
    \item Assume that $x\tri(x\tri y)\neq y$ and $x\tri(x\tri (x\tri y))\neq y$. Since $X$ is braided, this implies that $x\tri (x\tri y)\neq y\tri x$. Let $z=x\tri(x\tri y)$. Then $z\neq y\tri x$ and since all left coideal subalgebras of $\B(X,q)$ in the category of $\ndN_0$-graded $\fK G_X$-comodules are generated in degree one, it follows from Theorem~\ref{thm:degree3} that $z$ commutes with at least one of the elements $x$, $y$, $x\tri y$. Since $(x\tri y)\tri z=x\neq z$ and $z\tri x=x\tri y\neq x$, this implies that $y\tri z=z$. Since $X$ is braided, it follows that $x\tri z=y\tri x$ and $x^4\tri y=y$. Moreover, $x\tri(y\tri(y\tri x))=y\tri(y\tri x)$ and $y^3\tri x=x\tri y$ since $X$ is braided. Hence $Y$ has $6$ elements and degree $4$ and there is a rack isomorphism $Y\rightarrow\mathcal{B}$ such that \[x\mapsto a=(2\ 3\ 4\ 5),\ y\mapsto b=(1\ 5\ 6\ 3)\] (see Table~\ref{tab:rackCubic}).      
    \end{enumerate}
    \end{proof}

\begin{prop}\label{prop:rackcontainingtetrahedron}
  Assume that $X$ is not generated by less than $3$ elements and that $X$ contains the rack $\mathcal{T}$ as subrack. Then there is a primitively generated left coideal subalgebra of $\B(X,q)$ in the category of $\ndN_0$-graded $\fK G_X$-subcomodules which can be extended in a degree $\leq 3$. 
\end{prop}

\begin{proof}
  Assume that $\B(X,q)$ does not contain a left coideal subalgebra in the category of $\ndN_0$-graded $\fK G_X$-subcomodules that has an extension in degree two. Then $X$ is braided by Proposition~\ref{prop:braid} and $\B(V)$ is slim in degree two by Theorem~\ref{thm:cocycle}. Since the subrack $\mathcal{T}$ is generated by each two of its elements but $X$ is not generated by less than $3$ elements by assumption, there is an element $z\in X$ with $z\notin\mathcal{T}$. Since $X$ is indecomposable, we can choose $z$ such that it does not commute with at least one element of $\mathcal{T}$. Then $z$ commutes with at most one element of $\mathcal{T}$. Indeed, if $z$ would commute with more than one element of $\mathcal{T}$, it would commute with any element of $\mathcal{T}$ since $\mathcal{T}$ is generated by each two of its elements. Hence there are elements $x,y\in\mathcal{T}$ such that $x\tri y\neq z$ and $z$ commutes with none of $x$, $y$ and $y\tri x$. By Theorem~\ref{thm:degree3}, it follows that the left coideal subalgebra $\langle v_x,v_y,v_{z\tri(y\tri x)},v_{z\tri x}\rangle\subset \B(X,q)$ in the category of $\ndN_0$-graded $\fK G_X$-comodules can be extended in degree three by the element $Z(x,y,z)$.      
\end{proof}

\begin{prop}\label{prop:rackcontainingcube}
 Assume that $X$ is not generated by less than $3$ elements and that $X$ contains the rack $\mathcal{B}$ as subrack. Then there is a primitively generated left coideal subalgebra of $\B(X,q)$ in the category of $\ndN_0$-graded $\fK G_X$-subcomodules that can be extended in a degree $\leq 4$.  
\end{prop}

\begin{proof}
  Assume that $\B(X,q)$ does not contain a left coideal subalgebra in the category of $\ndN_0$-graded $\fK G_X$-subcomodules that has an extension in degree two. Then $X$ is braided by Proposition~\ref{prop:braid} and $\B(V)$ is slim in degree two by Theorem~\ref{thm:cocycle}.
 Recall that $\mathcal{B}$ is generated by every two not commuting elements of $\mathcal{B}$. Since $X$ is not generated by less than $3$ elements and since $X$ is indecomposable, there is an element $z\in X\setminus\mathcal{B}$ which does not commute with at least one element of $\mathcal{B}$. We distinguish two cases. 

(1) Assume that $z\tri x\neq x$ for all $x\in\mathcal{B}$. Choose elements $a,b\in\mathcal{B}$ with $a\tri b\neq b$. Then $a\tri b\neq z$ since $z\notin\mathcal{B}$. Moreover, the element $z$ commutes with none of $a$, $b$, $b\tri a$ by assumption in the considered case. By Theorem~\ref{thm:degree3}, it follows that the left coideal subalgebra $\langle v_a,v_b,v_{z\tri(b\tri a)},v_{z\tri a}\rangle\subset\B(X,q)$ in the category of $\ndN_0$-graded $\fK G_X$-subcomodules can be extended in degree $3$ by the element $Z(a,b,z)$.

(2) Assume that there is an element $a\in\mathcal{B}$ such that $z\tri a=a$. Let $b\in\mathcal{B}$ such that $a\tri b\neq b$. Then $z\tri b\neq b$ since otherwise $z$ would commute with both elements of the generating set $\{a,b\}$ of $\mathcal{B}$ and hence with all elements of $\mathcal{B}$. Let $c=a\tri b$. Then $a\tri c\neq c$, $a\tri c\neq b$ and $b\tri(a\tri c)=a\tri c$ since these relations hold for each two not commuting elements $a$, $b$ of the cube rack $\B$. Moreover, $z\tri c\neq c$ since  $z\tri a=a$, $z\tri b\neq b$ and $a\tri b\neq b$. Since $b\tri c=a$, it follows that $z\tri (b\tri c) =b\tri c$ by assumption and $z\neq b\tri c$ since $b\tri c\in\mathcal{B}$ but $z\notin\mathcal{B}$. Then by Theorem~\ref{thm:degree4}, there is a left coideal subalgebra of $\B(X,q)$ in the category of $\ndN_0$-graded $\fK G_X$-subcomodules that can be extended by the element $Z(z,b,a,c)$ of degree $4$ since $zba\tri c \neq c$. Indeed, if $zba\tri c=c$, then $z\tri (a\tri c)=c$ and hence $a\tri (z\tri c)=c$. It follows $z\tri c=b$ and hence $z=c\tri b\in\B$ but $z\notin\mathcal{B}$. 
\end{proof}

\begin{lem}\label{le:T3property}
Assume that all left coideal subalgebras of $\B(X,q)$ in the category of $\ndN_0$-graded $\fK G_X$-comodules are generated in degree one. Assume that the subrack of $X$ generated by $x$ and $y$ is isomorphic to $\mathbb{T}_3$ for all $x,y\in X$ with $x\tri y\neq y$. Then for all pairwise not commuting elements $x,y,z\in X$ the equality $x\tri (y\tri z)=y\tri z$ holds.
\end{lem}

\begin{proof}
Since all left coideal subalgebras of $\B(X,q)$ in the category of $\ndN_0$-graded $\fK G_X$-comodules are generated in degree one, $X$ is braided by Proposition~\ref{prop:braid} and $\B(V)$ is slim in degree two by Theorem~\ref{thm:cocycle}. Assume that there are pairwise not commuting elements $x,y,z\in X$ such that $x\tri (y\tri z)\neq y\tri z$. Then $x\tri z\neq z$, $x\tri y\neq y$ and $x\tri (y\tri z)\neq y\tri z$. Moreover, $y\tri z=z\tri y$ since the subrack generated by $z$ and $y$ is the rack $\mathbb{T}_3$ by assumption. Hence $x\neq z\tri y$ and by Theorem~\ref{thm:degree3} there is a left coideal subalgebra of $\B(X,q)$ in the category of $\ndN_0$-graded $\fK G_X$-comodules that is not generated in degree one. This is a contradiction to the assumption. 
\end{proof}

\begin{lem}\label{le:T3property2}
Assume that all left coideal subalgebras of $\B(X,q)$ in the category of $\fK G_X$-comodules are generated in degree one. Assume that the subrack of $X$ generated by $x$ and $y$ is isomorphic to $\mathbb{T}_3$ for all $x,y\in X$ with $x\tri y\neq y$. Let $x,a,b,c\in X$ be pairwise distinct elements such that $a\tri b=c$ and $x\tri a\neq a$. Then $x$ commutes with precisely one element of $\{b,c\}$.
\end{lem}

\begin{proof}
Since all left coideal subalgebras of $\B(X,q)$ in the category of $\ndN_0$-graded $\fK G_X$-comodules are generated in degree one, $X$ is braided by Proposition~\ref{prop:braid} and hence $a=b\tri c$. Assume that $x\tri b=b$ and $x\tri c=c$. Then $x\tri a=x\tri (b\tri c)=b\tri c=a$, but $x\tri a\neq a$ by assumption. Hence $x$ commutes with at most one element of $\{b,c\}$. Assume that $x\tri b\neq b$ and $x\tri c\neq c$. Then $x\tri (b\tri c)=b\tri c$ by Lemma~\ref{le:T3property}. But this is a contradiction to the assumption $x\tri a\neq a$.   
\end{proof}

\begin{rema}\label{rema:subracklcsas}
    Assume that all left coideal subalgebras of $\B(X,q)$ in the category of $\ndN_0$-graded $\fK G_X$-comodules are generated in degree one. Let $Y\subset X$ be a subrack and $q':Y\times Y\rightarrow \fK$ the restriction of $q$ on $Y\times Y$. Then all left coideal subalgebras of $\B(Y,q')$ in the category of $\ndN_0$-graded $\fK G_Y$-comodules are generated in degree one 
\end{rema}

\begin{prop}\label{prop:transpositioncomp}
    Assume that $X$ is generated by $n\geq 2$ elements and that any $n-1$ elements of $X$ generate a proper subrack. Assume that all left coideal subalgebras of $\B(X,q)$ in the category of $\ndN_0$-graded $\fK G_X$-comodules are generated in degree one. Assume that the subrack of $X$ generated by $x$ and $y$ is isomorphic to $\mathbb{T}_3$ for all $x,y\in X$ with $x\tri y\neq y$. Then $X$ is the rack of transpositions $\mathbb{T}_{n+1}$.  
\end{prop}

\begin{proof}
Since all left coideal subalgebras of $\B(X,q)$ in the category of $\ndN_0$-graded $\fK G_X$-comodules are generated in degree one, $X$ is braided by Proposition~\ref{prop:braid}. We prove the claim by induction. 

For $n=2$ the claim holds by assumption. To clarify the idea of the proof, we additionally show the proof for the case $n=3$. Choose $x, y\in X$ such that $x\tri y\neq y$ 
and let $X_2\subset X$ be the subrack generated by $x$ and $y$. Then $X_2\cong\mathbb{T}_3$ by assumption. Hence we can identify the elements of $X_2$ with transpositions $x=(1\ 2)$, $y=(1\ 3)$ and $z=(2\ 3)$. Since $X$ is generated by $3$ elements and since $X$ is indecomposable, there exists an element $a\in X\setminus X_2$ that does not commute with at least one element of $X_2$. Hence $a$ does not commute with precisely two elements of $X_2$ by Lemma~\ref{le:T3property2}. Without loss of generality we may assume that $a$ does not commute with $(1\ 2)$ and $(1\ 3)$. It follows that there is a rack isomorphism $X\rightarrow \mathbb{T}_4$ such that \[(1\ 2)\mapsto (1\ 2),\  (1\ 3)\mapsto (1\ 3),\ a\mapsto (1\ 4).\]  

Let $n\geq 4$ and let $X_{n-1}\subset X$ be a subrack generated by $n-1$ elements. Then $X_{n-1}\cong\mathbb{T}_n$ by induction hypothesis and Remark~\ref{rema:subracklcsas}. We identify the elements of $X_{n-1}$ with transpositions, i.e. \[X_{n-1}=\{(i\ j)\mid 1\leq i<j\leq n\}.\] Let $a\in X\setminus X_{n-1}$. Since $X$ is indecomposable, we can assume that $a$ does not commute with at least one element of $X_{n-1}$. Let
\[ T=\{y\in X_{n-1}\mid a\tri y\neq y\}. \]
Then $T^c=X_{n-1}\setminus T$ is a subrack. Indeed, for $x, y\in T^c$ we have $a\tri x=x$ and $a\tri y=y$ and hence $a\tri (x\tri y)=x\tri y$. Hence there is a number $p\in\mathbb{N}$ and a disjoint decomposition
\[ \dot\cup_{\alpha=1}^p M_{\alpha}=\{1,\dots ,n\}\]
such that $T^c=\cup_{\alpha=1}^p \{(i\ j)\mid i,j\in M_\alpha ,i< j\}$. We illustrate the situation by drawing $n$ points and connecting the points $i$ and $j$ if and only if $(i\ j)\in T^c$. 
For an example see Figure~\ref{fig:n7p3}, where we have $n=6$ and $p=3$ and assumed $T^c=\{(1\ 2),(1\ 3),(2\ 3),(4\ 5)\}$. 

\begin{figure}[H]
\centering
\begin{tikzpicture}

\filldraw[black] (0,0) circle (2pt) node[anchor=east]{1};
\filldraw[black] (-0.5,-0.5) circle (2pt) node[anchor=east]{2};
\filldraw[black] (-0.25,1) circle (2pt) node[anchor=east]{3};
\draw[black] (-0.4,0.25) circle (40pt);

\filldraw[black] (4,0) circle (2pt) node[anchor=east]{4};
\filldraw[black] (3.5,0.6) circle (2pt) node[anchor=east]{5};
\draw[black] (3.5,0.25) circle (30pt);

\filldraw[black] (7,0.3) circle (2pt) node[anchor=east]{6};
\draw[black] (7,0.3) circle (20pt);

\draw[gray, thick] (-0.5,-0.5) -- (-0.25,1);
\draw[gray, thick] (0,0) -- (-0.5,-0.5);
\draw[gray, thick] (0,0) -- (-0.25,1);
\draw[gray, thick] (4,0) -- (3.5,0.6);
\end{tikzpicture}
\caption{$n=6$, $p=3$}
\label{fig:n7p3}
\end{figure}

Since $X$ is indecomposable, we have $T\neq\emptyset$ and hence $p>1$. We are going to prove that $p=2$ and one of $M_1$, $M_2$ has cardinality $1$. 

Claim 1: There is at most one $\alpha\in\{1,\dots ,p\}$ such that $\vert M_{\alpha}\vert\geq 2$. 

Proof of claim 1: Assume that there are $\alpha,\beta\in\{1,\dots ,p\}$ and pairwise distinct $i,j,k,l\in\{1,\dots ,n\}$ such that $i,j\in M_{\alpha}$ and $k,l\in M_{\beta}$. Then $(i\ j), (k\ l)\in T^c$ and $\{(i\ k), (i\ l), (j\ k), (j\ l)\}\subset T$. In Figure~\ref{fig:deg4ex} we indicate this constellation. There elements of $T$ are represented by a dotted line. We only plotted relevant points. 

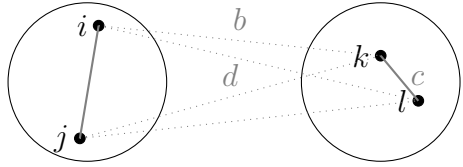
\begin{figure}[H]
\centering
\begin{tikzpicture}

\filldraw[black] (-0.5,-0.5) circle (2pt) node[anchor=east]{$j$};
\filldraw[black] (-0.25,1) circle (2pt) node[anchor=east]{$i$};
\draw[black] (-0.4,0.25) circle (30pt);

\filldraw[black] (4,0) circle (2pt) node[anchor=east]{$l$};
\filldraw[black] (3.5,0.6) circle (2pt) node[anchor=east]{$k$};
\draw[black] (3.5,0.25) circle (30pt);

\draw[gray, thick] (-0.5,-0.5) -- (-0.25,1);
\draw[gray, thick] (4,0) -- node[right]{$c$} (3.5,0.6);
\draw[gray, dotted] (-0.5,-0.5) -- (4,0);
\draw[gray, dotted] (-0.5,-0.5) -- node[above]{$d$} (3.5,0.6);
\draw[gray, dotted] (-0.25,1) -- node[above]{$b$} (3.5,0.6);
\draw[gray, dotted] (-0.25,1) -- (4,0);
\end{tikzpicture}
\caption{Impossible constellation with extension in degree 4}
\label{fig:deg4ex}
\end{figure}

Let $b=(i\ k)$, $c=(k\ l)$ and $d=(j\ k)$. Then $a, b, c, d$ are pairwise not commuting elements except $a\tri c=c$ but $a\neq c$ since $a\notin X_{n-1}$. The elements fulfill the following relations:
\begin{align*}
    &c\tri d=(j\ l)\neq b,\\
    &b\tri(c\tri d)=(i\ k)\tri (j\ l)=(j\ l)=c\tri d,\\
    &a\tri (b\tri d)=a\tri (i\ j)=(i\ j)=b\tri d\ \text{ and }\\
    &b\tri d\neq a \text{ since } a\notin X_{n-1}, (i\ j)\in X_{n-1}.  
\end{align*}
Moreover, $a\tri(b\tri(c\tri d))\neq d$ since otherwise we would have \[a\tri(b\tri(c\tri d))=a\tri (j\ l)=(j\ k).\] Then $(j\ l)\tri (j\ k)=a$ since $X$ is braided. But $a\notin X_{n-1}$. Hence by Theorem~\ref{thm:degree4}, there is a left coideal subalgebra of $\B(X,q)$ in the category of $\ndN_0$-graded $\fK G_X$-comodules that is not generated in degree one. This is a contradiction to the assumption.      
Claim 2: $p=2$.

Proof of claim 2: 
Assume that $p\geq 3$. Let $i\in M_1$, $j\in M_2$ and $k\in M_3$. Then $(i\ j)$, $(j\ k)$, $(i\ k)\in T$ and $(i\ j)\tri (j\ k)=(i\ k)$. This a contradiction to Lemma~\ref{le:T3property2}.   
\begin{figure}[H]
\centering
\begin{tikzpicture}

\filldraw[black] (0,0) circle (2pt) node[anchor=west]{$k$};
\filldraw[black] (-0.5,-0.5) circle (2pt);
\filldraw[black] (-0.25,1) circle (2pt); 
\draw[black] (-0.4,0.25) circle (40pt);

\filldraw[black] (3.5,0.6) circle (2pt) node[anchor=east]{$j$};
\draw[black] (3.5,0.25) circle (30pt);

\filldraw[black] (7,0.3) circle (2pt) node[anchor=east]{$i$};
\draw[black] (7,0.3) circle (20pt);

\draw[gray, thick] (-0.5,-0.5) -- (-0.25,1);
\draw[gray, thick] (0,0) -- (-0.5,-0.5);
\draw[gray, thick] (0,0) -- (-0.25,1);
\end{tikzpicture}
\caption{Impossible constellation}
\label{fig:p3}
\end{figure}
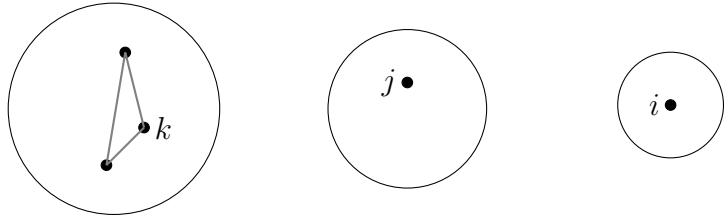

From claims 1 and 2 it follows that $p=2$ and one of $M_1$, $M_2$ has cardinality $1$.
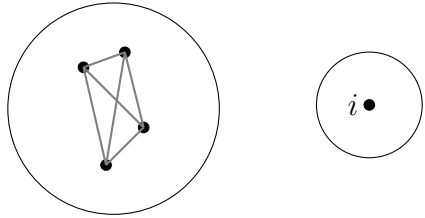
\begin{figure}[H]
\centering
\begin{tikzpicture}

\filldraw[black] (0,0) circle (2pt); 
\filldraw[black] (-0.5,-0.5) circle (2pt);
\filldraw[black] (-0.25,1) circle (2pt); 
\filldraw[black] (-0.8,0.8) circle (2pt);
\draw[black] (-0.4,0.25) circle (40pt);

\filldraw[black] (3,0.3) circle (2pt) node[anchor=east]{$i$};
\draw[black] (3,0.3) circle (20pt);

\draw[gray, thick] (-0.5,-0.5) -- (-0.25,1);
\draw[gray, thick] (0,0) -- (-0.5,-0.5);
\draw[gray, thick] (0,0) -- (-0.25,1);
\draw[gray, thick] (0,0) -- (-0.8,0.8);
\draw[gray, thick] (-0.25,1) -- (-0.8,0.8);
\draw[gray, thick] (-0.5,-0.5) -- (-0.8,0.8);
\end{tikzpicture}
\caption{Possible constellation}
\end{figure}

Hence there is an element $i\in\{1,...,n\}$ such that for all $k, l\in\{1,...,n\}$, $k<l$ the element $a$ commutes with $(k\ l)$ if and only if $k=i$ or $l=i$. Without loss of generality we may assume that $i=1$. It follows that there is a rack isomorphism $X\rightarrow \mathbb{T}_{n+1}$ such that
\begin{align*}
    a\mapsto (1\ n+1),\ (1\ j)\mapsto (1\ j)\ \text{for all } j\in\{2,...,n\}.
\end{align*}
\end{proof}

Now all preparations are made for the proof of the following main Theorem. 

\begin{thm}\label{thm:mainthm}
Let $(X,\tri)$ be a finite indecomposable rack of a conjugacy class of a group, $q: X\times X\rightarrow\fK^{\times}$ a two-cocycle and $(\fK X,c_q)$ the corresponding braided vector space. Let $V=\fK X$ be its Yetter-Drinfeld realization over $\fK G_X$. Then all left coideal subalgebras of $\B(V)$ in the category of $\ndN_0$-graded $\fK G_X$-comodules are generated in degree one if and only if $V$ appears in the following list (besides isomorphy). 
\begin{enumerate}   
    \item $V=\fK\mathbb{T}_3=M(g_{(1\ 2)},\rho_t)$ with $t=-1$ or $t^2-t+1=0$, where $\rho_t$ is the linear character on the centralizer $G_{\mathbb{T}_3}{}^{g_{(1\ 2)}}=\langle g_{(1\ 2)}\rangle$ with $\rho_t(g_{(1\ 2)})=t$. The corresponding $G_{\mathbb{T}_3}$-action on $V$ is explicitly given in Table~\ref{tab:twococycleS3}. 
    \item $V=\fK\mathbb{T}_4=M(g_{(1\ 2)},\rho_{t,\lambda})$ with $t=-1$ and $\lambda\in\{1,-1\}$, where $\rho_{t,\lambda}$ is the linear character on the centralizer $G_{\mathbb{T}_4}{}^{g_{(1\ 2)}}=\langle g_{(1\ 2)},g_{(3\ 4)}\rangle$ with $\rho_{t,\lambda}(g_{(1\ 2)})=t$ and $\rho_{t,\lambda}(g_{(3\ 4)})=\lambda t$. The corresponding two-cocycles are the constant two-cocycle $-1$ and the two-cocycle $q=\chi$ given in Equation~(\ref{eq:SntwococycleChi}), respectively.
    \item $V=\fK\mathbb{T}_5=M(g_{(1\ 2)},\rho_{t,\lambda})$ with $t=-1$ and $\lambda\in\{1,-1\}$, where $\rho_{t,\lambda}$ is the linear character on the centralizer $G_{\mathbb{T}_5}{}^{g_{(1\ 2)}}=\langle g_{(1\ 2)},g_{(3\ 4)},g_{(3\ 5)},g_{(4\ 5)}\rangle$ with $\rho_{t,\lambda}(g_{(1\ 2)})=t$ and $\rho_{t,\lambda}(g_{(i\ j)})=\lambda t$ for all $2<i<j\leq 5$. The corresponding two-cocycles are the constant two-cocycle $-1$ and the two-cocycle $q=\chi$ given in Equation~(\ref{eq:SntwococycleChi}), respectively.   
    \item $V=\fK\mathcal{T}=M(g_a,\rho_{t,\lambda})$ with $t\in\fK$ such that $t^2+t+1=0$ and $\lambda=-t^{-1}$, where $\rho_{t,\lambda}$ is the linear character on the centralizer $G_{\mathcal{T}}{}^{g_{a}}=\langle g_a,g_bg_d\rangle$ with $\rho_{t,\lambda}(g_a)=t$ and $\rho_{t,\lambda}(g_bg_d)=\lambda$. The  
    $G_{\mathcal{T}}$-action on $V$ is explicitly given in Table~\ref{tab:twococycleT}.
    \item $V=\fK\mathcal{B}=M(g_a,\rho_{t,\lambda})$ with $t=-1$ and $\lambda=1$, where $\rho_{t,\lambda}$ is the linear character on the centralizer $G_{\mathcal{B}}{}^{g_{a}}=\langle g_a,g_f\rangle$ with $\rho_{t,\lambda}(g_a)=t$ and $\rho_{t,\lambda}(g_f)=\lambda t$. The 
    $G_{\mathcal{B}}$-action on $V$ is explicitly given in Table~\ref{tab:twococycleB}.
    \item $V=\fK X$, where $X$ is the rack with one element and $q$ is an arbitrary two-cocycle. 
\end{enumerate}
\end{thm}

\begin{proof}
First assume that all left coideal subalgebras of $\B(V)$ in the category of $\ndN_0$-graded $\fK G_X$-comodules are generated in degree one. The rack $X$ is braided by Proposition~\ref{prop:braid}. Assume that $X$ is not the rack with one element. Then there are two elements $x,y\in X$ with $x\tri y\notin\{x,y\}$ since $X$ is indecomposable. Hence $X$ includes $\mathbb{T}_3$, $\mathcal{T}$ or $\mathcal{B}$ as subrack by Lemma~\ref{le:noextensiont_3tC}. It follows that $X$ is $\mathbb{T}_3$, $\mathcal{T}$ or $\mathcal{B}$ if $X$ is generated by two elements. If $X=\mathbb{T}_3$, it follows from Theorem~\ref{thm:3elrack} that the only possible Yetter-Drinfeld structures on $\fK\mathbb{T}_3$ such that all left coideal subalgebras of $\B(\fK\mathbb{T}_3)$ in the category of $\ndN_0$-graded $\fK G_{\mathbb{T}_3}$-subcomodules are generated in degree one are those in (1). For $X=\mathcal{T}$ it follows from Theorem~\ref{thm:Track} that the only choices for suitable Yetter-Drinfeld structures are those in (4). For $X=\mathcal{B}$ it follows from Theorem~\ref{thm:Brack} that the only choice for a suitable Yetter-Drinfeld structure is that in (5). Assume that $X$ is not generated by less than $3$ elements. Proposition~\ref{prop:rackcontainingtetrahedron} and Proposition~\ref{prop:rackcontainingcube} imply that $X$ does not contain the rack $\mathcal{T}$ or the rack $\mathcal{B}$ as subrack. If $X$ is generated by $n\geq 3$ elements (and not by less than $n$ elements) and if $X$ does not contain $\mathcal{T}$ or $\mathcal{B}$ as a subrack, then for all $x,y\in X$ with $x\tri y\neq y$ the subrack generated by $x$ and $y$ is isomorphic to $\mathbb{T}_3$ by Lemma~\ref{le:noextensiont_3tC} and hence $X$ is the rack of transpositions of $\mathbb{S}_{n+1}$ by Proposition~\ref{prop:transpositioncomp}. Then if $n\in\{4,5\}$, the claim follows from Theorem~\ref{thm:generalexttranspos}. If $n\geq 6$, the claim follows from Theorem~\ref{thm:T6rack} and from Remark~\ref{rema:subracklcsas}.

The other implication follows in case (1) from Theorem~\ref{thm:3elrack} (1), in cases (2) and (3) from Theorem~\ref{thm:Tnrack}, in case (4) from Theorem~\ref{thm:Track} (4) and in case (5) from Theorem~\ref{thm:Brack} (4). In case (6) it is obvious.   
\end{proof}

\begin{rema}
The corresponding Nichols algebras $\B(V)$ to the braided vector spaces $V$ in (2)-(5) and to that with $t=-1$ in (1) of Theorem~\ref{thm:mainthm} were already known to be finite-dimensional (see Remark~\ref{rema:12dim} for (1), Remark~\ref{re:twisteq} for (2) and (3), Proposition~\ref{prop:5184} for (4) and Proposition~\ref{prop:576NA} for (5)). The Nichols algebras of the braided vector spaces in Theorem~\ref{thm:mainthm} (1) with $t^2-t+1=0$ is infinite-dimensional by \cite{heckenberger2023simple}. The $72$-dimensional Nichols algebra in Theorem~\ref{thm:Track}~(2) and Remark~\ref{rema:72dim} is an example for a finite-dimensional Nichols algebra that contains a left coideal subalgebra which is not generated in degree one.
\end{rema}

\begin{cor}\label{cor:braiddecomposition}
    Let $G$ be a group and $V$ a Yetter-Drinfeld module over $\fK G$ such that $\dim V_g\leq 1$ for all $g\in G$. Then all left coideal subalgebras of $\B(V)$ in the category of $\ndN_0$-graded $\fK G$-comodules are generated in degree one if and only if there are braided subspaces\footnote{A subspace $W\subseteq V$ is a braided subspace if $c_V(W\ot W)=W\ot W$.} $V_1,\dots, V_n\subseteq V$ such that 
    \begin{enumerate}
        \item $V_i$ is a $\fK G$-subcomodule for all $i\in\{1,\dots,n\}$, 
        \item $V=\bigoplus_{i=1}^n V_i$,
        \item $c_V{}^2\vert V_i\ot V_j=\id_{V_i\ot V_j}$ for all $i,j\in\{1,\dots,n\}$ with $i\ne j$, and
        \item $V_i$ is isomorphic to one of the braided vector spaces in (1)-(6) of Theorem~\ref{thm:mainthm} as braided vector space for all $i\in\{1,\dots,n\}$. 
    \end{enumerate}
\end{cor}

\begin{proof}
First assume that all left coideal subalgebras of $\B(V)$ in the category of $\ndN_0$-graded $\fK G$-comodules are generated in degree one. Let $V=\bigoplus_{k=1}^m W_k$ be the decomposition of $V$ into its irreducible $\fK G$-Yetter-Drinfeld submodules. Then $c_V{}^2\vert W_k\ot W_l=\id_{W_k\ot W_l}$ for all $k,l\in\{1,\dots,m\}$ with $k\ne l$ by Proposition~\ref{prop:reducibleext} and for all $1\leq k\leq m$ all left coideal subalgebras of $\B(W_k)$ in the category of $\ndN_0$-graded $\fK G$-comodules are generated in degree one by Proposition~\ref{prop:lcsadecompextension}. For all $1\leq k\leq m$ let $X_k=\supp W_k\subseteq G$ be the underlying rack\footnote{Note that the racks $X_k$ may be decomposable since it is not provided that $\supp W_k$ generates $G$.} and $\widetilde{W_k}$ the $\fK G_{X_k}$-Yetter-Drinfeld realization of the underlying braided vector space $\fK X_k=W_k$ (see Lemma~\ref{le:YDrackcocylce} and Lemma~\ref{le:YDrealization1} for the construction). Then for all $1\leq k\leq m$ all left coideal subalgebras of $\B(\widetilde{W_k})$ in the category of $\ndN_0$-graded $\fK G_{X_k}$-comodules are generated in degree one by Proposition~\ref{prop:chooseenvelopingrealization} (2). If $\widetilde{W_k}$ is irreducible as Yetter-Drinfeld module over $\fK G_{X_k}$, then $X_k$ is indecomposable by Lemma~\ref{le:irredindecomp}. By Theorem~\ref{thm:mainthm}, it follows that $\widetilde{W_k}$ is isomorphic to one of the braided vector spaces mentioned in (1)-(6) of Theorem~\ref{thm:mainthm} as braided vector space. If $\widetilde{W_k}$ is reducible as Yetter-Drinfeld module over $\fK G_{X_k}$, then we decompose it into a direct sum of its irreducible $\fK G_{X_k}$-Yetter-Drinfeld submodules and repeat the steps of the proof until all summands are irreducible. Note that $\fK G_{X_k}$-Yetter-Drinfeld submodules of $\widetilde{W_k}$ are also $\fK G$-subcomodules of $V$ by construction. 

For the other implication assume that there are braided subspaces $V_1,\dots, V_n\subseteq V$ that fulfill the conditions (1)-(4). Let $X=\supp V$ and let $V_{G_X}$ be the $\fK G_X$-Yetter-Drinfeld realization of the braided vector space $V$. By Proposition~\ref{prop:chooseenvelopingrealization} (2) it suffices to prove that all left coideal subalgebras of $\B(V_{G_X})$ in the category of $\ndN_0$-graded $\fK G_X$-comodules are generated in degree one. The construction of $V_{G_X}$ is described in Lemma~\ref{le:YDrackcocylce} and Lemma~\ref{le:YDrealization1}. In particular, the $\fK G$-subcomodules $V_i$ are also $\fK G_X$-subcomodules and the braidings $c_V$ and $c_{V_{G_X}}$ coincide. The assumptions $\dim V_g\le 1$ for all $g\in G$ and $c_V{}^2\vert V_i\ot V_j=\id_{V_i\ot V_j}$ for all $i,j\in\{1,\dots,n\}$ with $i\ne j$ imply that $V_h=V_{ghg^{-1}}$ for all $g,h\in X$ with $g,h$ belonging to supports of different direct summands $V_i$ of $V$. Since $X$ generates $G_X$ and since $\dim V_g\leq 1$ for all $g\in G$ by assumption, it follows that the subcomodules $V_i$ are $\fK G_X$-Yetter-Drinfeld submodules of $V_{G_X}$. By Theorem~\ref{thm:mainthm}, for all $1\leq i\leq n$ all left coideal subalgebras of $\B(V_i)$ in the category of $\ndN_0$-graded $\fK G_{\supp V_i}$-comodules are generated in degree one. By Proposition~\ref{prop:chooseenvelopingrealization} (2) it follows that all left coideal subalgebras of $\B(V_i)$ in the category of $\ndN_0$-graded $\fK G_X$-comodules are generated in degree one. Then the claim follows from Proposition~\ref{prop:lcsadecompextension} since $c_V{}^2\vert V_i\ot V_j=\id_{V_i\ot V_j}$ for all $1\leq i,j\leq n$ with $i\ne j$ and $\dim V_g\leq 1$ for all $g\in G$ by assumption.  
\end{proof}

\providecommand{\bysame}{\leavevmode\hbox to3em{\hrulefill}\thinspace}
\providecommand{\MR}{\relax\ifhmode\unskip\space\fi MR }
\providecommand{\MRhref}[2]{%
  \href{http://www.ams.org/mathscinet-getitem?mr=#1}{#2}
}
\providecommand{\href}[2]{#2}

\end{document}